\documentclass[a4paper,12pt]{article}

\usepackage[utf8]{inputenc}
\usepackage[english]{babel}
\usepackage{enumitem}
\usepackage{float}
\usepackage{amsmath}
\usepackage{amssymb}
\usepackage{amsthm}
\usepackage{amsfonts}
\usepackage{multirow}
\usepackage{graphicx}
\usepackage{booktabs}
\usepackage{epsfig}
\usepackage{epstopdf}
\usepackage{cite}
\usepackage{multirow} 
\usepackage{algorithm,algpseudocode}
\usepackage{subcaption}
\usepackage{colortbl}
\usepackage{color,xcolor}
\usepackage{multirow}
\usepackage[linkcolor=blue,colorlinks=true]{hyperref}

\usepackage{caption}
\definecolor{light-gray}{gray}{0.9}
\definecolor{light-gray0}{gray}{0.7}

\newtheorem{dfn}{Definition}[section]
\newtheorem{lem}[dfn]{Lemma}
\newtheorem{thm}[dfn]{Theorem}

\theoremstyle{definition}

\newtheorem{asm}[dfn]{Assumption}
\newtheorem{rem}[dfn]{Remark}

\allowdisplaybreaks

\title{New Douglas-Rashford Splitting Algorithms for Generalized DC Programming with Applications in Machine Learning}

\date{\today}

\author{Yonghong Yao \footnote{School of Mathematical Sciences, Tiangong University, Tianjin 300387, China; and Center for Advanced Information Technology, Kyung Hee University, Seoul 02447, South Korea; e-mail: yyhtgu@hotmail.com}
\hspace*{0.8mm}, Lateef O. Jolaoso\footnote{School of Mathematical Sciences, University of Southampton, SO17 1BJ, United Kingdom;  e-mail: l.o.jolaoso@soton.ac.uk.}
\hspace*{0.8mm}, Yekini Shehu\footnote{(Corresponding Author) School of Mathematical Sciences, Zhejiang Normal University, Jinhua 321004, People’s Republic of China; e-mail: yekini.shehu@zjnu.edu.cn}
\hspace*{0.8mm}, Jen-Chih Yao\footnote{Center for General Education, China Medical University, Taichung 40402, Taiwan, Academy of Romanian Scientists, Bucharest, Romania; e-mail: yaojc@mail.cmu.edu.tw}}

\begin{document}
\maketitle

\begin{abstract}
\noindent
In this work, we propose some new Douglas-Rashford splitting algorithms for solving a class of generalized DC (difference of convex functions) in real Hilbert spaces. The proposed methods leverage the proximal properties of the nonsmooth component and a fasten control parameter which improves the convergence rate of the algorithms. We prove the convergence of these methods to the critical points of nonconvex optimization under reasonable conditions. We evaluate the performance and effectiveness of our methods through experimentation with three practical examples in machine learning. Our findings demonstrated that our methods offer efficiency in problem-solving and outperform state-of-the-art techniques like the DCA (DC Algorithm) and ADMM.\\



\end{abstract}

\section{Introduction}\label{Sec:Intro}
 \noindent
Let us consider the following class of generalized DC programming in a real Hilbert space $H$:
\begin{equation}\label{dc1}
  \min_{x \in H} \{p(x)=f (x)+g(x)-h(x)\},
\end{equation}
where $f, g: H\rightarrow (-\infty,+\infty]$ are proper, convex, and lower semicontinuous (not necessarily smooth) functions, and $h: H\rightarrow (-\infty,+\infty]$ is a convex and smooth function. The DC programming \eqref{dc1} was first presented by Tao et al. \cite{Tao}, which has received attention due to its applications in image processing \cite{Tu}, compressed sensing \cite{Lou2}, statistics and machine learning \cite{Gaudioso,Le Thi,Li,Nouiehed}, dimensionality reduction \cite{Ding} and multiple-input-multiple-output (MIMO) \cite{Wu}. For example, in the application of DC programming \eqref{dc1} in machine learning, the function $f$ stands for a loss function that denotes the data fidelity and $g-h$ represents a regularizer that induces some expected structures in the solution \cite{Liu,Lou1}.  \\

\noindent
The DCA (DC algorithm) is one of the most prominent methods for solving DC programming \eqref{dc1}. The DCA linearizes the concave part of the DC programming \eqref{dc1} at the current iteration and obtains the next iteration via a convex subproblem.  Further studies on the DCA have been given in \cite{Gotoh,Lu,Wen}.  In \cite[Algorithm 1]{Chuang}, Chuang et al. introduced the following unified Douglas-Rachdford splitting algorithm to solve DC programming \eqref{dc1}:

\begin{eqnarray}\label{dc2}
\left\{\begin{array}[c]{ll}
	&y_n = \underset{v \in H}{\rm argmin} \Big\{f (v)+\frac{1}{2\beta}\|v-x_n \|^2  \Big\},\\
    &z_n = \underset{v \in H}{\rm argmin}\Big\{g(v)+\frac{1}{2\beta}\|v-(2y_n-x_n+\beta \nabla h(y_n)) \|^2  \Big\},\\
    &x_{n+1}=x_n+\kappa_n(z_n-y_n),
     \end{array}
      \right.
      \end{eqnarray}
with $\beta>0$ and $\kappa_n \in (0,2)$, and obtained, under certain conditions, that $\underset{n\rightarrow \infty} \lim \|y_n-x\|=0$, $\underset{n\rightarrow \infty} \lim x_n=y$, where $x$ is a stationary point of
DC programming \eqref{dc1} and ${\rm prox}_{\beta f} (y)=x$. In \cite{Bian}, Bian and Zhang extended the three-operator splitting algorithm of Davis-Yin \cite{Davis} from solving the optimization problem of the sum of three convex functions to solving nonconvex optimization problems. This three-operator in \cite{Davis} could be seen as a slight extension of the Douglas-Rachford splitting algorithm \cite{Douglas} and the generalized forward-backward splitting algorithm in \cite{Bauschkebook,Raguet}.\\

\noindent
Motivated by the forward-backward splitting algorithm with deviations proposed in \cite{Sadeghi}, Hu and Dong \cite{Hu} proposed the following weakly convergent three-operator splitting algorithm with deviations for solving DC programming \eqref{dc1}:

\begin{eqnarray}\label{dc3}
\left\{\begin{array}[c]{ll}
    &s_n= x_n+u_n-\beta Lv_n,\\
	&y_n=\underset{v \in H}{\rm argmin}\Big\{f (v)+\frac{1}{2\beta}\|v-s_n \|^2  \Big\},\\
     &t_n=y_n-\frac{\beta L}{2}v_n,\\
     &q_n= y_n+v_n,\\
    &z_n = \underset{v \in H}{\rm argmin}\Big\{g(v)+\frac{1}{2\beta}\|v-(2t_n-s_n+\beta \nabla h(q_n)) \|^2  \Big\},\\
    &x_{n+1}=x_n+\kappa_n(z_n-y_n),
     \end{array}
      \right.
      \end{eqnarray}
where the vectors $u_{n+1}$ and $v_{n+1}$ are chosen such that
\begin{equation}\label{dc4}
\frac{\kappa_{n+1}}{2-\kappa_{n+1}}\|u_{n+1}\|^2 +\kappa_{n+1}\beta L \|v_{n+1}\|^2\leq \zeta_n l_n^2
    \end{equation}
with
\begin{equation}\label{dc5}
l_n^2=\kappa_n(2-\kappa_n)\|z_n-y_n+\frac{1}{2-\kappa_n}u_n \|^2.
    \end{equation}
Consequently, both one-step and two-step inertial three-operator splitting algorithms are deduced in \cite[Algorithm 2, Algorithm 3]{Hu} from the deviation vector $u_n$. Numerical comparisons of  \cite[Algorithm 2, Algorithm 3]{Hu} with unified Douglas-Rachdford splitting algorithm \eqref{dc2} using DC regularized sparse recovery problems showed that \cite[Algorithm 2, Algorithm 3]{Hu} outperformed unified Douglas-Rachdford splitting algorithm \eqref{dc2} in terms of cpu time and number of iterations.\\

\noindent
\textbf{Contribution.}
\noindent Our aim in this paper is to continue the approach of  Chuang et al.\cite{Chuang} and Hu and Dong \cite{Hu} by designing a new splitting algorithm to solve DC programming \eqref{dc1} with the following contributions:

\begin{itemize}
  \item
we introduce new fast Douglas-Rachford splitting algorithms with the aim of solving DC programming \eqref{dc1} in Hilbert spaces, which are also extensions of the unified Douglas-Rachford splitting algorithm proposed in \eqref{dc2};
\item in our proposed algorithms, we relax the strict norm conditions \eqref{dc4} and \eqref{dc5} used in \cite{Hu} in proving the convergence of the methods to critical point of the nonconvex optimization problem;
\item we provide some numerical experiments and implement the proposed methods for solving real-life problems in machine learning. The results of the experiments indicate the accuracy and efficiency of the proposed methods over popular methods such as the DC algorithm and ADMM in the literature.
\end{itemize}

\noindent
\textbf{Organization.}
We structure the rest of the paper as follows: in Section \ref{Sec:Prelims}, we put in place some basic concepts and lemmas while in Section \ref{Sec:Method}, we introduce our proposed algorithms and weak convergence analysis. Section \ref{Sec:Numerics} deals with numerical experiments where we implement the proposed methods in three DC models in machine learning. Some final remarks and future considerations are given in Section \ref{Sec:Final}.

\section{Preliminaries}\label{Sec:Prelims}
\noindent
In the convergence analysis of this paper, we shall denote the weak convergence by the symbol "$\rightharpoonup$" and denote the strong convergence by the symbol "$\rightarrow$".\\

\noindent The following definitions and basic concepts can be found in \cite{Bauschkebook}.

\begin{dfn}
A mapping $T: H \to H$ is called
\begin{itemize}
	\item[(i)] nonexpansive if $\|Tx - Ty \| \leq \|x-y \|,$ for all $x,y \in H;$
	\item[(ii)] $L$-Lipschitz continuous if there exists $L>0$ such that $\|Tx - Ty\| \leq L\|x- y \|$ for all $x,y \in H.$
\end{itemize}
\end{dfn}
\noindent We define the domain of a function $f:H\rightarrow (-\infty,+\infty]$ to be ${\rm dom} f:=\{x\in H: f(x)<+\infty\}$ and say that $f$ is proper if ${\rm dom} f \neq \emptyset$.

\begin{dfn}
Suppose $f:H\rightarrow (-\infty,+\infty]$ is proper. The subdifferential of $f$ at $x$ is defined by
$$
\partial f (x):=\{u \in H:f(z) \geq f(x)+\langle z-x,u \rangle, \forall z \in H \}.
$$
 \noindent We say that $f$ is subdifferentiable at $x$ if $\partial f (x) \neq \emptyset$.  In this case, the elements of $\partial f (x)$ are termed the subgradients of $f$ at $x$.

\end{dfn}

\begin{dfn}
An operator $A:H \to 2^H$ is said to be monotone if for any $x,y \in H,$
\begin{equation}
\langle x - y, u - v \rangle \geq 0, \forall u \in Ax, v \in Ay.
\end{equation}
The Graph of $A$ is defined by $$Gr(A):= \{(x,u)\in H \times H : u \in Ax\}.$$ If $Gr(A)$ is not properly contained in the graph of any other monotone mapping, then $A$ is maximal monotone operator.  It is well-known that for each $x \in H$, and
$\lambda>0$, there is a unique $z \in H$ such that $x \in (I+\lambda A)z$. Furthermore, $A$ is $\rho$-strongly monotone (with $\rho>0$) if
\begin{equation*}
\langle x - y, u - v \rangle \geq \rho\|u-v \|^2, \forall u \in Ax, v \in Ay.
\end{equation*}
\end{dfn}

\begin{dfn}
\begin{enumerate}
  \item [(i)]
  A function $f$ is called convex if ${\rm dom} f$ is a convex set and if $\forall x,y \in {\rm dom} f, \delta \in [0,1]$, we have
$$
f(\delta x+(1-\delta)y) \leq \delta f(x)+(1-\delta)f(y).
$$
  \item [(ii)] A function $f$ is said to be $\rho$-strongly convex with $\rho>0$ if $f-\frac{\rho}{2}\|.\|^2$ is convex. Thus, $\forall x,y \in {\rm dom} f, \delta \in [0,1]$,
  $$
f(\delta x+(1-\delta)y) \leq \delta f(x)+(1-\delta)f(y)-\frac{\rho}{2}\theta(1-\theta)\|x-y\|^2.
  $$
\end{enumerate}
Moreover, if $f$ is convex, we have
$$
f(x) \geq f(y) +\langle u, x-y\rangle, x,y \in {\rm dom} f,
$$
with $u \in \partial f(y)$ is arbitrary. Also, if $f$ is $\rho$-strongly convex with $\rho>0$, we have
$$
f(x) \geq f(y) +\langle u, x-y\rangle+ \frac{\rho}{2}\|x-y\|^2, x,y \in {\rm dom} f,
$$
with $u \in \partial f(y)$ is arbitrary.
\end{dfn}

\begin{dfn}
Given $\beta>0$ and $f:H\rightarrow (-\infty,+\infty]$ a proper, lower semicontinuous and convex function. The proximal operator of $f$ with $\beta$ is defined by
$$
{\rm prox}_{\beta f}(x):=\underset{u \in H}{\rm argmin} \Big\{ f(u)+\frac{1}{2\beta}\|u-x\|^2 \Big\}
$$
\noindent for each $x \in H$.
\end{dfn}

\noindent
Using the first-order optimality condition of the minimization $\underset{u\in H}\min \Big\{ f(u)+\frac{1}{2\beta}\|u-x\|^2 \Big\}$, we have that
${\rm prox}_{\beta f}(x)=(I+\beta \partial f)^{-1}(x)$, where $I$ is the identity on $H$.

\begin{lem}\label{lem22}(\cite[Example 22.4]{Bauschkebook})
Suppose $\rho>0$ and $f: H\rightarrow (-\infty, +\infty]$ is a proper, lower-semicontinuous and convex function. If $f$ is $\rho$-strongly convex, then $\partial f$ is $\rho$-strongly monotone.
  \end{lem}

\begin{lem}\label{lem23} (\cite[Proposition 16.36]{Bauschkebook})
Suppose $f: H\rightarrow (-\infty, +\infty]$ is a proper, lower-semicontinuous and convex function. If $\{v_n\}$  and $\{x_n\}$ are sequences in $H$ with $(x_n,v_n)\in Gr(\partial f) $ for each $n\in \mathbb{N}$, $x_n\rightarrow x$ and $v_n\rightharpoonup v$, then $(x,v) \in Gr(\partial f)$.
  \end{lem}

\begin{lem}\label{lm2}
If $x,y\in H$, we have
\begin{itemize}
\item[(i)]
 $2\langle x,y\rangle=\|x\|^{2}+\|y\|^{2}-\|x-y\|^{2}=\|x+y\|^{2}-\|x\|^{2}-\|y\|^{2}.$
\item[(ii)]
Let $x,y\in H$ and $a \in \mathbb{R}$. Then
\begin{eqnarray*}
\|(1-a)x+ay\|^2= (1-a)\|x\|^2+a\|y\|^2-(1-a)a\|x-y\|^2.
\end{eqnarray*}
\end{itemize}
\end{lem}

\begin{lem}\label{lem:w-converge} (\cite{Opial})
Let $K$ be a nonempty subset of a Hilbert space $H$ and let $\{x_n\}$ be a bounded sequence in $H$.
Assume the following two conditions are satisfied:
\begin{itemize}
\item $\lim_{n\to\infty}\|x_n-x\|$ exists for each $x\in K$,
\item every weak cluster point of $\{x_n\}$ belongs to $K$.
\end{itemize}
Then $\{x_n\}$ converges weakly to a point in $K$.
\end{lem}

\begin{dfn}
  We say that $x$ is a stationary point of $p$ in DC programming \eqref{dc1} if
  $$
  0\in \partial f(x)+\partial g(x)-\nabla h(x).
  $$
\end{dfn}
\noindent
We shall denote by $\Omega$, the set of all the stationary points of $p$ in DC programming \eqref{dc1} and consequently assume that $\Omega \neq \emptyset$ throughout this paper.\\

\noindent In DC programming \eqref{dc1}, assume for the rest of this paper that $\beta>0$ is given, $f$ and $g$ are $\rho$-strongly convex, and $h$ is a smooth convex function having a Lipschitz continuous gradient with constant $L>0$. The following results have been established in \cite[Proposition 3.1]{Hu} regarding the set of stationary points of DC programming \eqref{dc1}.

\begin{lem} \label{DC1}
If $2\rho>L$  in DC programming \eqref{dc1}, then the stationary point of $p$ is unique.
\end{lem}

\begin{lem}\label{DC2}(\cite[Proposition 3.1]{Chuang})
Suppose $\beta>0$. Then $x \in \Omega$ if and only if there exists $y \in H$ such that

\begin{eqnarray*}
\left\{\begin{array}[c]{ll}
	x&={\rm prox}_{\beta f} (y),\\
	x&= {\rm prox}_{\beta f} (2x-y+\beta \nabla h(x)).
     \end{array}
      \right.
\end{eqnarray*}
\end{lem}

\noindent
In order to obtain our convergence analysis of the proposed Algorithm \ref{alg2}, we give the following conditions.

\begin{asm}\label{ass1} Suppose the iterative parameters in Algorithm \ref{alg2} fulfill these conditions:
 \begin{enumerate}
   \item [(i)] $2\rho >L$;
   \item [(ii)] $0<a\leq \kappa_n \leq b<2$.
 \end{enumerate}

\end{asm}

\noindent Going by Lemma \ref{DC2}, given $\beta>0$, we set
$$
\Gamma_\beta:=\{ y: {\rm prox}_{\beta f} (y)=x, {\rm prox}_{\beta f} (2x-y+\beta \nabla h(x))=x, x \in \Omega\}.
$$

\section{Proposed Algorithm and Its Convergence} \label{Sec:Method}
\noindent In this section, we introduce the proposed algorithms and obtain associated weak convergence results to the critical point of DC programming \eqref{dc1}.

\begin{algorithm}[H]
\caption{First Version of Douglas-Rachford Algorithm}\label{alg2}
\begin{algorithmic}[1]
\State {\bf Initialization:} Choose  $\beta >0, \theta \geq 0$, and $x_0, v_0 \in H$ arbitrarily. Set $n=0.$

\State {\bf Iteration Step:} Given the iterates $x_n, v_n$, compute
        \begin{eqnarray}\label{yi1}
\left\{\begin{array}[c]{ll}
	&u_n= \frac{1}{1+\theta}x_n+\frac{\theta}{1+\theta}v_n,\\
	&y_n = \underset{v \in H}{\rm argmin}\Big\{f (v)+\frac{1}{2\beta}\|v-u_n \|^2  \Big\},\\
    &z_n = \underset{v \in H}{\rm argmin}\Big\{g(v)+\frac{1}{2\beta}\|v-(2y_n-u_n+\beta \nabla h(y_n)) \|^2  \Big\},\\
    &x_{n+1}=u_n+\kappa_n(z_n-y_n),\\
    &v_{n+1}=\frac{1}{1+\theta}x_{n+1}+\frac{\theta}{1+\theta}v_n.
     \end{array}
      \right.
      \end{eqnarray}
	\noindent Set $n\leftarrow n+1,$ and go to {\bf Iteration Step} if the stopping criterion is not met.
\end{algorithmic}
\end{algorithm}

\begin{rem}\label{remark1}
If $\theta=0$ in Algorithm \ref{alg2}, then $u_n=x_n$ and $v_{n+1}=x_{n+1}$. Consequently, we obtain  \cite[Algorithm 1]{Chuang}. Therefore, \cite[Algorithm 1]{Chuang} is recovered from Algorithm \ref{alg2} when $\theta=0$. It suffices to obtain the weak convergence of Algorithm \ref{alg2} when $\theta>0$ in Theorem \ref{thm2aa} below.
\end{rem}

\noindent We can easily obtain this lemma by following same arguments as in \cite[Theorem 3.1]{Chuang}. For the sake of completeness and to make our results reader-friendly, we include the complete proof.

\begin{lem}\label{LEMMA1}
Suppose $\beta, \kappa>0$, for each $u \in H$, let
        \begin{eqnarray}\label{adidi}
\left\{\begin{array}[c]{ll}
	&y = \underset{v \in H}{\rm argmin}\Big\{f (v)+\frac{1}{2\beta}\|v-u\|^2  \Big\},\\
    &z = \underset{v \in H}{\rm argmin}\Big\{g(v)+\frac{1}{2\beta}\|v-(2y-u+\beta \nabla h(y)) \|^2  \Big\},\\
    &x=u+\kappa(z-y).
     \end{array}
      \right.
      \end{eqnarray}
Then for any $\bar{x} \in \Omega$, there exists  $\bar{y}\in \Gamma_\beta$ such that
\begin{eqnarray}\label{ade11}
\|x-\bar{y}\|^2 &\leq& \|u-\bar{y}\|^2-\beta\kappa(2\rho-L)\Big(\|y-\bar{x}\|^2+\|z-\bar{x}\|^2\Big)\nonumber \\
&&-\kappa(2-\kappa)\|z-y\|^2.
  \end{eqnarray}
\end{lem}

\begin{proof}
Using the first-order optimality condition on $y={\rm prox}_{\beta f}(u)$ and $z={\rm prox}_{\beta g}(2y-u+\beta \nabla h(y))$, we obtain

\begin{equation}\label{ade1a}
\frac{1}{\beta}(u-y) \in \partial f(y),
\end{equation}
and
\begin{equation}\label{ade1b}
\nabla h(y)+\frac{1}{\beta}(2y-z-u) \in \partial g(z)
\end{equation}
respectively. Now for any $\bar{x} \in \Omega$, there exists a $\bar{y}\in \Gamma_\beta$ such that
\begin{eqnarray}\label{ade2a}
\frac{1}{\beta}(\bar{y}-\bar{x}) \in \partial f (\bar{x})
\end{eqnarray}
and
\begin{equation}\label{ade2b}
{\rm and}~~\frac{1}{\beta}(\bar{x}-\bar{y}) +\nabla h(\bar{x}) \in \partial g(\bar{x}).
\end{equation}
\noindent
By observing that $f$ is $\rho$-strongly convex, we have from Lemma \ref{lem22} that $\partial f$ is $\rho$-strongly monotone. Consequently,

\begin{equation}\label{ade3}
  \langle \zeta_{y}-\zeta_{\bar{x}},y-\bar{x} \rangle \geq \rho \|y-\bar{x}\|^2, ~~\forall \zeta_{y} \in \partial f(y), \zeta_{\bar{x}} \in \partial f(\bar{x}).
\end{equation}
If we apply \eqref{ade3} in \eqref{ade1a} and \eqref{ade2a}, we obtain

\begin{equation*}
\rho \|y-\bar{x}\|^2 \leq \frac{1}{\beta} \langle y-\bar{x},u-y-\bar{y}+\bar{x} \rangle
  \end{equation*}
which implies that
\begin{equation}\label{ade4}
  (2+2\beta \rho) \|y-\bar{x}\|^2 \leq 2\langle y-\bar{x},u-\bar{y} \rangle.
\end{equation}
Similarly, let us apply the property that $g$ is $\rho$-strongly convex in \eqref{ade1b} and \eqref{ade2b}. Then we obtain
\begin{eqnarray}\label{ade5}
\rho \|z-\bar{x}\|^2 \leq \langle z-\bar{x},\nabla h(y)-\nabla h(\bar{x})+\frac{1}{\beta}(2y-z-u)-\frac{1}{\beta}(\bar{x}-\bar{y})\rangle.
  \end{eqnarray}
If we rearrange \eqref{ade5}, noting the Cauchy-Schwarz inequality and Lipschitz continuity of $\nabla h$, we get
\begin{eqnarray}\label{ade6}
2\beta\rho \|z-\bar{x}\|^2 &\leq& 2\beta\langle z-\bar{x},\nabla h(y)-\nabla h(\bar{x})+\frac{1}{\beta}(2y-z-u)-\frac{1}{\beta}(\bar{x}-\bar{y})\rangle \nonumber\\
&=&2\beta\langle z-\bar{x},\nabla h(y)-\nabla h(\bar{x})\rangle + 2\langle z-\bar{x},2y-z-u-\bar{x}\rangle \nonumber\\
&\leq & 2\beta L\|z-\bar{x}\|\|y-\bar{x}\|+2\langle z-\bar{x},2y-z-u-\bar{x}\rangle \nonumber\\
&\leq & \beta L\|z-\bar{x}\|^2+\beta L\|y-\bar{x}\|^2+\|z-\bar{x}\|^2\nonumber \\
&&+\|2y-z-u+\bar{y}-\bar{x}\|^2-\|2z-2y+u-\bar{y}\|^2.
  \end{eqnarray}
Therefore, we have from \eqref{ade6} that
\begin{eqnarray}\label{ade7}
\|2z-2y+u-\bar{y}\|^2 &\leq& (1+\beta L-2\beta\rho)\|z-\bar{x}\|^2 +\beta L\|y-\bar{x}\|^2\nonumber \\
&&+\|2y-z-u+\bar{y}-\bar{x}\|^2.
  \end{eqnarray}
Now, consider (noting \eqref{ade4})

\begin{eqnarray}\label{ade8}
\|2y-z-u+\bar{y}-\bar{x}\|^2 &=& \|(y-z)+(y-\bar{x})-(u-\bar{y})\|^2\nonumber \\
&=&\|y-z\|^2+\|y-\bar{x}\|^2+\|u-\bar{y}\|^2+2\langle y-\bar{x}, y-z\rangle \nonumber \\
&&-2\langle u-y,y-z \rangle-2\langle u-\bar{y},y-\bar{x} \rangle \nonumber \\
&=& 2\|y-z\|^2+2\|y-\bar{x}\|^2+\|u-\bar{y}\|^2-\|z-\bar{x}\|^2\nonumber \\
&&-2 \langle u-\bar{y},y-z \rangle-2\langle u-\bar{y},y-\bar{x} \rangle \nonumber \\
&\leq& 2\|y-z\|^2+2\|y-\bar{x}\|^2+\|u-\bar{y}\|^2-\|z-\bar{x}\|^2\nonumber \\
&&-2 \langle u-\bar{y},y-z \rangle-(2+2\beta\rho)\|y-\bar{x}\|^2\nonumber \\
&=&2\|y-z\|^2+\|u-\bar{y}\|^2-\|z-\bar{x}\|^2-2\beta\rho \|y-\bar{x}\|^2\nonumber \\
&&-2 \langle u-\bar{y},y-\bar{x} \rangle-2 \langle u-\bar{y},\bar{x}-z \rangle \nonumber \\
&\leq& 2\|y-z\|^2+\|u-\bar{y}\|^2-\|z-\bar{x}\|^2\nonumber \\
&&-(2+4\beta\rho) \|y-\bar{x}\|^2+2\langle u-\bar{y},z-\bar{x} \rangle.
  \end{eqnarray}
We next estimate the term $2\langle u-\bar{y},z-\bar{x} \rangle$ in \eqref{ade8}. Now, if we look at \eqref{ade5} and we rearrange, we obtain
\begin{eqnarray}\label{ade9}
2\langle u-\bar{y},z-\bar{x} \rangle
&\leq& 2\beta \langle z-\bar{x},\nabla h(y)-\nabla h(\bar{x}) \rangle +2\langle z-\bar{x},y-z\rangle \nonumber\\
&&+2\langle z-\bar{x},y-\bar{x}\rangle-2\beta\rho\|z-\bar{x}\|^2 \nonumber \\
&\leq& \beta L\|z-\bar{x}\|^2+\beta L\|y-\bar{x}\|^2+\|y-\bar{x}\|^2-\|z-y\|^2 \nonumber \\
&&-\|z-\bar{x}\|^2+\|z-\bar{x}\|^2+\|y-\bar{x}\|^2-\|z-y\|^2\nonumber \\
&&-2\beta\rho\|z-\bar{x}\|^2 \nonumber \\
&=&(\beta L-2\beta\rho)\|z-\bar{x}\|^2+(2+\beta L)\|y-\bar{x}\|^2\nonumber \\
&&-2\|z-y\|^2.
\end{eqnarray}
 Plugging \eqref{ade9} into \eqref{ade8}, we get
\begin{eqnarray}\label{ade10}
\|2y-z-u+\bar{y}-\bar{x}\|^2 &\leq& 2\|y-z\|^2+\|u-\bar{y}\|^2-\|z-\bar{x}\|^2\nonumber \\
&&-(2+4\beta\rho) \|y-\bar{x}\|^2+(\beta L-2\beta\rho)\|z-\bar{x}\|^2\nonumber \\
&&+(2+\beta L) \|y-\bar{x}\|^2\nonumber \\
&&-2\|z-y\|^2\nonumber \\
&=& \|u-\bar{y}\|^2+(\beta L-2\beta\rho-1)\|z-\bar{x}\|^2\nonumber \\
&&+(\beta L-4\beta\rho) \|y-\bar{x}\|^2.
   \end{eqnarray}
Now, let us plug \eqref{ade10} into \eqref{ade7} to obtain

\begin{eqnarray*}
\|2z-2y+u-\bar{y}\|^2 &\leq& (1+\beta L-2\beta\rho)\|z-\bar{x}\|^2 +\beta L\|y-\bar{x}\|^2\nonumber \\
&&+\|u-\bar{y}\|^2+(\beta L-2\beta\rho-1)\|z-\bar{x}\|^2\nonumber \\
&&+(\beta L-4\beta\rho) \|y-\bar{x}\|^2 \nonumber \\
&=&\|u-\bar{y}\|^2-2\beta(2\rho-L)\Big(\|y-\bar{x}\|^2+\|z-\bar{x}\|^2\Big).
  \end{eqnarray*}
We then obtain from $x=u+\kappa(z-y)$ that (noting immediate last inequality)
\begin{eqnarray*}
\|x-\bar{y}\|^2 &=& \|u+\kappa(z-y)-\bar{y}\|^2\nonumber \\
&=&\|\Big(1-\frac{\kappa}{2}\Big)(u-\bar{y})+\frac{\kappa}{2}(2z-2y+u-\bar{y})\|^2 \nonumber \\
&=&\Big(1-\frac{\kappa}{2}\Big)\|u-\bar{y}\|^2+\frac{\kappa}{2}\|2z-2y+u-\bar{y}\|^2 \nonumber \\
&&-\frac{\kappa}{2}\Big(1-\frac{\kappa}{2}\Big)\|(2z-2y+u-\bar{y})-(u-\bar{y})\|^2\nonumber \\
&=&\Big(1-\frac{\kappa}{2}\Big)\|u-\bar{y}\|^2+\frac{\kappa}{2}\|2z-2y+u-\bar{y}\|^2 \nonumber \\
&&-\kappa(2-\kappa)\|z-y\|^2\nonumber \\
&\leq& \Big(1-\frac{\kappa}{2}\Big)\|u-\bar{y}\|^2+\frac{\kappa}{2}\|u-\bar{y}\|^2\nonumber \\
&&-\beta \kappa(2\rho-L)\Big(\|y-\bar{x}\|^2+\|z-\bar{x}\|^2\Big)-\kappa(2-\kappa)\|z-y\|^2\nonumber \\
&=&\|u-\bar{y}\|^2-\beta \kappa(2\rho-L)\Big(\|y-\bar{x}\|^2+\|z-\bar{x}\|^2\Big)\nonumber \\
&&-\kappa(2-\kappa)\|z-y\|^2.
  \end{eqnarray*}

\end{proof}

\noindent We now give the following weak convergence result of Algorithm \ref{alg2} using Lemma \ref{LEMMA1}.

\begin{thm}\label{thm2aa}
The sequence $\{x_n\}$ generated by Algorithm \ref{alg2} converges weakly to a point in $\Gamma_\beta$ under Assumption \ref{ass1}.
\end{thm}

\begin{proof}
We shall only consider the case $\theta>0$ in our convergence analysis since the case $\theta=0$ has already been considered in
\cite[Algorithm 1]{Chuang} with convergence results given in \cite[Theorem 3.1]{Chuang}. Now by Lemma  \ref{lm2} (ii), we have from Algorithm \ref{alg2} that (putting $\bar{y}=y$ and $\bar{x}=x$  in Lemma \ref{LEMMA1})
\begin{eqnarray}\label{titi}
\|u_n-y\|^2&=& \big\|\frac{1}{1+\theta}(x_n-y)+\frac{\theta}{1+\theta}(v_n-y)\big\|^2\nonumber\\
&=&\frac{1}{1+\theta}\|x_n-y\|^2+\frac{\theta}{1+\theta}\|v_n-y\|^2 -\frac{\theta}{(1+\theta)^2}\|x_n-v_n\|^2\qquad
\end{eqnarray}
and
\begin{eqnarray}\label{titi2}
\|v_{n+1}-y\|^2&=& \frac{\theta}{1+\theta}\|v_n-y\|^2+\frac{1}{1+\theta}\|x_{n+1}-y\|^2 -\frac{\theta}{(1+\theta)^2}\|x_{n+1}-v_n\|^2.\qquad
\end{eqnarray}
If we plug \eqref{titi} into \eqref{ade11} in Lemma \ref{LEMMA1}, we have
\begin{eqnarray}\label{titi3}
\|x_{n+1}-y\|^2&\leq& \|u_n-y\|^2-\beta \kappa_n(2\rho-L)\Big(\|y_n-x\|^2+\|z_n-x\|^2\Big)\nonumber \\
&&-\kappa_n(2-\kappa_n)\|z_n-y_n\|^2\nonumber\\
&=&\frac{1}{1+\theta}\|x_n-y\|^2+\frac{\theta}{1+\theta}\|v_n-y\|^2 \nonumber \\
&&-\frac{\theta}{(1+\theta)^2}\|x_n-v_n\|^2-\beta \kappa_n(2\rho-L)\Big(\|y_n-x\|^2+\|z_n-x\|^2\Big)\nonumber \\
&&-\kappa_n(2-\kappa_n)\|z_n-y_n\|^2.
\end{eqnarray}
We then obtain from \eqref{titi2} and \eqref{titi3} that
\begin{eqnarray}\label{ihua}
\|x_{n+1}-y\|^2&+&\theta\|v_{n+1}-y\|^2\leq\frac{1}{1+\theta}\|x_n-y\|^2+\frac{\theta}{1+\theta}\|v_n-y\|^2\nonumber\\
&&-\frac{\theta}{(1+\theta)^2}\|x_n-v_n\|^2-\beta \kappa_n(2\rho-L)\Big(\|y_n-x\|^2+\|z_n-x\|^2\Big)\nonumber \\
&&-\kappa_n(2-\kappa_n)\|z_n-y_n\|^2+\frac{\theta}{1+\theta}\|x_{n+1}-y\|^2+ \frac{\theta^2}{1+\theta}\|v_n-y\|^2\nonumber\\
&&-\frac{\theta^2}{(1+\theta)^2}\|x_{n+1}-v_n\|^2.
\end{eqnarray}
Therefore,
\begin{eqnarray}\label{app4aa}
&&\frac{1}{1+\theta}\|x_{n+1}-y\|^2+\theta\|v_{n+1}-y\|^2+\frac{\theta^2}{(1+\theta)^2} \|x_{n+1}-v_n\|^2\nonumber\\
&\leq& \frac{1}{1+\theta}\|x_n-y\|^2+\theta\|v_n-y\|^2-\frac{\theta}{(1+\theta)^2}\|x_n-v_n\|^2\nonumber\\
&&-\beta \kappa_n(2\rho-L)\Big(\|y_n-x\|^2+\|z_n-x\|^2\Big)\nonumber \\
&&-\kappa_n(2-\kappa_n)\|z_n-y_n\|^2 +\frac{\theta^2}{(1+\theta)^2}\|x_n-v_{n-1}\|^2\nonumber \\
&&-\frac{\theta^2}{(1+\theta)^2}\|x_n-v_{n-1}\|^2\nonumber\\
&=&\frac{1}{1+\theta}\|x_n-y\|^2+\theta\|v_n-y\|^2+\frac{\theta^2}{(1+\theta)^2}\|x_n-v_{n-1}\|^2\nonumber\\
&&-\frac{\theta^2}{(1+\theta)^2}\|x_n-v_{n-1}\|^2-\beta \kappa_n(2\rho-L)\Big(\|y_n-x\|^2+\|z_n-x\|^2\Big)\nonumber \\
&&-\kappa_n(2-\kappa_n)\|z_n-y_n\|^2-\frac{\theta}{(1+\theta)^2}\|x_n-v_n\|^2.
\end{eqnarray}
Define
\begin{eqnarray}\label{app5aa}
c_n:&=&\frac{1}{1+\theta}\|x_n-y\|^2+\theta\|v_n-y\|^2\nonumber\\
&&+\frac{\theta^2}{(1+\theta)^2}\|x_n-v_{n-1}\|^2.
\end{eqnarray}
Then, \eqref{app4aa} becomes
\begin{eqnarray}\label{app7aa}
c_{n+1}&\leq&c_n-\frac{\theta^2}{(1+\theta)^2}\|x_n-v_{n-1}\|^2-\beta \kappa_n(2\rho-L)\Big(\|y_n-x\|^2+\|z_n-x\|^2\Big)\nonumber \\
&&-\kappa_n(2-\kappa_n)\|z_n-y_n\|^2-\frac{\theta}{(1+\theta)^2}\|x_n-v_n\|^2.
\end{eqnarray}
Thus, $\{c_n\}$ is monotone non-increasing and $\{c_n\}$ is convergent. Moreover, $\{c_n\}$ is bounded. Furthermore,  we obtain from \eqref{app7aa} that
$$
\underset{n\rightarrow \infty}\lim \|x_n-v_{n-1}\|=0, ~~\underset{n\rightarrow \infty}\lim \|x_n-v_n\|=0.
$$
\noindent Furthermore, we have from \eqref{app7aa} that
\begin{equation*}
 \beta \kappa_n(2\rho-L)\Big(\|y_n-x\|^2+\|z_n-x\|^2\Big)\leq c_n-c_{n+1},
\end{equation*}
which implies that
\begin{equation*}
\underset{n\rightarrow \infty}\lim  \beta \kappa_n(2\rho-L)\Big(\|y_n-x\|^2+\|z_n-x\|^2\Big)=0.
\end{equation*}
Since $2\rho>L$ and $0<a\leq \kappa_n \leq b<2$, we get
\begin{equation}\label{app7ab}
\underset{n\rightarrow \infty}\lim  \|y_n-x\|=0,~~{\rm and}~~\underset{n\rightarrow \infty}\lim\|z_n-x\|=0.
\end{equation}

\noindent Let us define
$$
s_n:=\|v_n-x_n\|^2+2\langle v_n-x_n,x_n-y\rangle
$$
\noindent and
$$
t_n:=-\frac{\theta}{(1+\theta)^2}\|x_n-v_{n-1}\|^2.
$$
\noindent Since $\{x_n\}$ is bounded and $\underset{n\rightarrow \infty}\lim \|x_n-v_n\|=0$, we have $
\underset{n\rightarrow \infty}\lim s_n=0$ and $t_n\rightarrow 0$, as $n\rightarrow \infty$.
Noting that
\begin{eqnarray*}
\|v_n-y\|^2&=&\|v_n-x_n\|^2+2\langle v_n-x_n,x_n-y\rangle \\
&&+\|x_n-y\|^2
\end{eqnarray*}
and
\begin{eqnarray*}
s_n&=&\|v_n-x_n\|^2+2\langle v_n-x_n,x_n-y\rangle \\
&=&\|v_n-y\|^2-\|x_n-y\|^2.
\end{eqnarray*}
We then obtain
$$
c_n-\theta s_n+\theta t_n=\frac{\theta^2+\theta+1}{1+\theta}\|x_n-y\|^2.
$$
\noindent By the existence of the limit $\underset{n\rightarrow \infty}\lim c_n$, we get
 $\underset{n\rightarrow \infty}\lim \|x_n-y\|$ exists. We conclude by Lemma \ref{lem:w-converge} that $\{x_n\}$ converges weakly to a point in $\Gamma_\beta$. Consequently, $\{v_n\}$ also  converges weakly to a point in $\Gamma_\beta$. This completes the proof.

\end{proof}

\begin{rem}
One can obtain from \eqref{app7aa} that 
$$
 \frac{\theta}{(1+\theta)^2}\|x_n-v_n\|^2\leq c_n-c_{n+1} $$
  \noindent and thus
$$
\frac{\theta}{(1+\theta)^2}\sum_{j=0}^{n}\|x_j-v_j\|^2\leq \sum_{j=0}^{n}(c_n-c_{n+1})\leq c_0, 
$$  
which implies that (if $x_0=v_0$)
$$
(n+1)\frac{\theta}{(1+\theta)^2}\min_{0\leq j\leq n}\|x_j-v_j\|^2 \leq c_0=\frac{1+\theta(1+\theta)}{1+\theta}\|x_0-y\|^2.
$$  
\noindent Consequently, we have
$$
\min_{0\leq j\leq n}\|x_j-v_j\|\leq \mathcal{O}(\frac{1}{\sqrt{n+1}}).
$$  
 \noindent In the same way, one can obtain
$$
\min_{0\leq j\leq n}\|x_j-v_{j-1}\|\leq \mathcal{O}(\frac{1}{\sqrt{n+1}}),
\min_{0\leq j\leq n}\|z_j-y_j\|\leq \mathcal{O}(\frac{1}{\sqrt{n+1}}).
$$
\end{rem}

\bigskip

\noindent
We design another version of Douglas-Rachford splitting algorithm for DC programming \eqref{dc1} and obtain weak convergence results.

\begin{algorithm}[H]
\caption{Second Version of Douglas-Rachford Algorithm}\label{alg1}
\begin{algorithmic}[1]
\State {\bf Initialization:} Choose $\alpha_n \in [0,1)$, $\beta >0$, and $x_0, v_0 \in H$ arbitrarily. Set $n=0.$

\State {\bf Iteration Step:} Given the iterates $x_n, v_n$, compute
        \begin{eqnarray}\label{ap}
\left\{\begin{array}[c]{ll}
	&u_n=(1-\alpha_n)x_n+\alpha_nv_n,\\
	&y_n = \underset{v \in H}{\rm argmin}\Big\{f (v)+\frac{1}{2\beta}\|v-u_n \|^2  \Big\},\\
    &z_n = \underset{v \in H}{\rm argmin}\Big\{g(v)+\frac{1}{2\beta}\|v-(2y_n-u_n+\beta \nabla h(y_n)) \|^2  \Big\},\\
    &x_{n+1}=u_n+\kappa_n(z_n-y_n),\\
    &v_{n+1}=(1-\alpha_n)v_n+\alpha_nx_n.
     \end{array}
      \right.
      \end{eqnarray}
	\noindent Set $n\leftarrow n+1,$ and go to {\bf Iteration Step} if the stopping criterion is not met.
\end{algorithmic}
\end{algorithm}

\noindent Since our Algorithm \ref{alg1} reduces to \cite[Algorithm 1]{Chuang} when $\alpha_n=0$ for each $n \in \mathbb{N}$ for which the convergence results have been obtained in \cite{Chuang}. It suffices to only consider the case $0<a\leq \alpha_n\leq b<1$ in our convergence result below for Algorithm \ref{alg1}.

\begin{thm}\label{thm1}
The sequence $\{x_n\}$ generated by Algorithm \ref{alg1} converges weakly to a point in $\Gamma_\beta$ under Assumption \ref{ass1} when either $0<a\leq \alpha_n\leq b<1$ or $\alpha_n=0$ for each $n \in \mathbb{N}$.
\end{thm}

\begin{proof}
We can easily obtain from Algorithm \ref{alg1} and Lemma \ref{LEMMA1} that (with $\bar{y}=y$)
\begin{eqnarray}\label{ade11add}
\|x_{n+1}-y\|^2 &\leq& \|u_n-y\|^2-\beta \kappa_n(2\rho-L)\Big(\|y_n-x\|^2+\|z_n-x\|^2\Big)\nonumber \\
&&-\kappa_n(2-\kappa_n)\|z_n-y_n\|^2.
  \end{eqnarray}
By Lemma \ref{lm2}(ii), we have
\begin{eqnarray}\label{app4}
\|u_n-y\|^2&=& \big\|(1-\alpha_n)(x_n-y)+\alpha_n(v_n-y)\big\|^2\nonumber\\
&=&(1-\alpha_n)\|x_n-y\|^2+\alpha_n\|v_n-y\|^2 -\alpha_n(1-\alpha_n)\|x_n-v_n\|^2\qquad
\end{eqnarray}
and
\begin{eqnarray}\label{app5}
\|v_{n+1}-y\|^2&=& \|(1-\alpha_n)(v_n-y)+\alpha_n(x_n-y)\|^2\nonumber\\
&=&(1-\alpha_n)\|v_n-y\|^2+\alpha_n\|x_n-y\|^2 -\alpha_n(1-\alpha_n)\|x_n-v_n\|^2.\qquad
\end{eqnarray}
Plugging \eqref{app4} into \eqref{ade11add}, we obtain
\begin{eqnarray}\label{app6}
\|x_{n+1}-y\|^2&\leq& (1-\alpha_n)\|x_n-y\|^2+\alpha_n\|v_n-y\|^2\nonumber\\
&&-\alpha_n(1-\alpha_n)\|x_n-v_n\|^2-\beta \kappa_n(2\rho-L)\Big(\|y_n-x\|^2+\|z_n-x\|^2\Big)\nonumber \\
&&-\kappa_n(2-\kappa_n)\|z_n-y_n\|^2.
  \end{eqnarray}
Denote
\begin{equation}\label{ineq-bai}
a_n:=\|x_n-y\|^2+\|v_n-y\|^2, \quad \forall n \in \mathbb{N}.
\end{equation}
Then, summing  \eqref{app5} and \eqref{app6} gives
\begin{eqnarray}\label{app8}
a_{n+1}&\leq& a_n-2\alpha_n(1-\alpha_n)\|x_n-v_n\|^2-\beta \kappa_n(2\rho-L)\Big(\|y_n-x\|^2+\|z_n-x\|^2\Big)\nonumber \\
&&-\kappa_n(2-\kappa_n)\|z_n-y_n\|^2,
  \end{eqnarray}
which, by $\alpha_n \in [0,1)$, and $\kappa_n \in (0,2)$, implies that $\{a_n\}$ is a monotone non-increasing and $\{a_n\}$ is convergent. Moreover, $\{a_n\}$ is bounded. We then obtain from the definition of $\{a_n\}$ that both $\{x_n\}$ and $\{v_n\}$ are bounded. Furthermore, $\{u_n\}$, $\{y_n\}$ and $\{z_n\}$ are also bounded.\\

\noindent We also obtain from \eqref{app8} that
\begin{equation*}
 \beta \kappa_n(2\rho-L)\Big(\|y_n-x\|^2+\|z_n-x\|^2\Big)\leq a_n-a_{n+1},
\end{equation*}
which implies that
\begin{equation*}
\underset{n\rightarrow \infty}\lim  \beta \kappa_n(2\rho-L)\Big(\|y_n-x\|^2+\|z_n-x\|^2\Big)=0.
\end{equation*}
Since $2\rho>L$ and $0<a\leq \kappa_n \leq b<2$, we get
\begin{equation}\label{app7ab}
\underset{n\rightarrow \infty}\lim  \|y_n-x\|=0,~~{\rm and}~~\underset{n\rightarrow \infty}\lim\|z_n-x\|=0.
\end{equation}
\noindent By the condition that $\underset{n\rightarrow \infty}\liminf \kappa_n(2-\kappa_n)>0$, we also have from \eqref{app8} that
$$
\underset{n\rightarrow \infty}\lim \|z_n-y_n\|=0.
$$
\noindent
Also from \eqref{app8},
$$
\underset{n\rightarrow \infty}\lim 2\alpha_n(1-\alpha_n)\|x_n-v_n\|^2=0.
$$
\noindent
Hence,
$$
\underset{n\rightarrow \infty}\lim \|x_n-v_n\|=0.
$$
\noindent We have from $v_{n+1}=(1-\alpha_n)v_k+\alpha_nx_n$ that
$$
v_{n+1}-v_n=\alpha_n(x_n-v_n)\rightarrow 0, n\rightarrow \infty.
$$
\noindent
So,
\begin{eqnarray*}
\|x_{n+1}-x_n\|\leq \|x_{n+1}-v_{n+1}\|+\|v_{n+1}-v_n\|+\|x_n-v_n\|\rightarrow 0, n \rightarrow \infty.
\end{eqnarray*}

 \noindent Since $\{x_n\}$, $\{y_n\}$ are $\{z_n\}$ are bounded, let $v \in H$ be a weak cluster point of $\{x_n\}$. Then there exists $\{x_{n_j}\} \subset \{x_n\}$ such that $x_{n_j} \rightharpoonup v,$ $j \to \infty$. Then $x_{n_j+1} \rightharpoonup v,$ $j \to \infty$, $z_{n_j} \rightharpoonup v,$ $j \to \infty$
and $u_{n_j} \rightharpoonup v,$ $j \to \infty$. Replacing $n$ with ${n_j}$ in \eqref{ade1a} and \eqref{ade1b}, we have from \eqref{app7ab} and Lemma \ref{lem23} that
$$
\frac{1}{\beta}(v-x) \in \partial f(x)~~{\rm and}~~\nabla h(x)+\frac{1}{\beta}(x-v) \in \partial g(x).
$$
\noindent This implies that ${\rm prox}_{\beta f}(v)=x$ and ${\rm prox}_{\beta g}(2x-v+\beta \nabla h(x))=x$. We then obtain that
$v \in \Gamma_\beta$. \\

\noindent
We now show that $\{x_n\}$ converges weakly to an element in $\Gamma_\beta$. Define
\begin{align*}
c_n:=\|v_n-x_n\|^2+2\langle v_n-x_n, x_n-y\rangle.
\end{align*}
Since $v_n-x_n\rightarrow 0$, as $n\rightarrow \infty$ and $\{x_n\}$ is bounded, we have that $c_n\rightarrow 0$, as $n\rightarrow \infty$. Observe also that
\begin{align*}
\|v_n-y\|^2=\|v_n-x_n\|^2+2\langle v_n-x_n,x_n-y\rangle+\|x_n-y\|^2
\end{align*}
which implies that
\[
c_n = \|v_n-y\|^2-\|x_n-y\|^2.
\]
Combine it with the definition of $a_n$ to have
\[
\|x_n-y\|^2 =\frac{1}{2}\big(a_n-c_n\big).
\]
\noindent Since $\lim_{n\rightarrow \infty} a_n$ exists,  then  $\lim_{k\rightarrow \infty} \|x_n-y\|$ exists.
By Lemma \ref{lem:w-converge}, we conclude that $\{x_n\}$ converges weakly to a point in $\Gamma_\beta$.
The proof is complete.
\end{proof}

\begin{rem}${}$
\noindent The proposed Algorithm \ref{alg1} can be viewed as an extension of an inertial-type algorithm. In particular, the following is a special case of \eqref{ap} in  Algorithm \ref{alg1}:
    \begin{eqnarray}\label{love1}
\left\{\begin{array}[c]{ll}
	&u_{n-1}=v_n+\theta_{n-1}(v_n-v_{n-1}),\\
	&y_{n-1} = \underset{v \in H}{\rm argmin}\Big\{f (v)+\frac{1}{2\beta}\|v-u_{n-1} \|^2  \Big\},\\
    &z_{n-1} = \underset{v \in H}{\rm argmin}\Big\{g(v)+\frac{1}{2\beta}\|v-(2y_{n-1}-u_{n-1}+\beta \nabla h(y_{n-1})) \|^2  \Big\},\\
    &x_n=u_{n-1}+\kappa_{n-1}(z_{n-1}-y_{n-1}),\\
    &v_{n+1}=(1-\alpha_n)v_n+\alpha_nx_n
     \end{array}
      \right.
      \end{eqnarray}  
 where $\theta_n:=\frac{1-2\alpha_n}{\alpha_n} \in(-1,+\infty)$ and $0<a\leq \alpha_n\leq b<1$. To see this, 
suppose $\bar{u}_n:=v_n-x_n$. Then the first update in \eqref{ap} implies
\begin{equation}\label{love2}
u_n= v_{n+1}+\theta_n(v_{n+1}-v_n)
 = x_n+\bar{u}_n+(1+\theta_n)(v_{n+1}-v_n).
\end{equation}
We obtain from the last update in \eqref{ap}  that
$
v_{n+1}-v_n
 = -\alpha_n \bar{u}_n.
$
If we substitute this into \eqref{love2}  together with $\theta_n=\frac{1-2\alpha_n}{\alpha_n}$, we have 
\begin{eqnarray*}
u_n
&=& x_n+(1-\alpha_n-\theta_n\alpha_n)\bar{u}_n\\
&=& x_n+(1-\alpha_n-\theta_n\alpha_n)(v_n-x_n)\\
&=&(1-\alpha_n)x_n+\alpha_nv_n.
\end{eqnarray*}
Since $\theta_n=\frac{1-2\alpha_n}{\alpha_n}\in(-1,+\infty)$, we have $\alpha_n=\frac{1}{2+\theta_n}\in (0,1)\subset [0,1).$ Therefore,  \eqref{love1} is a special case of \eqref{ap}.

\end{rem}

\section{Numerical Experiments}\label{Sec:Numerics}
\noindent

\noindent The numerical experiments is devoted to the implementation of the proposed algorithms in solving real-life problems. Three examples are considered where we applied the algorithms to solve optimization models in machine learning and compare their performance with other methods in the literature such as \cite[Algorithm 1]{Chuang} and \cite[Algorithm 2]{Hu} and the popular ADMM method. All codes are written on Python using Jupyter Notebook. Also, the datasets used are source from open source such as the UCI database.

\subsection{Tuning of parameters}\label{exs1}
It is well known that control parameters are very important to the performance of the iterative algorithms. Hence, we start the numerical section by studying the behaviour of the proposed algorithms to change in the control parameters. For this purpose, we tune the values of the parameters $\theta$ and $\beta$ in the Douglas-Rashford splitting Algorithm \ref{alg2}. We consider the following optimization model:
\begin{equation}\label{ex1}
	\underset{x \in \mathbb{R}^n}{\min}~ f(x) + g(x) - h(x)
\end{equation}
where $f:\mathbb{R}^n \to [+\infty,-\infty)$ is the quadratic function defined by $f(x) = \frac{1}{2}x^\top Qx + c^\top x + d,$ with $Q \in \mathbb{R}^{M \times N}$ is  a positive definite matrix, $c \in \mathbb{R}^N$ is a constant vector and $d \in \mathbb{R}$ is a constant scalar; $g: \mathbb{R}^n \to \mathbb{R}$ is the $\ell_1$ regularization term and $h: \mathbb{R}^n \to \mathbb{R}$ is the Huber loss function given by
\begin{equation*}
	h(x) = \begin{cases}
		\frac{1}{2}x^2, \qquad & \|x\|_{2} \leq \delta, \\
		\delta(|x| - \frac{1}{2}\delta), & \mbox{otherwise}
	\end{cases}
\end{equation*}
and $\delta$ being a constant scalar. This type of model is very common in machine learning. For example, in robust classification with sparsity regularization (see, e.g. \cite{Wen}), problem \eqref{ex1} is used to learn a robust classifier that is also sparse. The Huber loss function is used in this case to handle noisy data, the quadratic term is used to regularize the model complexity and the $\ell_1$ regularization term is used to induce sparsity in the model parameters. Also problem \eqref{ex1} is used in to reconstruct a signal from noisy measurements while promoting sparsity in signal processing (see, e.g. \cite{Candes&Wakins}). The quadratic fidelity term is used to measure the fidelity of the reconstructed signal, the $\ell_1$ regularization term is used to enforce sparsity and the Huber loss function is used to handle outliers or noise in the measurement.

\noindent The proximal operator of the quadratic term is calculated analytically, which is given by
\begin{equation*}
	{\rm prox}_{\beta f}(x) = \left(Q + \frac{1}{\beta}I\right)^{-1}\left(\frac{1}{\beta}x - c\right),
\end{equation*}
where $I$ is the identity operator (see Appendix for details). To handle the inverse of the matrix $Q$ in the case where it is a non-square matrix, we used the pseudo-inverse $Q^{+}$ of $Q$ which is calculated by $Q^{+} = (Q^\top Q)^{-1}Q^\top,$ where $Q^{+}$ is the transpose of $Q$. This allows us to handle non-square of $Q$ properly. In the first experiment, we take $\kappa_n = \kappa = 0.009, \delta = 0.001$,  $\max\_\mbox{iter} = 1000.$ The algorithm is stopped if $Err = \|x_{n+1}-x_n\|^2 < 10^{-5}$ or $\max\_\mbox{iter}$ is reached. In the first instance, we choose $\theta = 0.0005, 0.05, 0.5, 1, 5$ and vary the values of $(M,N)$ as $(100,25), ~ (200,100), ~ (250, 150)$ and $(300,200)$. The starting points $x_0$ and $ v_0$ are generated randomly. The results of the experiments are shown in Table \ref{tab1} and Figure \ref{fig1}.

\noindent Furthermore, taking $\theta = 0.0005,$ we test the algorithm for  $\beta = 0.0001, 0.001, 0.1, 1, 10$ and vary the value of $(M,N)$ as $(100,25), ~ (200,100), ~ (300,150),$ (400,200), (500,300) and  $(600, 400).$ The results are shown in Table \ref{tab2} and \ref{fig2}.

\begin{figure}[h!]
	\begin{subfigure}{1.0\textwidth}
		\centering
		\includegraphics[width=.3\linewidth]{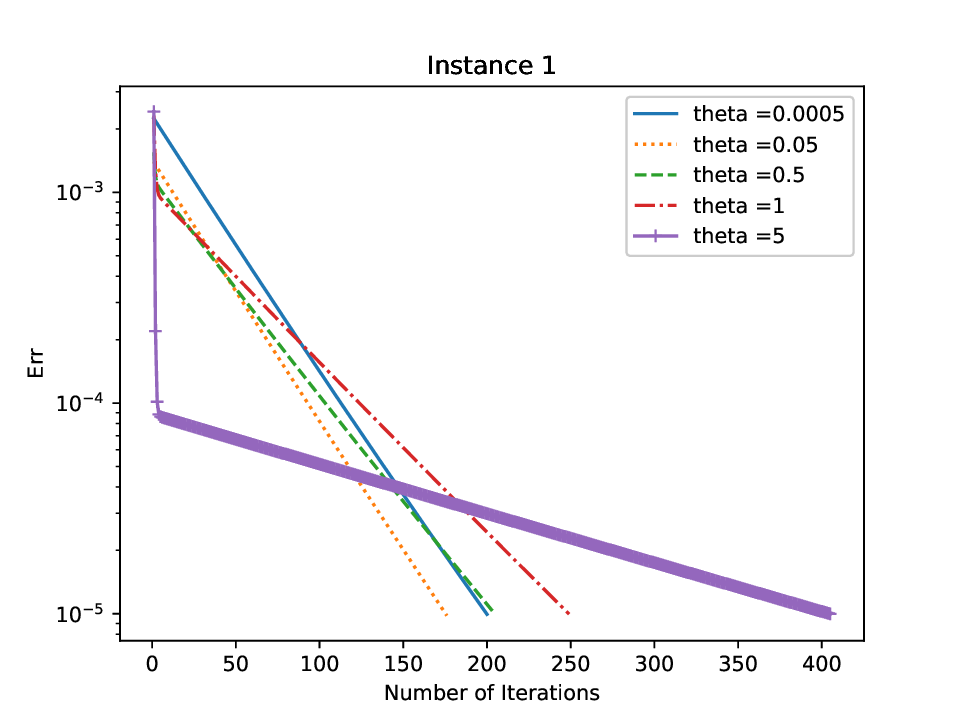}
		\includegraphics[width=.3\linewidth]{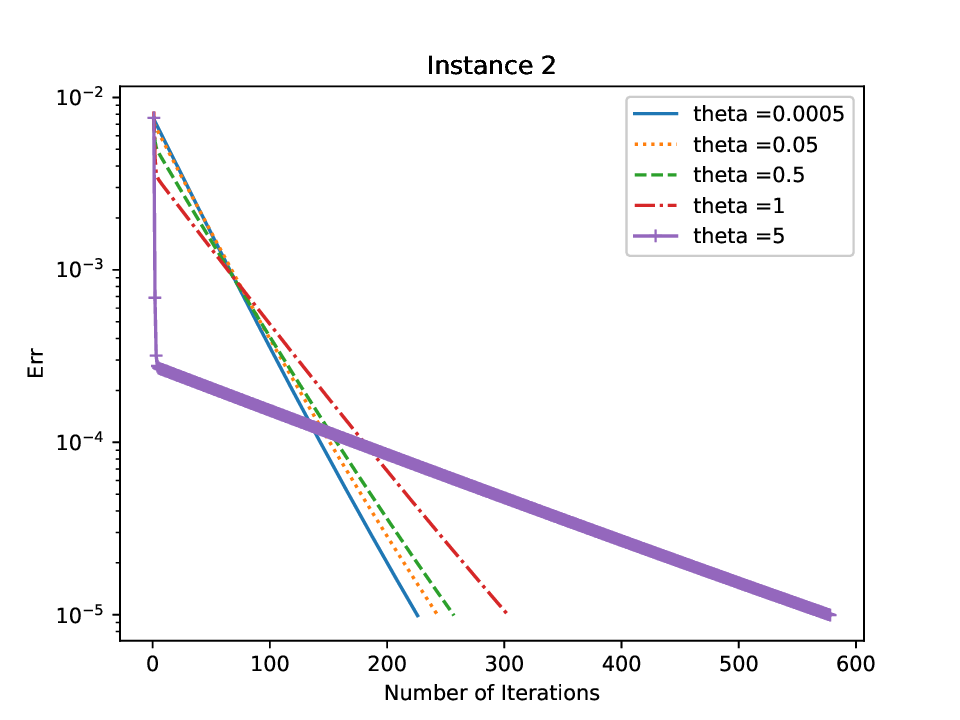}
		\includegraphics[width=.3\linewidth]{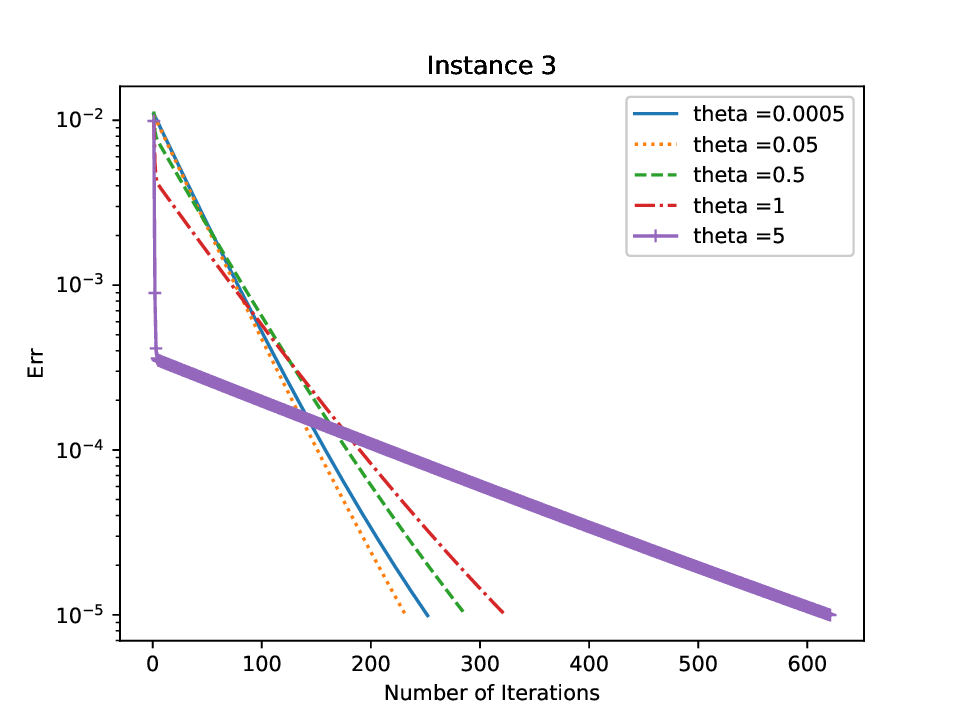}
		\includegraphics[width=.3\linewidth]{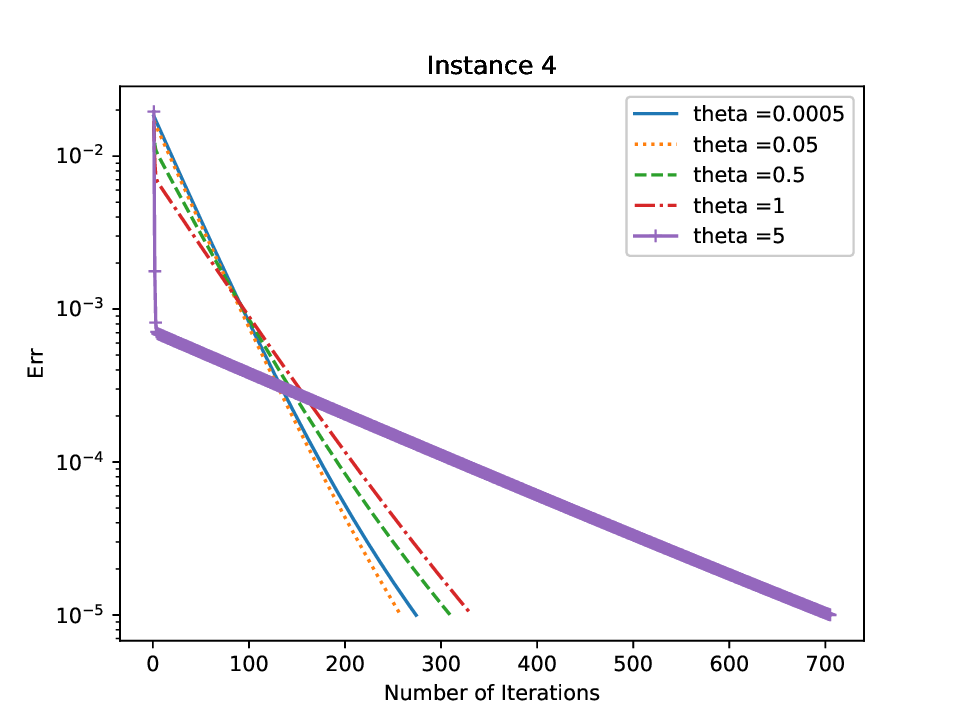}
		\includegraphics[width=.3\linewidth]{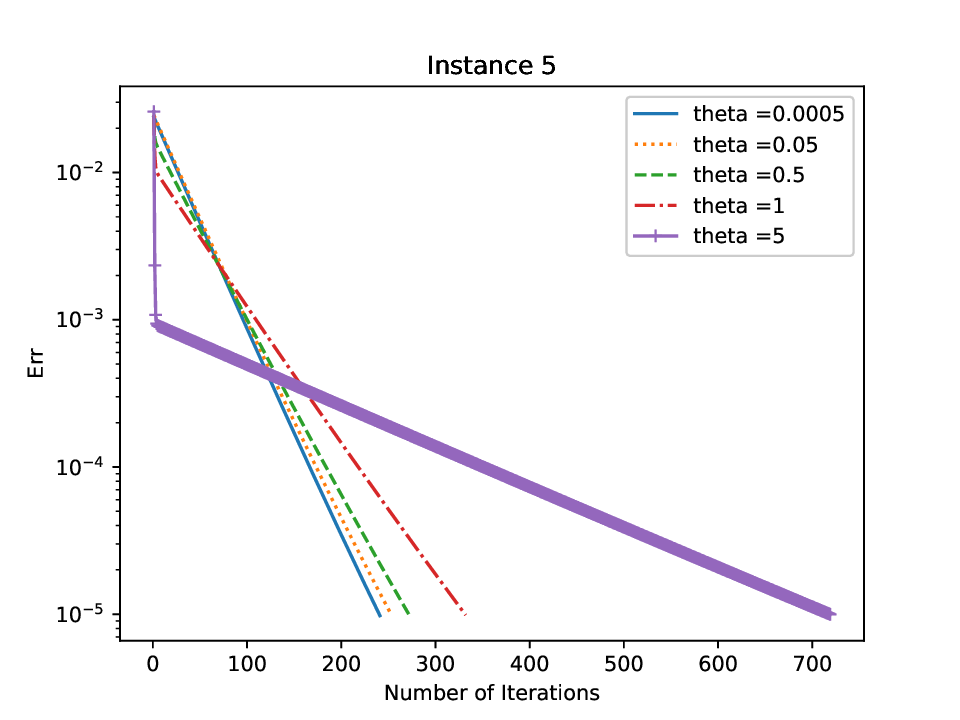}
		\includegraphics[width=.3\linewidth]{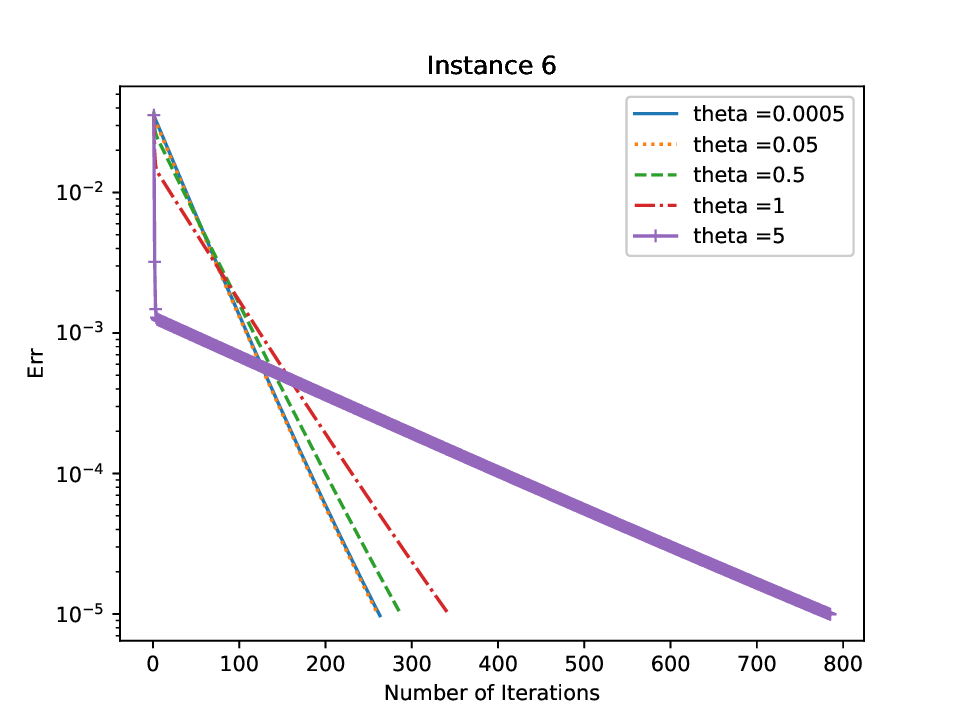}
	\end{subfigure}
	\caption{ Comparing the performance of proposed Douglas-Rashford splitting algorithm for various values of $\theta$ in Example \ref{exs1} using various dimension of $(M,N)$. Instance 1: (100,25); Instance 2: (200,100), Instance 3: (300,150), Instance 4: (400, 200), Instance 5: (500,300) and Instance 6: (600,400).} \label{fig1}
\end{figure}

\begin{figure}[h!]
	\begin{subfigure}{1.0\textwidth}
		\centering
		\includegraphics[width=.3\linewidth]{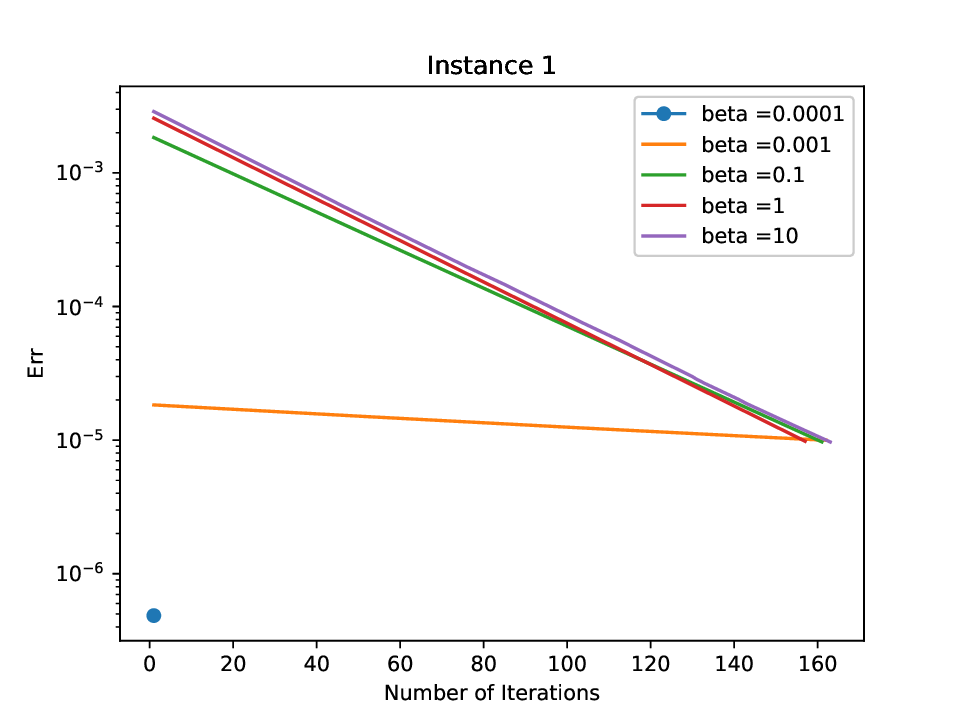}
		\includegraphics[width=.3\linewidth]{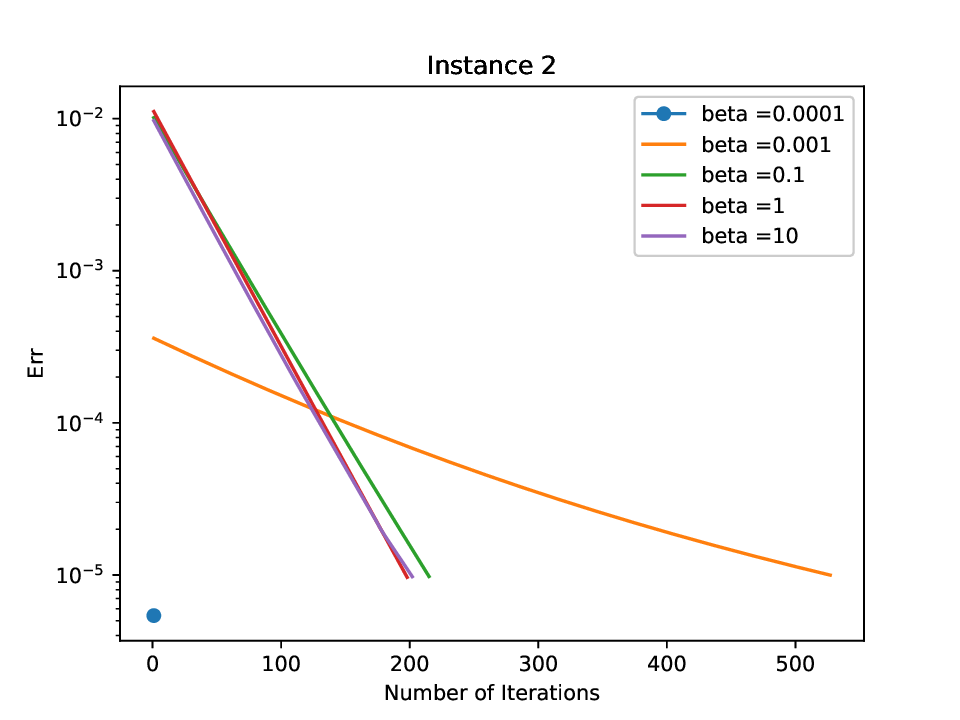}
		\includegraphics[width=.3\linewidth]{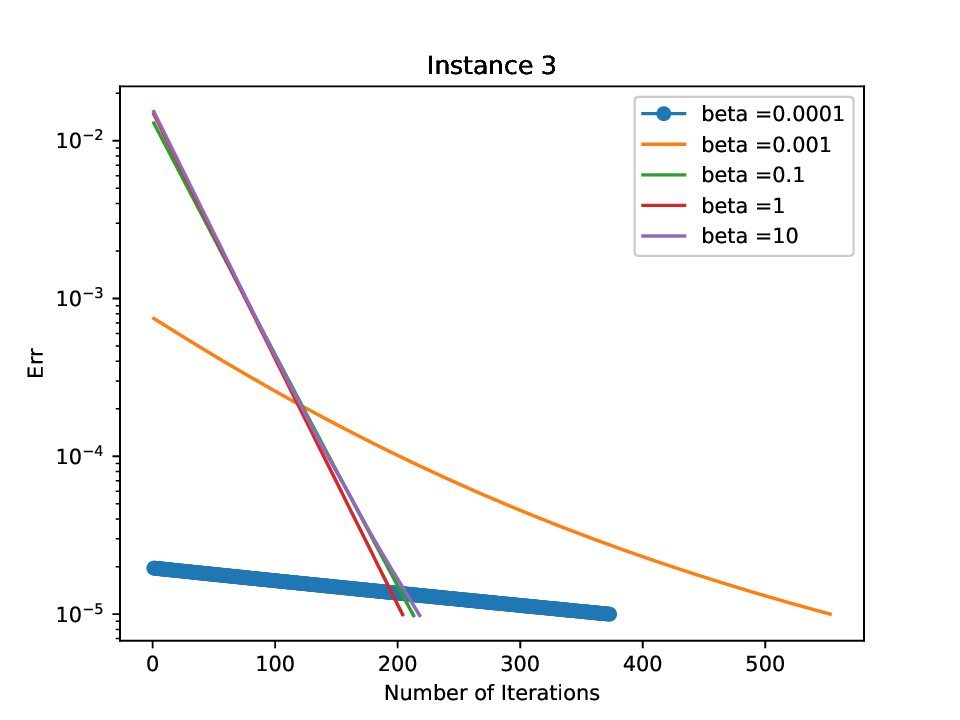}
		\includegraphics[width=.3\linewidth]{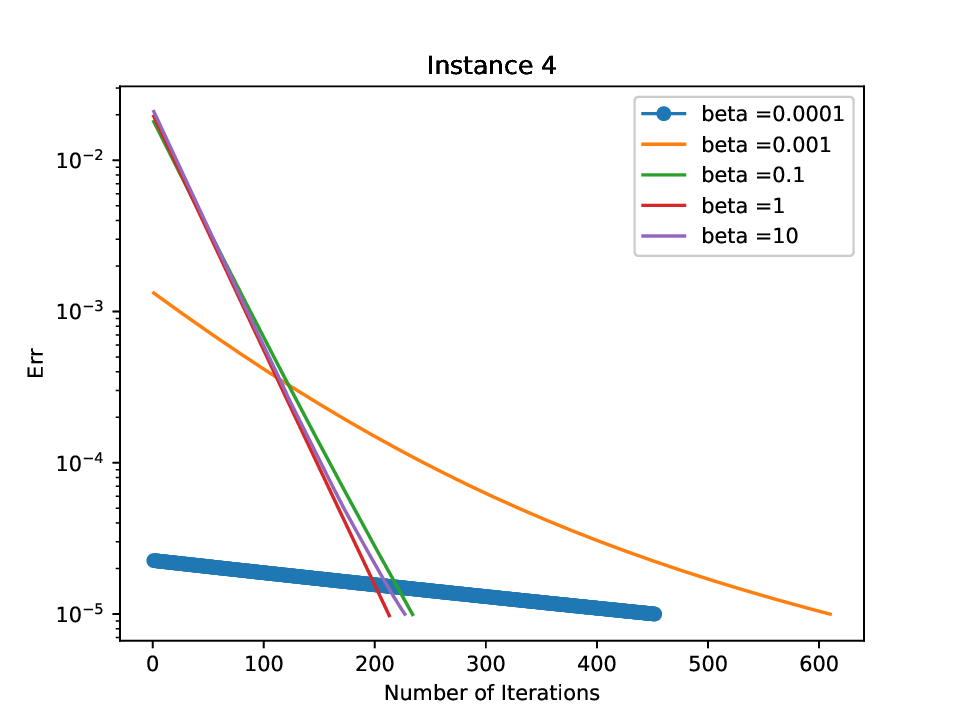}
		\includegraphics[width=.3\linewidth]{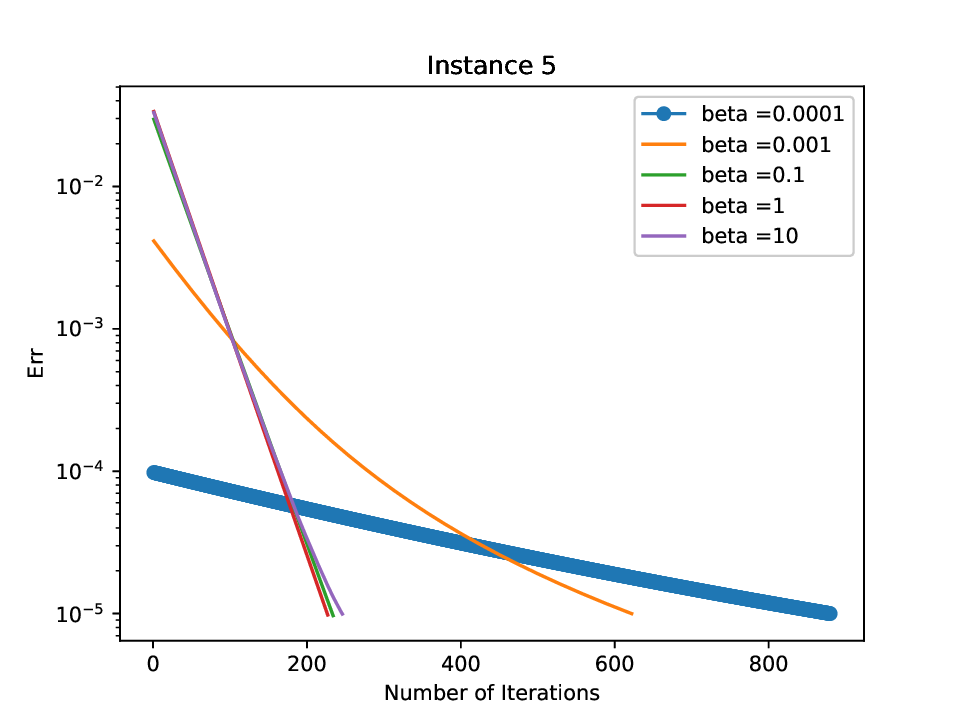}
		\includegraphics[width=.3\linewidth]{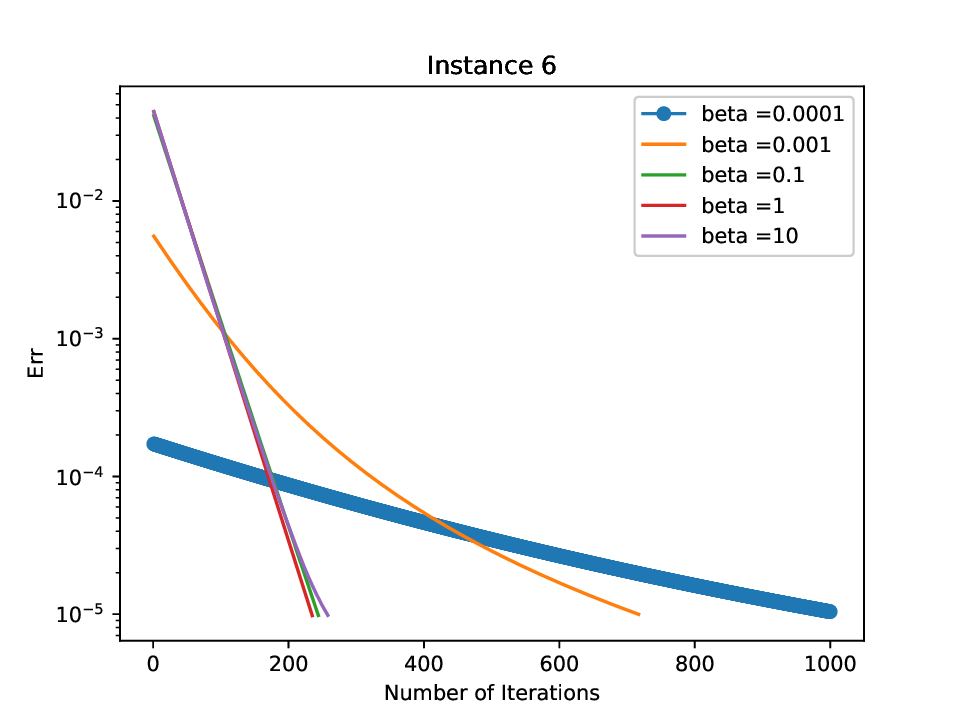}
	\end{subfigure}
	\caption{ Comparing the performance of proposed Douglas-Rashford splitting algorithm for various values of $\beta$ in Example \ref{exs1} using various dimension of $(M,N)$. Instance 1: (100,25); Instance 2: (200,100), Instance 3: (300,150), Instance 4: (400, 200), Instance 5: (500,300) and Instance 6: (600,400).} \label{fig2}
\end{figure}

\begin{table}[!h]
	\centering
	\begin{tabular}{llllll}
		\toprule
	$(M,N)$ & $\theta = 0.0005$ & $\theta = 0.05 $ & $\theta = 0.5 $ & $\theta = 1 $ & $\theta = 5$ \\
	 &	Iter/Time & Iter/Time & Iter/Time & Iter/Time & Iter/Time \\
		\midrule
		(100,25) & 201/ 0.0712 & 177/0.0532 & 206/0.0473 & 250/0.0626 & 406/0.1098 \\
		(200,100) & 227/0.2670 & 244/0.3138 & 258/0.3298 & 304/0.4152 & 579/0.6904 \\
		(300,150) & 253/0.276 & 233/0.2830 & 288/0.3451 & 324/0.3999 & 622/0.8319 \\
		(400,250) & 275/0.5952 & 259/0.5378 & 311/0.6616 & 333/4.2368 & 706/1.5850 \\
		(600,300) & 242/2.7268 & 254/2.5572 & 273/2.8729 & 333/4.2368 & 720/9.0847 \\
		(600,400) & 264/6.1672 & 262/5.8376 & 289/6.8894 & 344/8.6013 & 786/26.6056 \\
		\bottomrule
	\end{tabular}
	\caption{Results of varying parameter $\theta$ in the proposed Douglas-Rashford splitting algorithm \ref{alg2}.} \label{tab1}
\end{table}

\begin{table}[!h]
	\centering
	\begin{tabular}{llllll}
		\toprule
		$(M,N)$ & $\beta = 0.0001$ & $\beta = 0.001 $ & $\beta = 0.1 $ & $\beta = 1 $ & $\beta = 10$ \\
		&	Iter/Time & Iter/Time & Iter/Time & Iter/Time & Iter/Time \\
		\midrule
		(100,25) & 2/0.0013  & 163/0.0584 & 162/0.0373 & 158/0.0313 & 164/0.0629 \\
		(200,100) & 2/0.0020 & 528/0.3301 & 216/0.1566 & 199/0.1415 & 203/0.1254 \\
		(300,150) & 374/0.8984 & 554/1.2979 & 214/0.5333 & 205/0.4861 & 219/0.6669 \\
		(400,250) & 453/1.1677 & 611/1.7133 & 235/0.6431 & 214/0.5807 & 228/ 0.6119 \\
		(600,300) & 881/9.2771 &  623/6.1512 & 235/2.3229 & 228/2.2761 & 247/2.9510 \\
		(600,400) & 1000/17.0568 & 718/12.6837 & 245/4.8480 & 236/4.2068 & 259/5.3359 \\
		\bottomrule
	\end{tabular}
	\caption{Results of varying stepsize $\beta$ in the proposed Douglas-Rashford splitting algorithm \ref{alg2}.} \label{tab2}
\end{table}

\noindent In the initial experiment, where we assess the performance of the algorithm across a range of $\theta$ values, we observe some distinct trends in the behaviour of the algorithm. Notably, we find that the effectiveness of the algorithm diminishes as $\theta$ exceeds a critical threshold $(\theta > 1)$, signaling a clear degradation in performance. This decline is further worsen with increase  in the values of both $M$ and $N$, providing evidence of an overestimation in the control parameter. Conversely, a noteworthy observation emerges when the value of $\theta$ is smaller. The algorithm exhibits enhanced performance with smaller values of $\theta$, even with larger values dimension of $M$ and $N$. Additionally, upon closer examination, we identify that the algorithm reaches its best performance when $\theta$ is set to 0.05, underscoring the importance of optimal parameter selection in maximizing the effectiveness of the proposed algorithm.

\noindent In the subsequent experiment, dedicated to evaluating the performance of the algorithm across different values of the stepsize $\beta$, some intriguing behaviour emerge. Particularly, for smaller values of $N$ and $M$, we observe improved algorithmic performance with decreased $\beta$, notably with $\beta = 0.0001$. However, as the values of $M$ and $N$ increase, a drastic deterioration in algorithmic efficacy is noted for smaller $\beta$ values, indicative of an underestimation in stepsize. Interestingly, regardless of the specific values of $M$ and $N$, the algorithm consistently achieves better performance with large value of $\beta$. It is noted that the best performance of the algorithm is achieved when $\beta = 1.$

\subsection{Regularized least squares problem with the logarithmic regularizer}\label{ex2}
\noindent Next, we shall consider the regularized least square (RLS) model with logarithmic regularizer given as follows:
\begin{equation}\label{dcexa}
  \min_{w \in \mathbb{R}^N} \mathcal{J}(w)=\frac{1}{2}\|Aw-b\|^2+\sum_{i=1}^{N}(\mu \log (|w_i|+\epsilon)-\mu \log \epsilon),
\end{equation}
where $A \in \mathbb{R}^{m \times  N}, b \in \mathbb{R}^m$,  $\epsilon>0$ is a constant and $\mu>0$ is the regularization parameter. RLS with logarithmic regularizer is very common in machine learning, mainly for preventing overfitting in models, see, e.g. \cite{hastie2001elements,Bishop,Murphy}.  It is easy to see that the model \eqref{dcexa} is a special case of DC programming \eqref{dc1} with $f(w)=\frac{1}{2}\|Aw-b\|^2$, $g(w)=\frac{\mu}{\epsilon}\|w\|_1$ and
$h(w)= \sum_{i=1}^{N}\mu\Big(\frac{|w_i|}{\epsilon}-\log (|w_i|+\epsilon)+\log \epsilon \Big)$. The first-order optimization condition of the first subproblem in Algorithm \ref{alg2} and Algorithm \ref{alg1} give (\cite{Hu})
$$
A^*(Ay_n-b)+\frac{1}{\beta}(y_n-u_n)=0,
$$
\noindent which yields
$$
(\beta A^* A+I)y_n=\beta A^* b+u_n,
$$
\noindent and can be solved effectively by the LU factorization method, the conjugate gradient method and so on \cite{Boyd}.\\

\noindent In our experiments, we test the performance of the proposed algorithms for finding the minimizer of \eqref{dcexa} and compare the results with other methods in the literature. The matrices $A \in \mathbb{R}^{m \times N}$ and $b \in \mathbb{R}^m$ are generated randomly  with i.i.d. standard Gaussian entries, and then normalized so that the
columns of $A$ have unit norms. Also the vector $b$ is generated randomly such that $b$ is non-zero vector. The proposed Algorithm \ref{alg2} and \ref{alg1} are compared with \cite{PhanDCA} (DCA) and \cite[Algorithm 1]{Chuang} 
 (GDCP). We initialize the algorithms by generating the starting points randomly, $\mu = 0.001, \epsilon = 0.5,$ $max\_iter = 1000, \beta = 0.04, \kappa_n  = \frac{n}{n+10},$ $\theta = 0.9$ for Alg \ref{alg2}, $\alpha_n = \frac{1}{n+1}$ for Alg \ref{alg1} and $\alpha = 0.09$ for GDCP.
We stopped the algorithms when $$
Err = \frac{\|x_n-x_{n-1}\|}{\max\{1,\|x_n\|\}}<10^{-5}
$$ and record the number of iteration and CPU time taken for each algorithm. The results of the experiments are recorded in Table \ref{tab_example2} and Figure \ref{fig_example2}.

\begin{table}[!h]
	\centering
	\begin{tabular}{|l|l|l|l|l|}
		\toprule
		$(m,N)$ & Alg 1 & Alg 2 & DCA & GDCP  \\
		&	Iter/Time & Iter/Time & Iter/Time & Iter/Time  \\
		\midrule
		(100,50) & 156/0.0708 & 212/0.0398 & 331/0.0785 & 1000/0.1567 \\
		(200,128) & 160/0.1762 & 217/0.1728 & 343/0.2659 & 1000/0.8564 \\
		(521,304) & 168/1.0966 & 228/1.4425 & 366/1.8522 & 760/3.8603 \\
		(700,500) & 169/6.3734 & 226/7.8930 & 365/12.1768 & 763/26.9791 \\
		(1000,700) & 171/11.5587 & 231/19.5223 & 369/30.5246 & 760/63.9168 \\
		(1500,1000) & 174/29.7912 & 236/47.7586 & 378/72.9514 & 759/145.9198 \\
		\bottomrule
	\end{tabular}
	\caption{Comparison of numerical results for Example \ref{ex2}.} \label{tab_example2}
\end{table}

\begin{figure}[h!]
	\begin{subfigure}{1.0\textwidth}
		\centering
		\includegraphics[width=.3\linewidth]{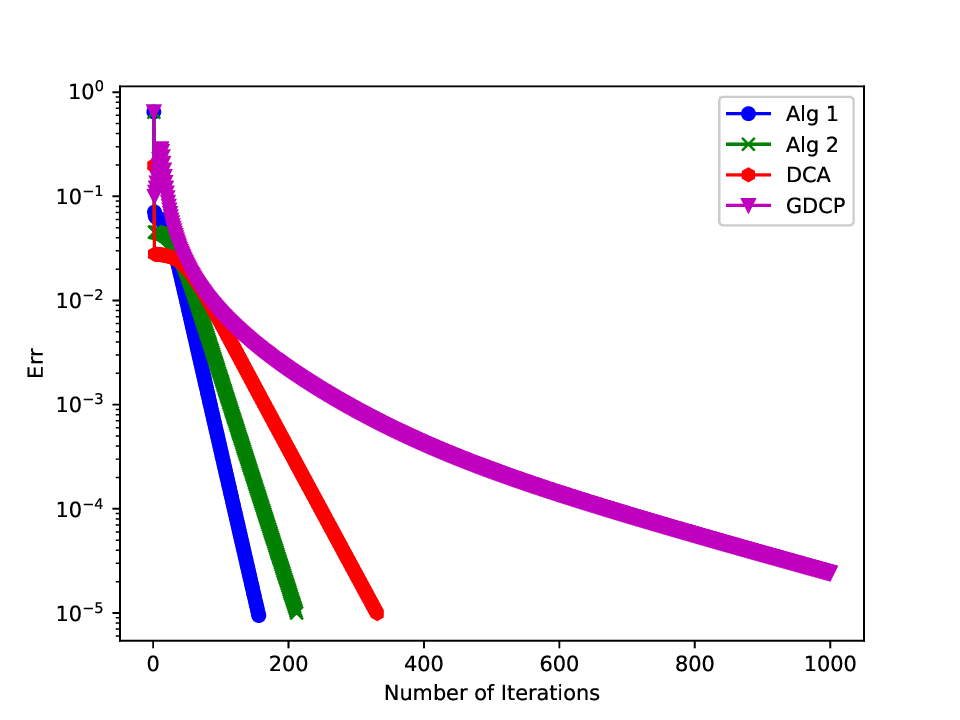}
		\includegraphics[width=.3\linewidth]{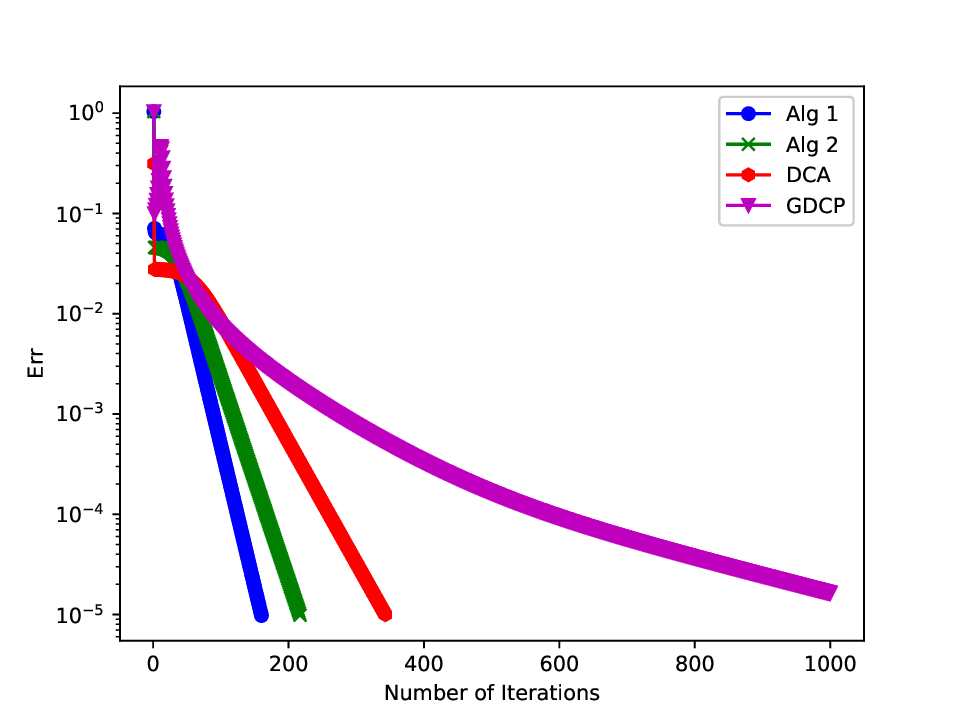}
		\includegraphics[width=.3\linewidth]{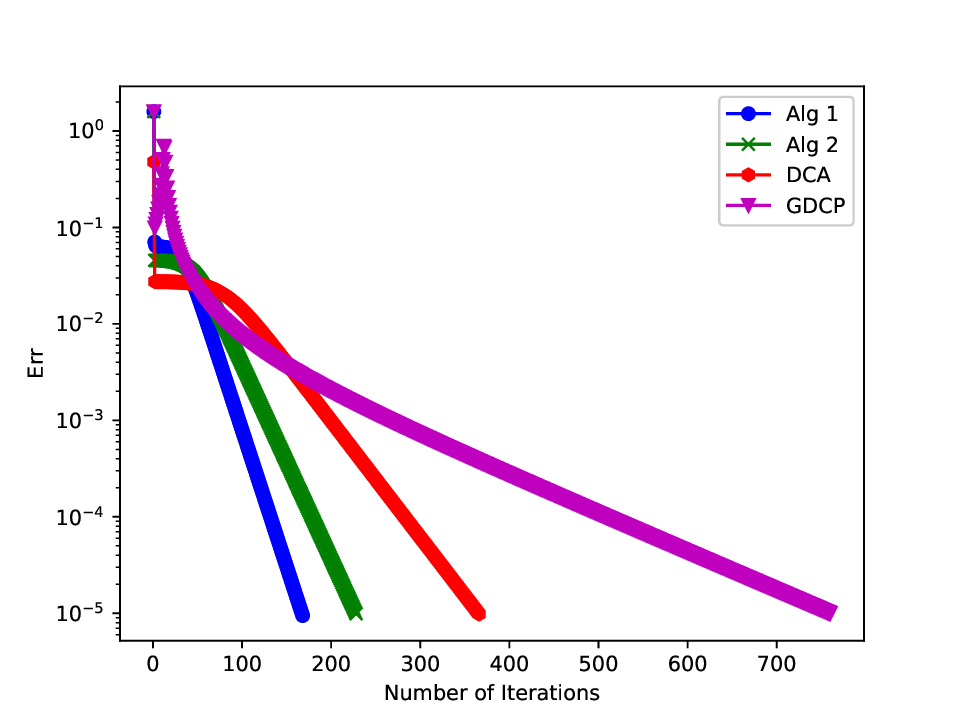}
		\includegraphics[width=.3\linewidth]{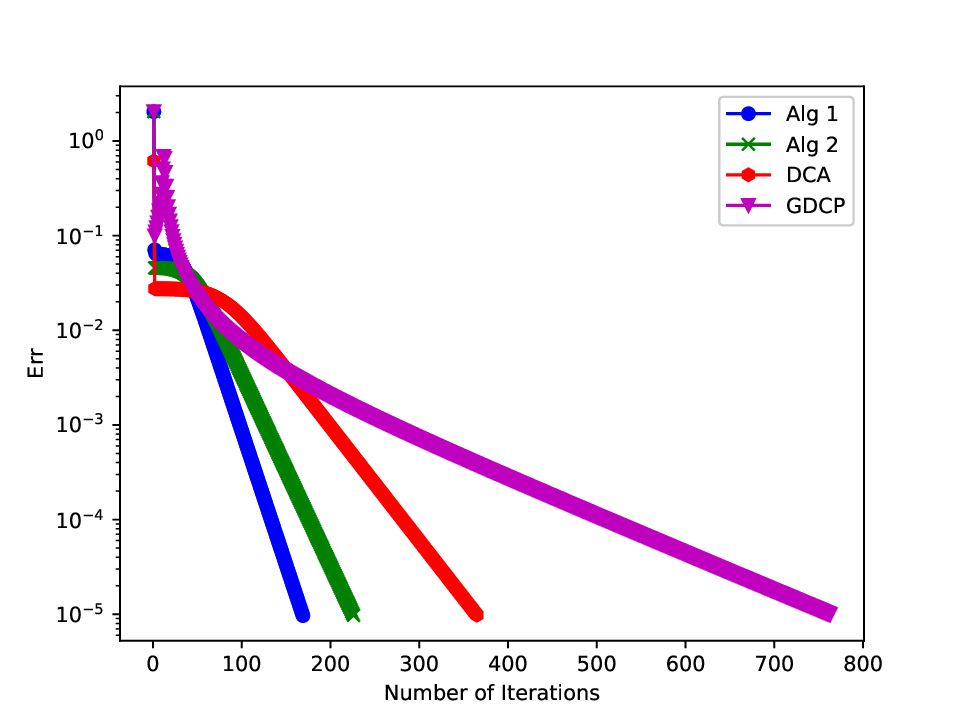}
		\includegraphics[width=.3\linewidth]{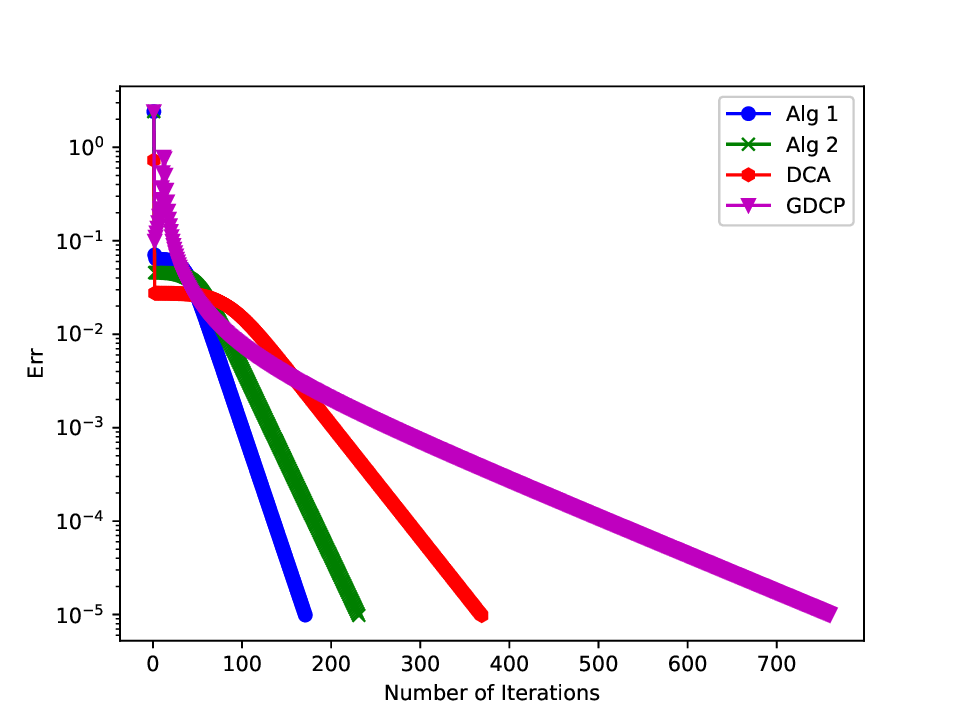}
		\includegraphics[width=.3\linewidth]{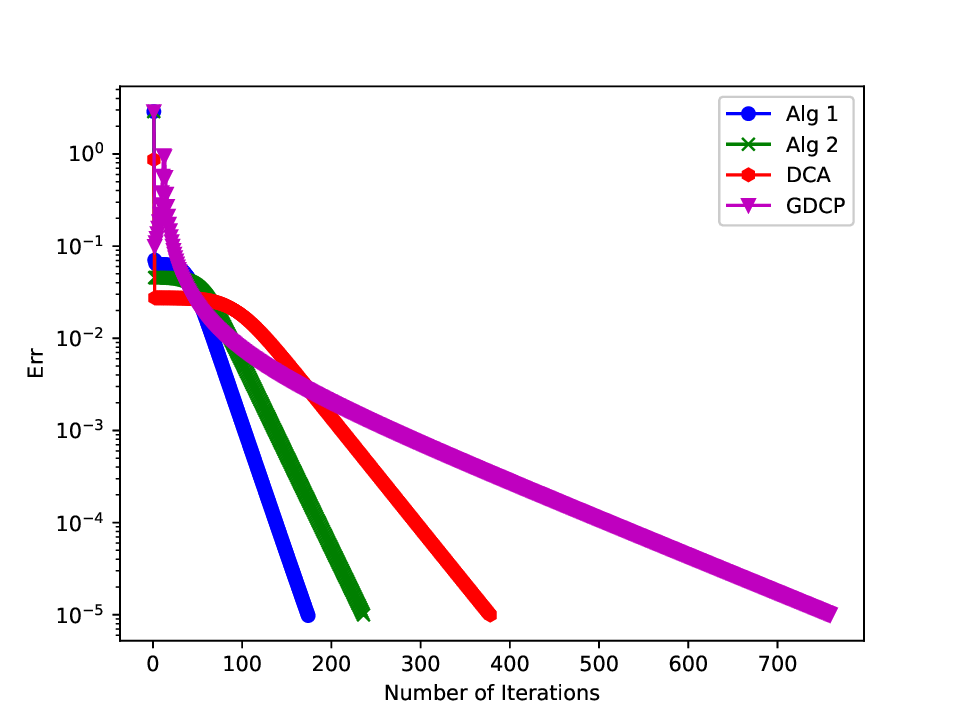}
	\end{subfigure}
	\caption{ Graphical illustration of values of Err against number of iterations by each method in Example \ref{ex2}. Instance 1: (100,50); Instance 2: (250,128), Instance 3: (521,304), Instance 4: (700, 500), Instance 5: (1000,700) and Instance 6: (1500,1000).} \label{fig_example2}
\end{figure}

The numerical results show that efficiency and accuracy of the proposed algorithms. Both proposed methods consistently have the best performance than the comparing methods in terms of number of iterations and time of execution irrespective of the size of the data size. This is a desirable result which give credence to the importance of the proposed methods.


\subsection{Support vector machines (SVM) with L1 regularization}

Given  a dataset $\{(x_{i},y_{i})\}_{i=1}^{n}$ where $x_{i}$ represents the number of the features and $y_{i}$ is the corresponding label $(+1~\mbox{or}~-1)$ of the $i$-th sample, the primal form of the SVM optimization model can be written as
\begin{equation}
	\underset{w,b}{\mbox{minimize}}~\frac{1}{2}\|w\|_{2}^2 + C \sum_{i=1}^{n}~ \max \big(0,1 - y_{i}(w^\top x_{i}+b)\big) + \lambda \|w\|_{1}
\end{equation}
where
$w$ is the weight of vector, $b$ is the bias term, $\lambda$ is the regularization parameter for the L1 penalty, $C$ is the regularization parameter for the hinge loss function. The first term of the model represents a convex function which deals with sparsity in the model, the second term is the hinge loss function penalizing misclassification and the third term is the L1 regularization term dealing with large weights. The objective function is a good example of a generalized DC programming problem written as
\begin{equation*}
	\min_{w,b} ~ f(w,b) + g(w) - h(w),
\end{equation*}
where $f(w,b) = \|w\|_{2}^2 + C\sum_{i=1}^{n} \max (0, 1- y_{i}(w^\top x_{i} + b)),$ $g(w) = \lambda \|w\|_{1}$ and $h(w) = \frac{1}{2}\|w\|_{2}^2$.

\noindent The SVM was introduced by Vapnik \cite{Cortes,Vapnik1} as a kernel-based model used for classification and regression tasks in machine learning. Its exceptional generalization ability, optimal solution, and discriminative power have garnered significant attention from the data mining, pattern recognition, and machine learning communities in recent times. SVM has been used as a powerful tool for addressing real-world binary classification problems. Researches have demonstrated the superiority of SVMs over alternative supervised learning techniques \cite{Ukey,Huang,Wang}. Due to their robust theoretical foundations and remarkable generalization capabilities, SVMs have risen to prominence as one of the most widely adopted classification methods in recent times. A detailed theoretical explanation and formulation of the SVM can be found in, for instance, \cite{Vapnik1,Bishop,Murphy,hastie2001elements}.

\subsubsection*{Related works}
Considerable research works have been dedicated to exploring the applications of DC programming across various domains including operational research, machine learning, and economics. An intriguing application of DC programming was presented in the 2006 paper of Argiou et al \cite{Argiou} titled "A DC programming algorithm for kernel selection". The author discussed a greedy algorithm aimed at learning kernels from a convex hull comprising basic kernels. While this methodology had been previously introduced, it was confined to a finite set of basic kernels. Additionally, the authors noted that the applicability of their approach was constrained by the non-convex nature of a critical maximization step involved in the learning process. However, they discovered that the optimization problem could be reformulated as a DC program without solving the problem. Another noteworthy paper discussed in this direction is the paper by Thi et al \cite{Thi}. The authors introduced an innovative application of DC programming within the context of SVM algorithms. Their focus lay in the optimal selection of representative features from data, a pivotal task in enhancing the efficiency and interpretability of machine learning models. The authors equate this problem to minimizing a zero norm function over step-k feature vectors. The zero-norm function counts the number of non-zero elements in the feature vector, essentially quantifying the sparsity of the selected features. This innovative approach opens up new avenues for enhancing the interpretability and performance of machine learning models in various applications. In the paper \cite{Thi3}, the authors leveraged DC decomposition techniques to address a challenging optimization problem within the context of machine learning. Specifically, they employed a DC decomposition-based algorithm known as DCA (DC Algorithm) to efficiently find local minima of the objective function. This algorithm iteratively updates the solution by solving convex subproblems and performing a DC decomposition-based update step, converging towards a local minimum of the objective function. The authors applied this approach to ten datasets, some of which were particularly sparse. In the thesis \cite{Nguyen}, the author conducted a comprehensive survey focusing on the unified DC programming method for feature selection within the framework of semi-supervised SVM and multiclass SVM. The primary algorithm emphasized in the paper is based on the subgradient optimality principle. The unified DC programming method discussed in the thesis aims to identify the most relevant subset of features from high-dimensional data while optimizing the performance of the SVM classifier, and demonstrating its applicability and efficacy across different problem domains. In the more recent thesis by Ho \cite{Ho}, the author delved into the application of DC  programming for learning a kernel tailored to a specific supervised learning task. Notably, what sets this approach apart from others in the literature is its capability to handle an infinite set of basic kernels. Other related results can be found in \cite{Thi2,Yao}.

\subsubsection*{Experimental setup}
We begin by describing the datasets used in the experiments. We consider the following two datasets which are source from the UCI machine learning database and Kaggle:
\begin{itemize}
	\item Dataset1: Bank-Note-Authentication: \url{https://archive.ics.uci.edu/dataset/267/banknote+authentication}
	\item Dataset2: Creditcard: \url{https://www.kaggle.com/datasets/mlg-ulb/creditcardfraud}
\end{itemize}
Dataset1 contains 1,372 rows with 5 columns consisting of 4 features and 1 target attribute/class classifying whether a given banknote is authentic or not. The data were extracted from images that were taken from genuine and forged back note specimen and has been used in many research works in machine learning. The target class contains two variables; namely, 0 which represents genuine note and 1 which represents fake note. More so, the data contains a balance ratio of both classes which is 55:45 (genuine:fake) and has no missing value. Hence, there is no need for a data cleaning purpose, which means that, we implemented the data directly in our experiment. Details of the dataset can be found in Table \ref{dataset1}.

\noindent Dataset2 contains transaction made by credit cards in September 2013 by European cardholders within two days. It contains 234,807 transaction out of which 492 are fraud which indicates a highly imbalanced dataset. Hence, we carried out a feature engineering task to avoid wrong classification of the model due to the imbalance in the dataset. We used the undersampling technique by reducing the number of samples in the majority class to match the number of samples in the minority class. This helps to balance the dataset and makes it less skewed towards the majority class. Then we used the 'resample' function from scikit-learn to randomly selects a subset of samples from the majority class without replacement, ensuring that each sample is selected only once. The number of samples selected is equal to the number of samples in the minority class, ensuring a balance class distribution. This is then used to form a new dataset which is used for our experiment. Details of the original imbalanced dataset and the balance dataset after undersampling can be seen in Figure \ref{dataset2a} and \ref{dataset2b}.

In this experiments, we compare the performance of Algorithm \ref{alg2} and \ref{alg1} with the ADMM \cite{SunADMM}, DCA \cite{PhanDCA} and GDCP \cite{Chuang}. In the model, we choose the following general parameters for implementing the algorithm: $C = 1,$ $\lambda = 0.001,$ for Algorithm \ref{alg2}, \ref{alg1}, ADMM, DCA and GDCP, we choose $\theta = 0.01, \beta = 0.001, \kappa_n = 0.3, \alpha_n = \frac{1}{10(n+1)}$ and the declared parameters used in the original papers of the comparing methods. The maximum allowed iteration is 2000 and we also set the algorithms to stop if $\|x_{k+1} - x_{k} \| < 10^{-4}.$

\begin{table}
	\centering
		\caption{Description of dataset1.}\label{dataset1}
	\begin{tabular}{|l|r|r|r|r|r|}
		\hline
		&    variance &   skewness &   curtosis &    entropy &       class \\
		\hline
		count & 1372        & 1372       & 1372       & 1372       & 1372        \\
		mean  &    0.433735 &    1.92235 &    1.39763 &   -1.19166 &    0.444606 \\
		std   &    2.84276  &    5.86905 &    4.31003 &    2.10101 &    0.497103 \\
		min   &   -7.0421   &  -13.7731  &   -5.2861  &   -8.5482  &    0        \\
		25\%   &   -1.773    &   -1.7082  &   -1.57498 &   -2.41345 &    0        \\
		50\%   &    0.49618  &    2.31965 &    0.61663 &   -0.58665 &    0        \\
		75\%   &    2.82147  &    6.81462 &    3.17925 &    0.39481 &    1        \\
		max   &    6.8248   &   12.9516  &   17.9274  &    2.4495  &    1        \\
		\hline
	\end{tabular}
\end{table}

\begin{figure}[h!]
	\begin{subfigure}{1.0\textwidth}
		\centering
		\includegraphics[width=.45\linewidth]{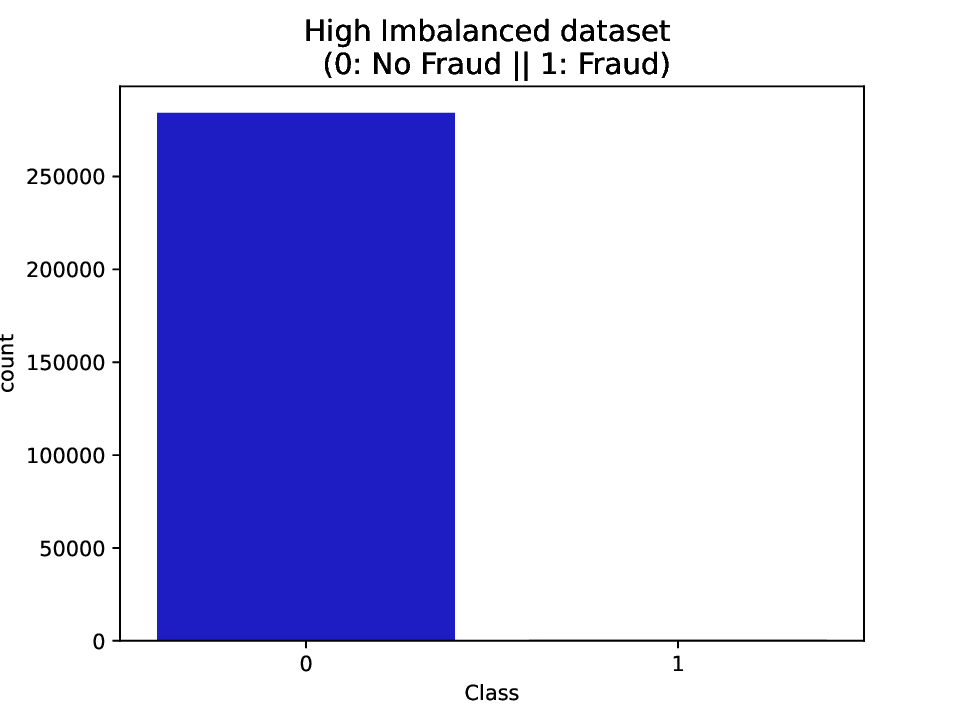}
		\includegraphics[width=.45\linewidth]{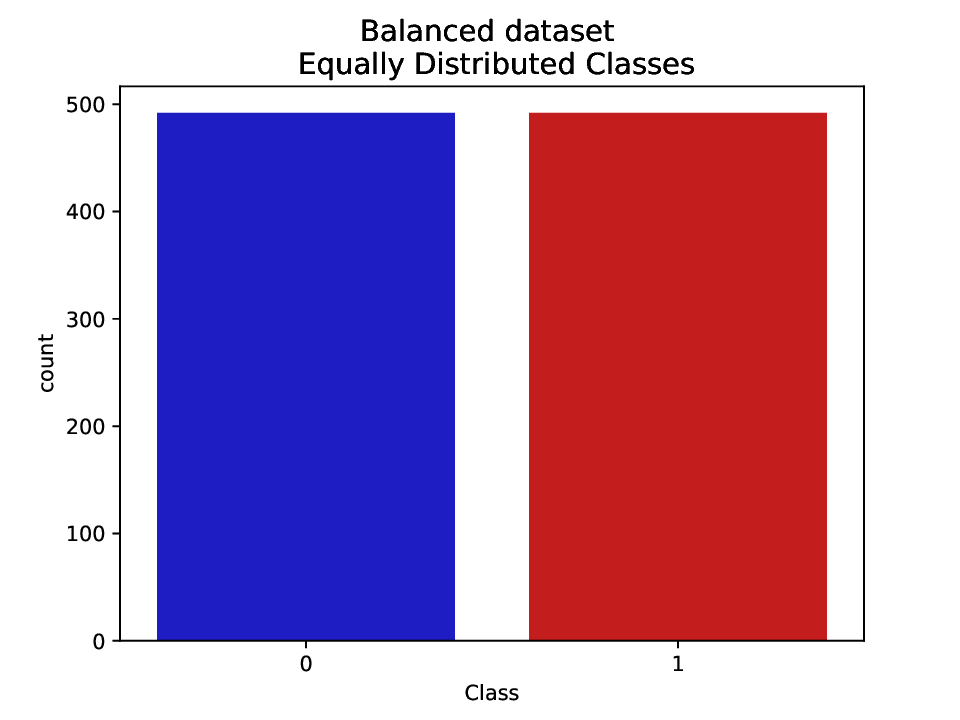}
	\end{subfigure}
	\caption{Highly imbalanced dataset2 (Left) and Balanced dataset2 after undersampling (Right)} \label{dataset2a}
\end{figure}

\begin{figure}[h!]
	\begin{subfigure}{1.0\textwidth}
		\centering
		\includegraphics[width=.47\linewidth]{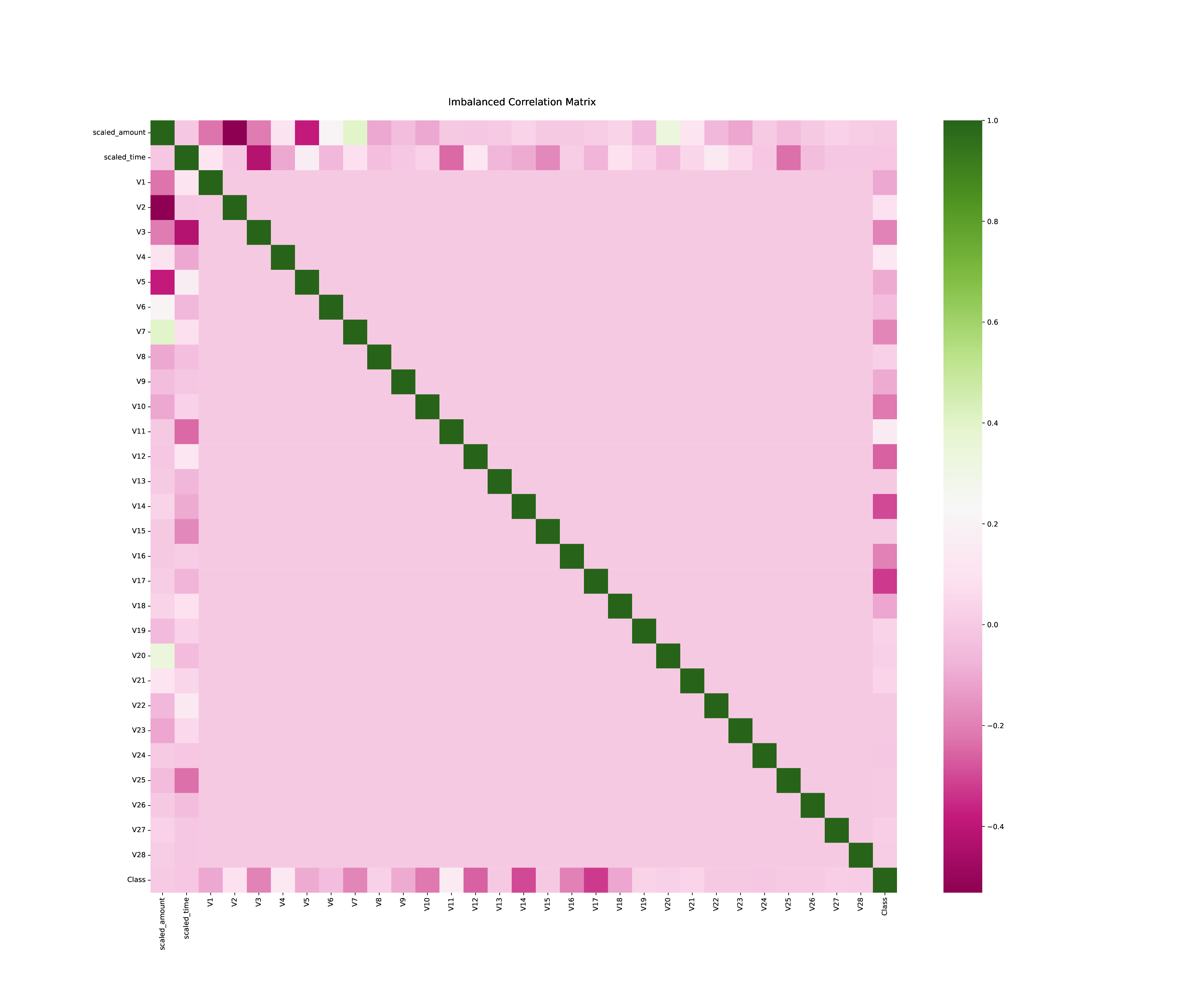}
		\includegraphics[width=.47\linewidth]{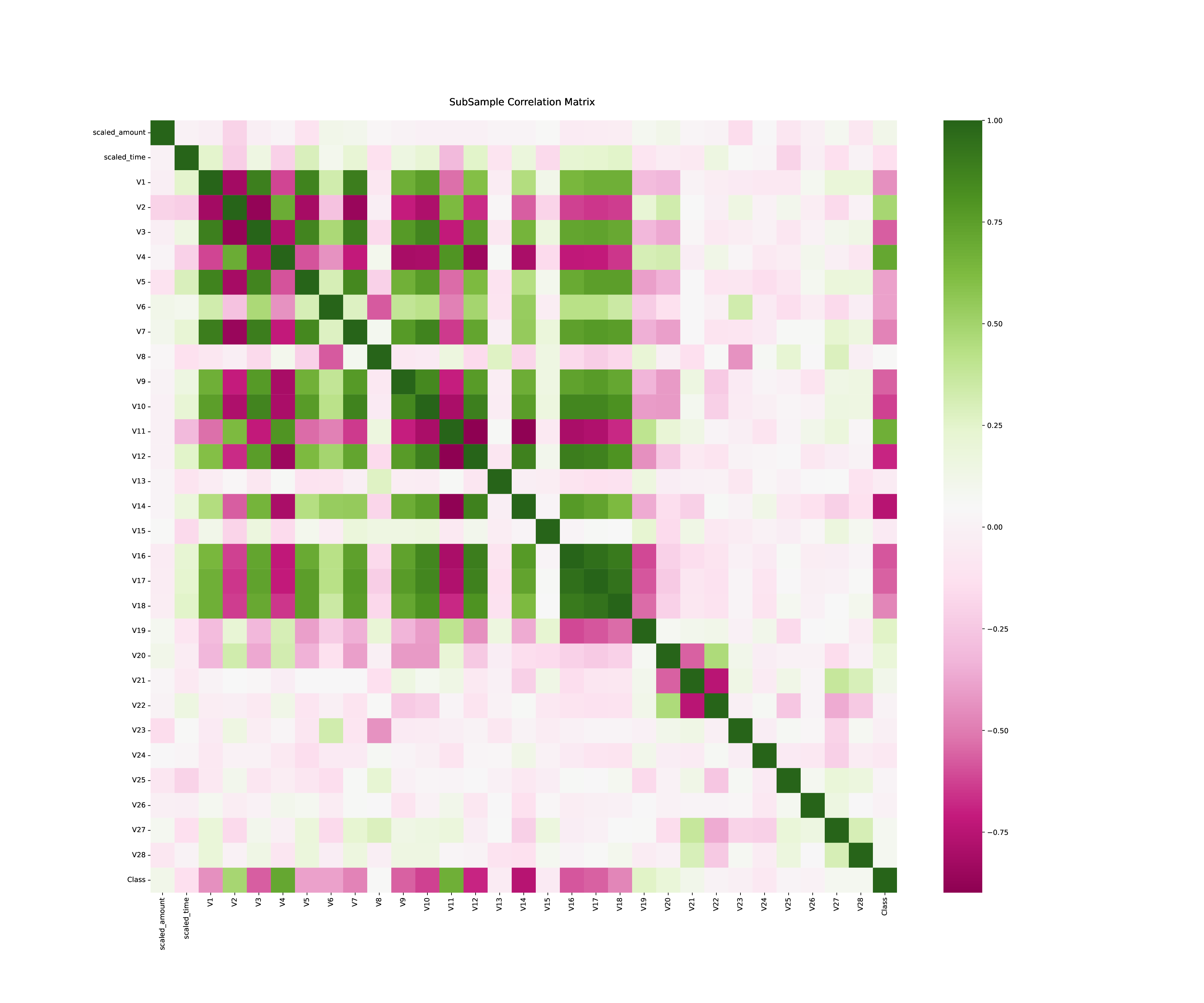}
	\end{subfigure}
	\caption{Correlation matrix of features of imbalanced dataset2 (Left) and correlation matrix of features of balanced dataset2 after undersampling (Right).} \label{dataset2b}
\end{figure}


\subsubsection*{Experimental results}
We divided the training data into four subsets with varying ratios of testing to training data: 10\% testing and 90\% training, 20\% testing and 80\% training, 30\% testing and 70\% training, and 40\% testing and 60\% training using the \texttt{train\_test\_split} function from \texttt{sklearn} in python. For dataset1, we carried out preprocssing technique to ensure that the data is standardized and suitable for training the model. For dataset2, we carried out undersampling feature selection to ensure that the model does not wrongly classify the testing data. We train each of the algorithm using the designated training data preference specified above. Consistent starting points and parameters are applied to ensure fair comparison among algorithms. The performance metrics are recorded for each algorithm to assess their effectiveness. These include accuracy, execution time, precision, MAE, MSE and RMSE scores which are defined below:
\begin{itemize}
	\item[(i)] Accuracy: This measures the proportion of correctly classified instances out of the total number of instances. It is the ratio of the number of correct predictions to the total number of predictions.
	\item[(ii)] Precision: This measures the proportion of true positive predictions among all positive predictions made by the model. It is calculated as the ratio of true positive to sum of true positive and false positive. A desirable high precision score indicates that the model makes fewer false predictions.
	\item[(iii)] Mean absolute error (MAE): This is use to measures the average absolute difference between the predicted values and the actual values, with the formula given by
	\[ MAE  = \frac{1}{n} \sum_{i=1}^{n}|y_{i}- \hat{y}_{i}|\]
	where $y_{i}$ is the actual value and $\hat{y}_{i}$ is the predicted value.
	\item[(iv)] Mean square error (MSE): This is the average squared difference between the predicted values and the actual values. It penalizes larger errors more heavily that smaller errors due to squaring and it is sensitive to outliers. It's formula is given by
	\[ MSE = \frac{1}{n}\sum_{i=1}^{n}(y_{i} - \hat{y}_{i})^2.\]
	\item[(v)] Root mean square error(RMSE): This represents the average magnitude of the errors in the same units as the target variable. It provides a measure of spread of the errors. It's formula is given by
	\[ RMSE = \sqrt{\frac{1}{n}\sum_{i=1}^{n}(y_{i}- \hat{y}_{i})^2}.\]
\end{itemize}
The experimental results for dataset1 can be found in Table \ref{data1_10} - \ref{data1_40}. We also compare the MAE, MSE and RMSE scores for each algorithm for the split dataset to observed the performance of the algorithm for predicting the value of the test data. The results are shown in Figure \ref{fig3c1} - \ref{fig3c4}. Similarly, the results of the experiments for dataset2 can be found in Table \ref{data2_10} - \ref{data2_40} and Figure \ref{fig3d1} - \ref{fig3d4}.

\begin{table}
	\centering
		\caption{Experimental results of algorithms for 10\% split of dataset1.}\label{data1_10}
	\begin{tabular}{|l|c|c|c|c|c|}
		\toprule
		 & ADMM & Alg1 & Alg2 & DCA & GDCP \\
		 \midrule
	Accuracy & 0.9855 & 0.9348 & 0.9348 & 0.9348 & 0.8768 \\
	Precision & 0.9841 & 1.0 & 1.0 & 1.0 & 1.0 \\
	Time & 0.2941 & 0.2497 & 0.2508 & 0.2826 & 0.2540 \\
	MAE & 0.0290 & 0.1304 & 0.1304 & 0.1304 & 0.2464 \\
	MSE & 0.0580 & 0.2609 & 0.2609 & 0.2609 & 0.4928 \\
	RMSE & 0.2408 & 0.5108 & 0.5108 & 0.5108 & 0.7020 \\
	\bottomrule	
	\end{tabular}
\end{table}
\begin{table}
	\centering
	\caption{Experimental results of algorithms for 20\% split of dataset1.}\label{data1_20}
	\begin{tabular}{|l|c|c|c|c|c|}
		\toprule
		& ADMM & Alg1 & Alg2 & DCA & GDCP \\
		\midrule
		Accuracy & 0.9818 & 0.9345 & 0.9345 & 0.9345 & 0.8981 \\
		Precision & 0.9747 & 0.9902 & 0.9902 & 0.9902 & 0.9592 \\
		Time & 0.2889 & 0.2599 & 0.2625 & 0.2858 & 0.2772 \\
		MAE & 0.0364 & 0.1309 & 0.1309 & 0.2464 & 0.2036 \\
		MSE & 0.0727 & 0.2618 & 0.2618 & 0.4928 & 0.4073 \\
		RMSE & 0.2697 & 0.5117 & 0.5117 & 0.7020 & 0.6382 \\
		\bottomrule	
	\end{tabular}
\end{table}

\begin{table}
	\centering
	\caption{Experimental results of algorithms for 30\% split of dataset1.}\label{data1_30}
	\begin{tabular}{|l|c|c|c|c|c|}
		\toprule
		& ADMM & Alg1 & Alg2 & DCA & GDCP \\
		\midrule
		Accuracy & 0.9836 & 0.9399 & 0.9399 & 0.9399 & 0.9290 \\
		Precision & 0.9881 & 0.9874 & 0.9874 & 0.9874 & 0.9863 \\
		Time & 0.2695 & 0.2156 & 0.2208 & 0.2635 & 0.2641 \\
		MAE & 0.0328 & 0.1202 & 0.1202 & 0.1202 & 0.1421 \\
		MSE & 0.0656 & 0.2404 & 0.2404 & 0.2404 & 0.2815 \\
		RMSE & 0.2561 & 0.4903 & 0.4903 & 0.4903 & 0.5331 \\
		\bottomrule	
	\end{tabular}
\end{table}

\begin{table}
	\centering
	\caption{Experimental results of algorithms for 40\% split of dataset1.}\label{data1_40}
	\begin{tabular}{|l|c|c|c|c|c|}
		\toprule
		& ADMM & Alg1 & Alg2 & DCA & GDCP \\
		\midrule
		Accuracy & 0.9878 & 0.9466 & 0.9466 & 0.9466 & 0.9150 \\
		Precision & 0.9752 & 0.9768 & 0.9768 & 0.9768 & 0.9808 \\
		Time & 0.5294 & 0.3004 & 0.3068 & 0.4760 & 0.3936 \\
		MAE & 0.0583 & 0.1068 & 0.1608 & 0.2036 & 0.1699 \\
		MSE & 0.1165 & 0.2136 & 0.2136 & 0.4073 & 0.3399 \\
		RMSE & 0.3413 & 0.4621 & 0.4621 & 0.6382 & 0.5829 \\
		\bottomrule	
	\end{tabular}
\end{table}

\begin{figure}[h!]
	\begin{subfigure}{1.1\textwidth}
		\centering
		\includegraphics[width=.3\linewidth]{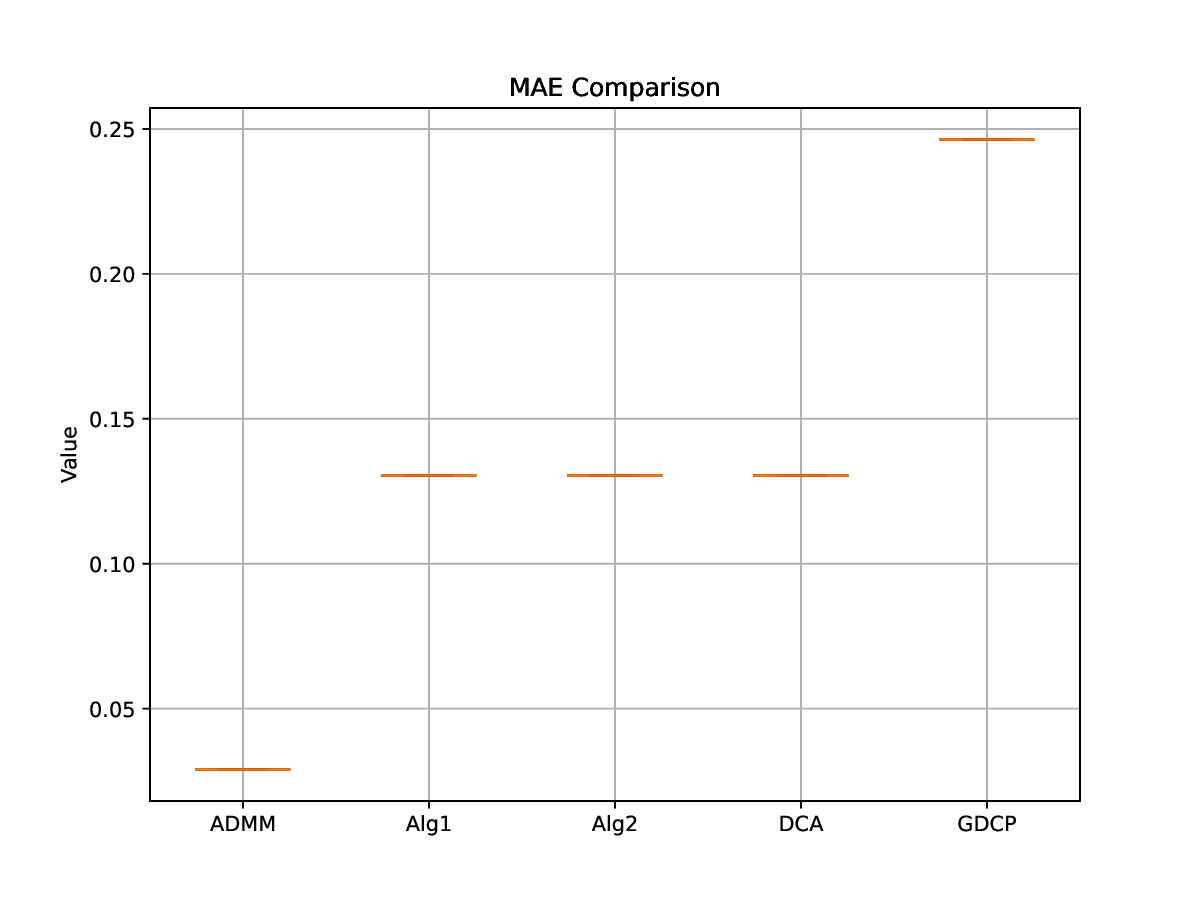}
		\includegraphics[width=.3\linewidth]{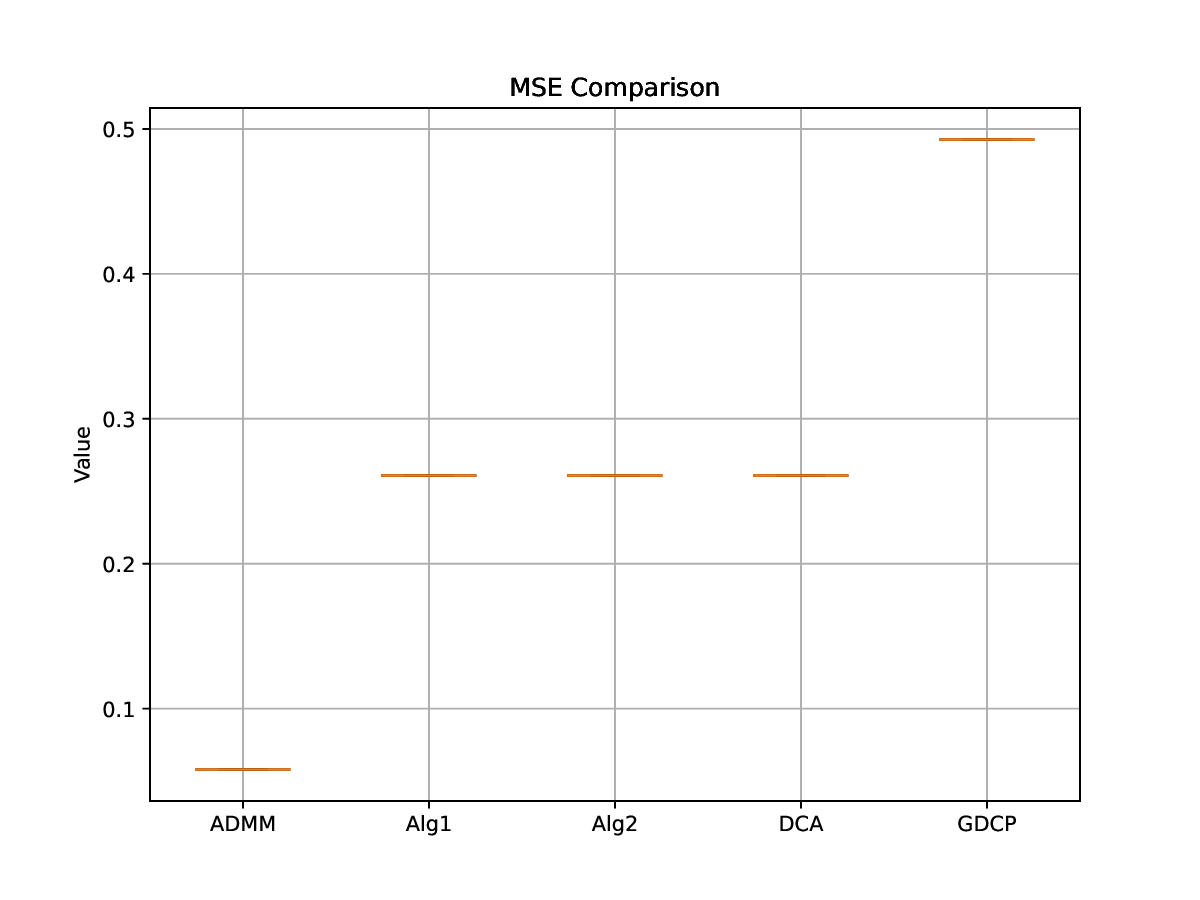}
		\includegraphics[width=.3\linewidth]{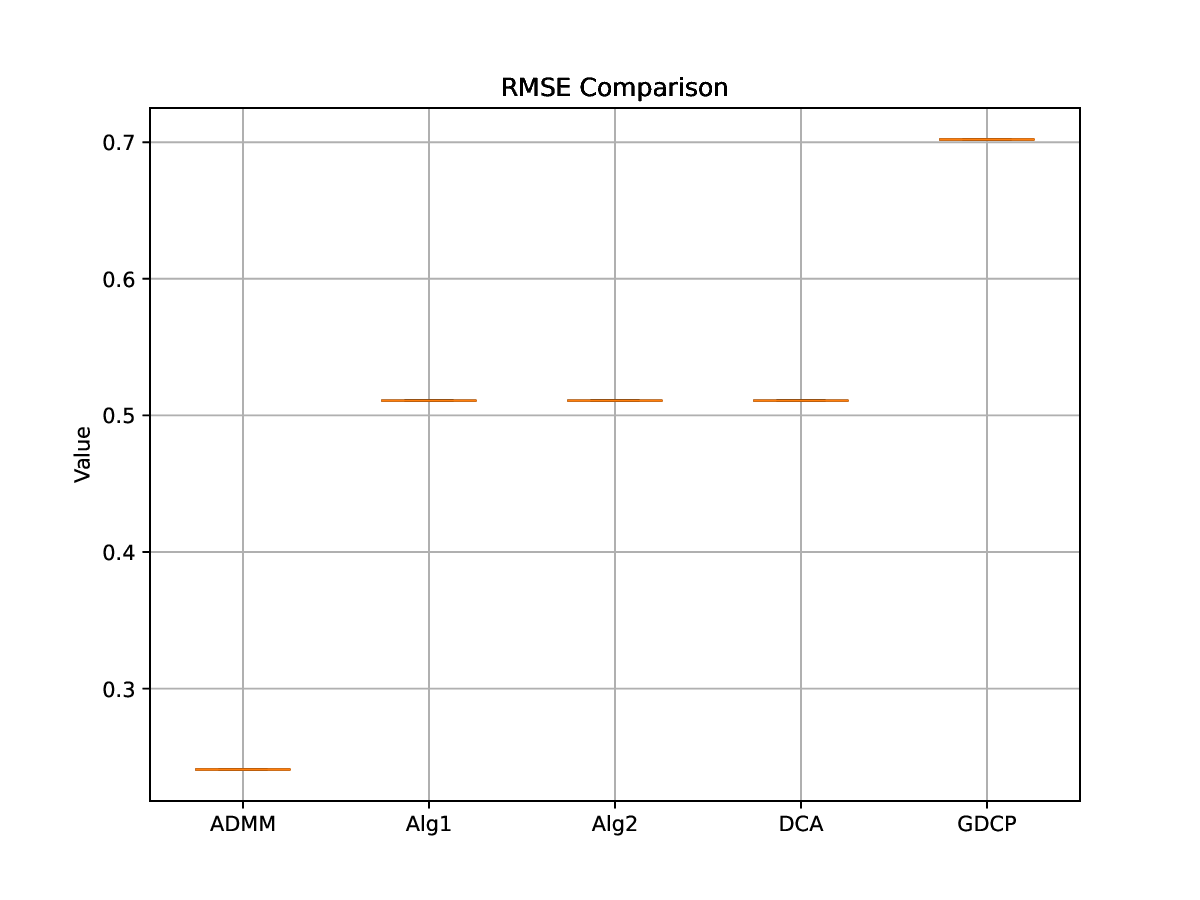}
	\end{subfigure}
	\caption{Comparison of MAE, MSE and RMSE for 10\% split of dataset1.} \label{fig3c1}
\end{figure}

\begin{figure}[h!]
	\begin{subfigure}{1.1\textwidth}
		\centering
		\includegraphics[width=.3\linewidth]{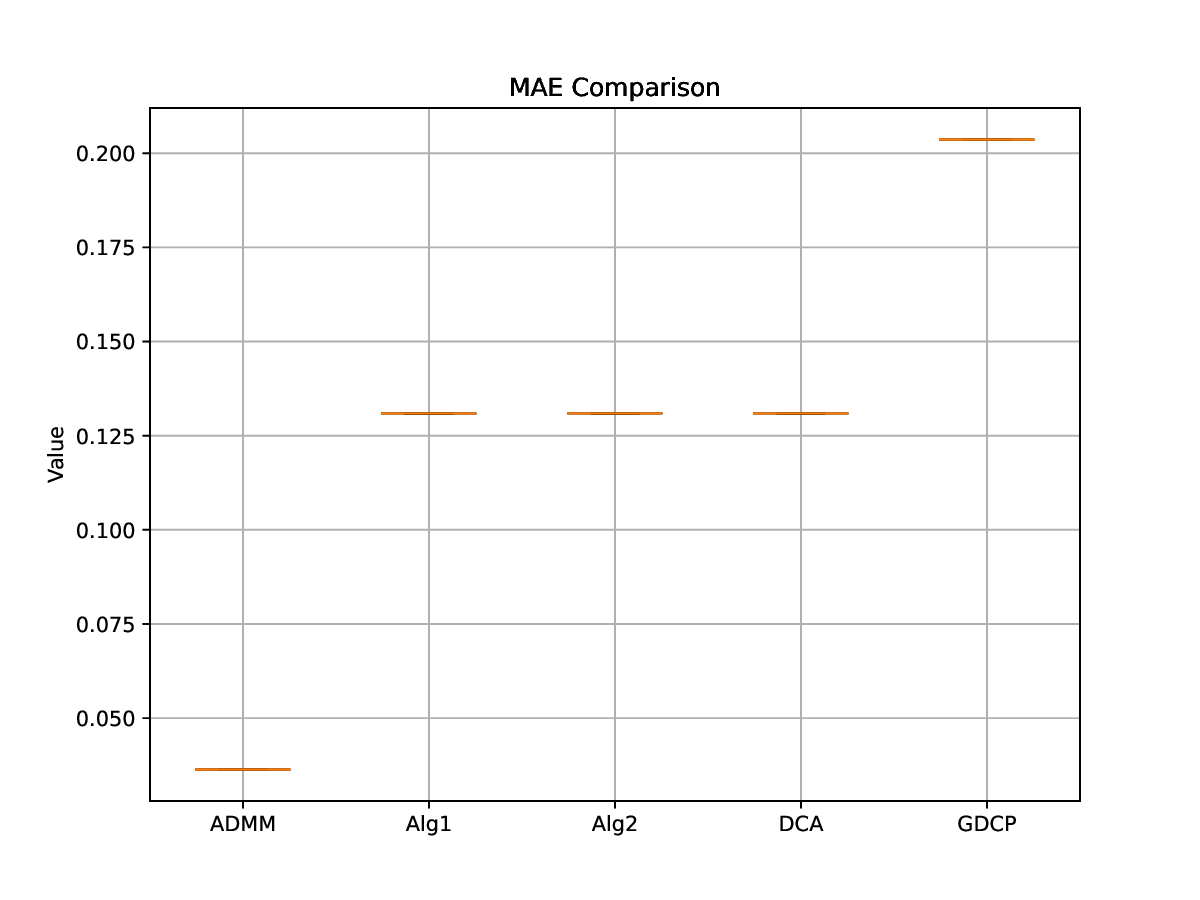}
		\includegraphics[width=.3\linewidth]{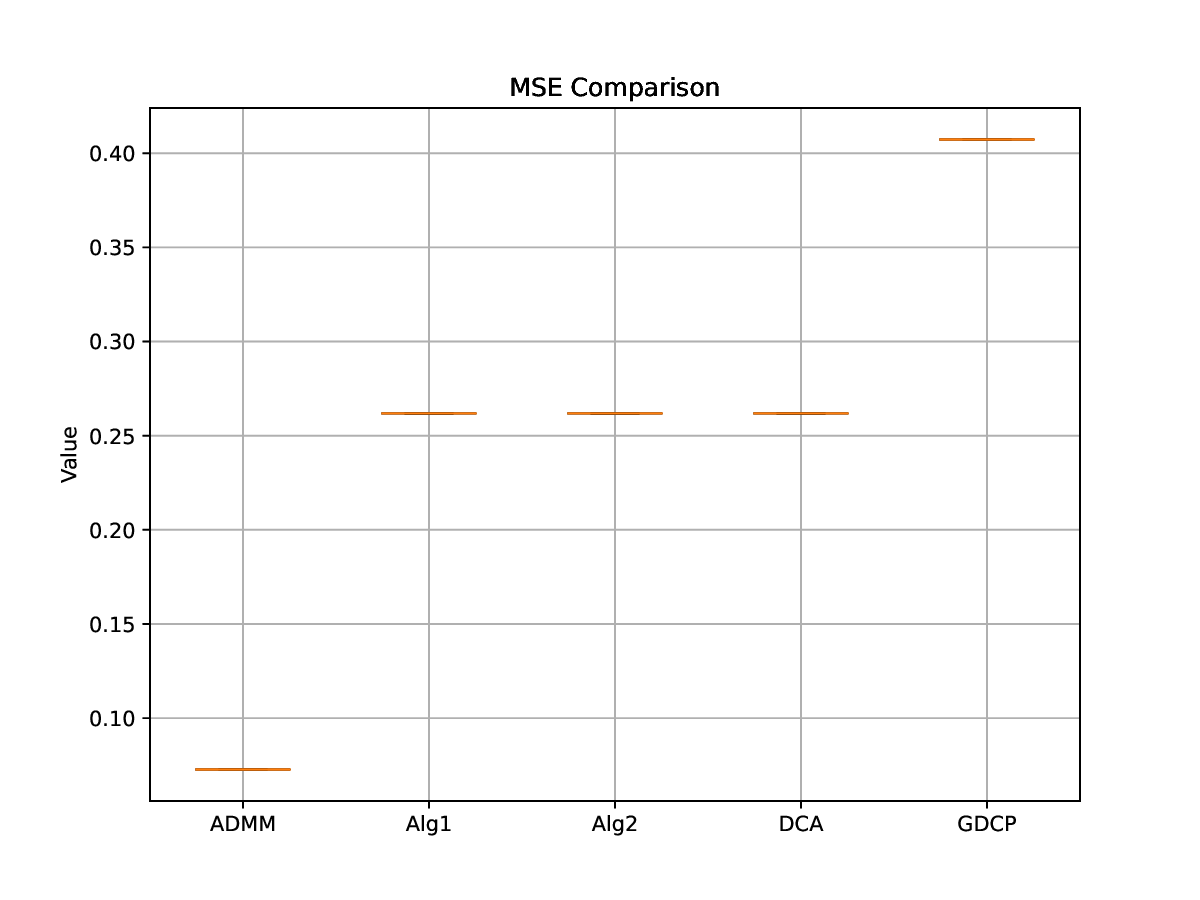}
		\includegraphics[width=.3\linewidth]{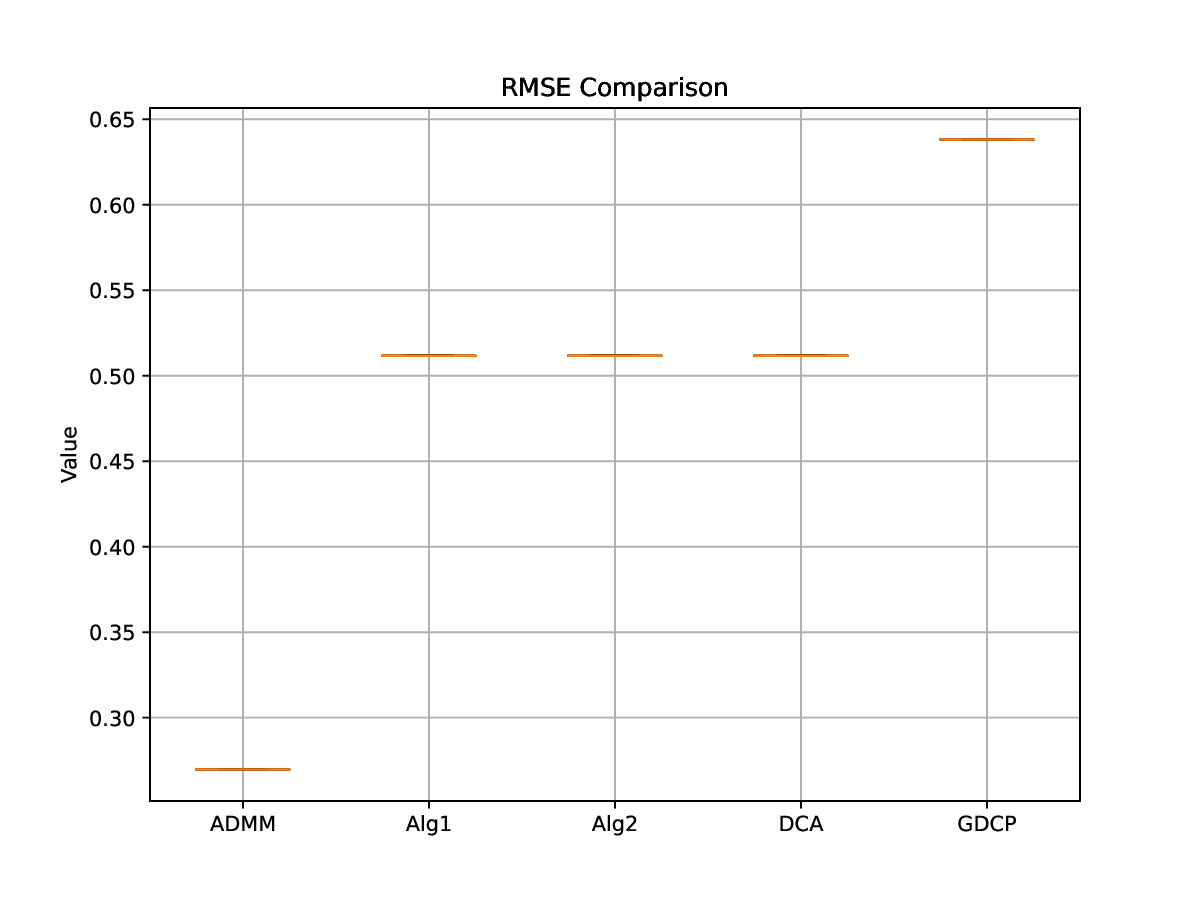}
	\end{subfigure}
	\caption{Comparison of MAE, MSE and RMSE for 20\% split of dataset1.} \label{fig3c2}
\end{figure}

\begin{figure}[h!]
	\begin{subfigure}{1.1\textwidth}
		\centering
		\includegraphics[width=.3\linewidth]{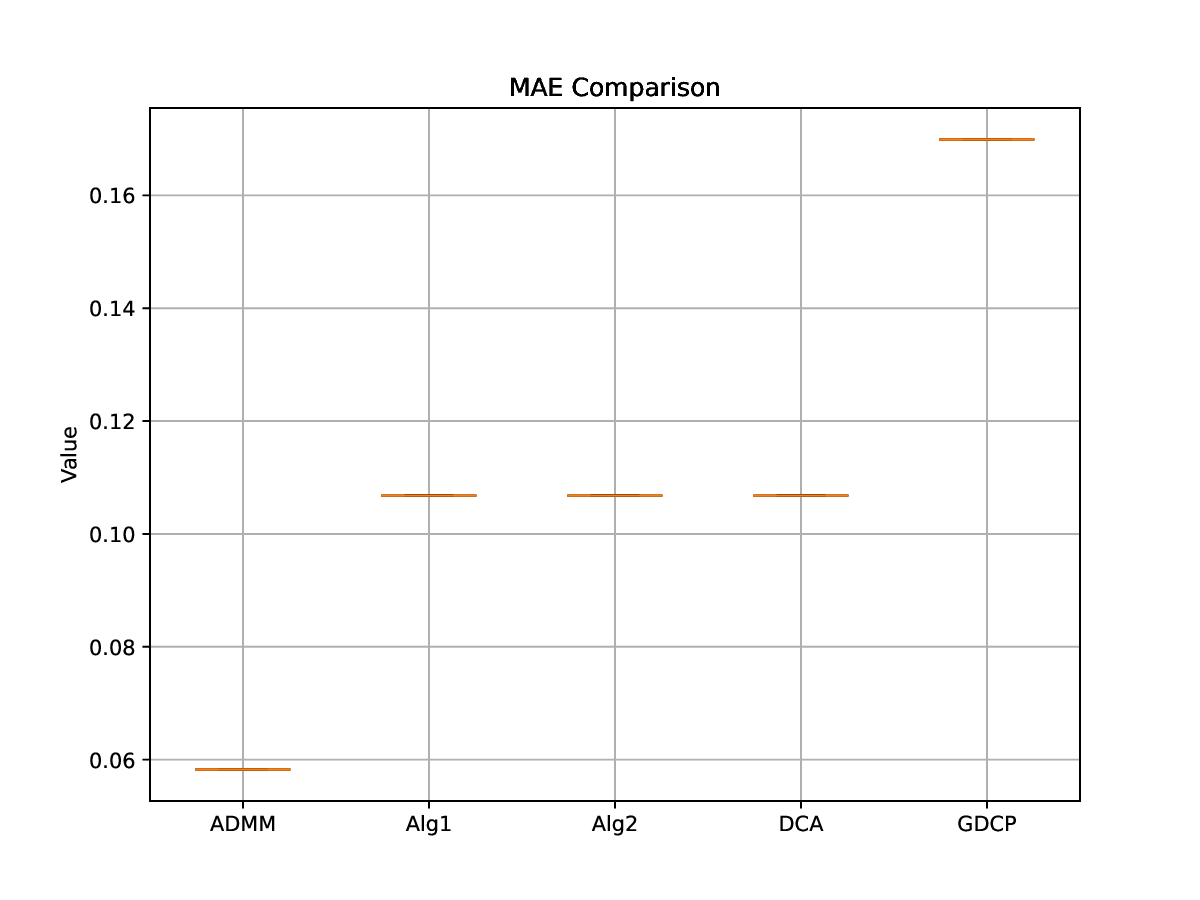}
		\includegraphics[width=.3\linewidth]{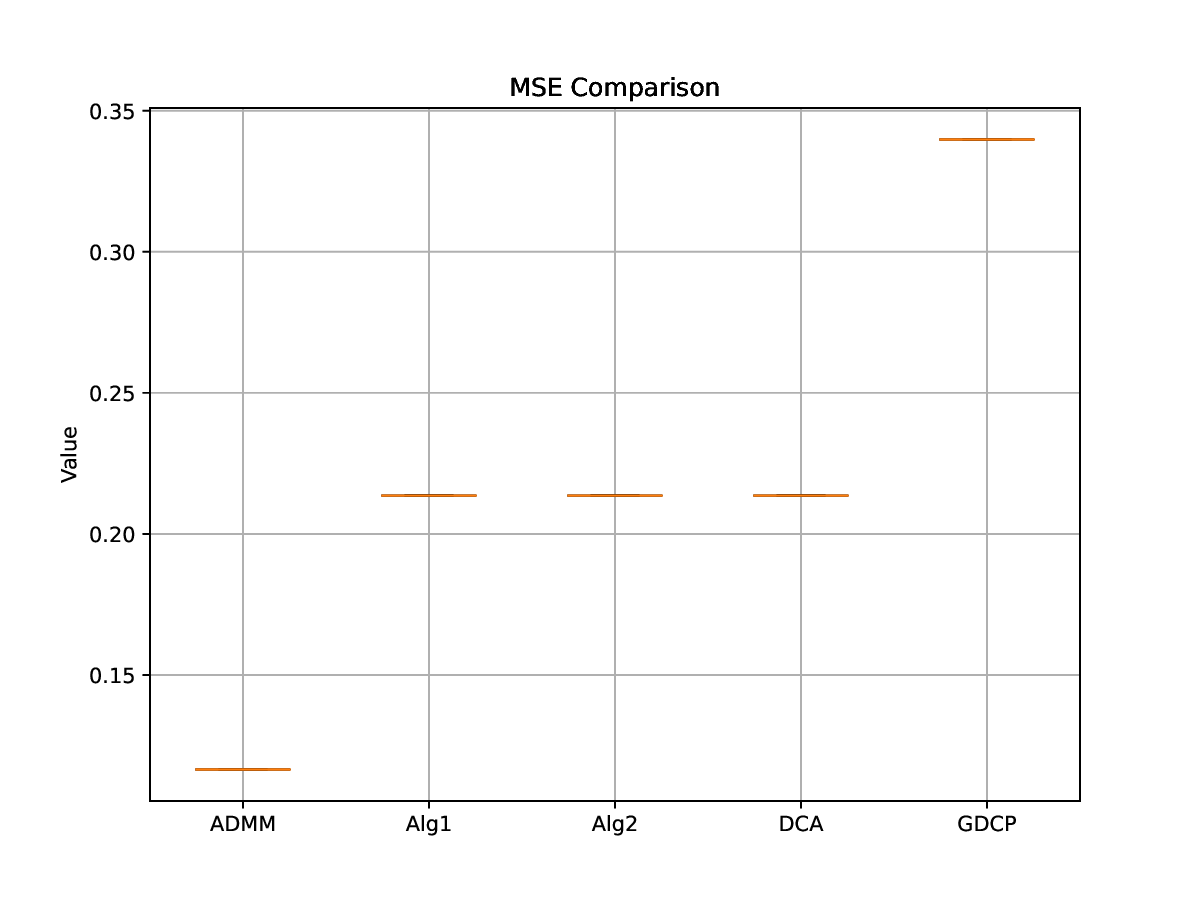}
		\includegraphics[width=.3\linewidth]{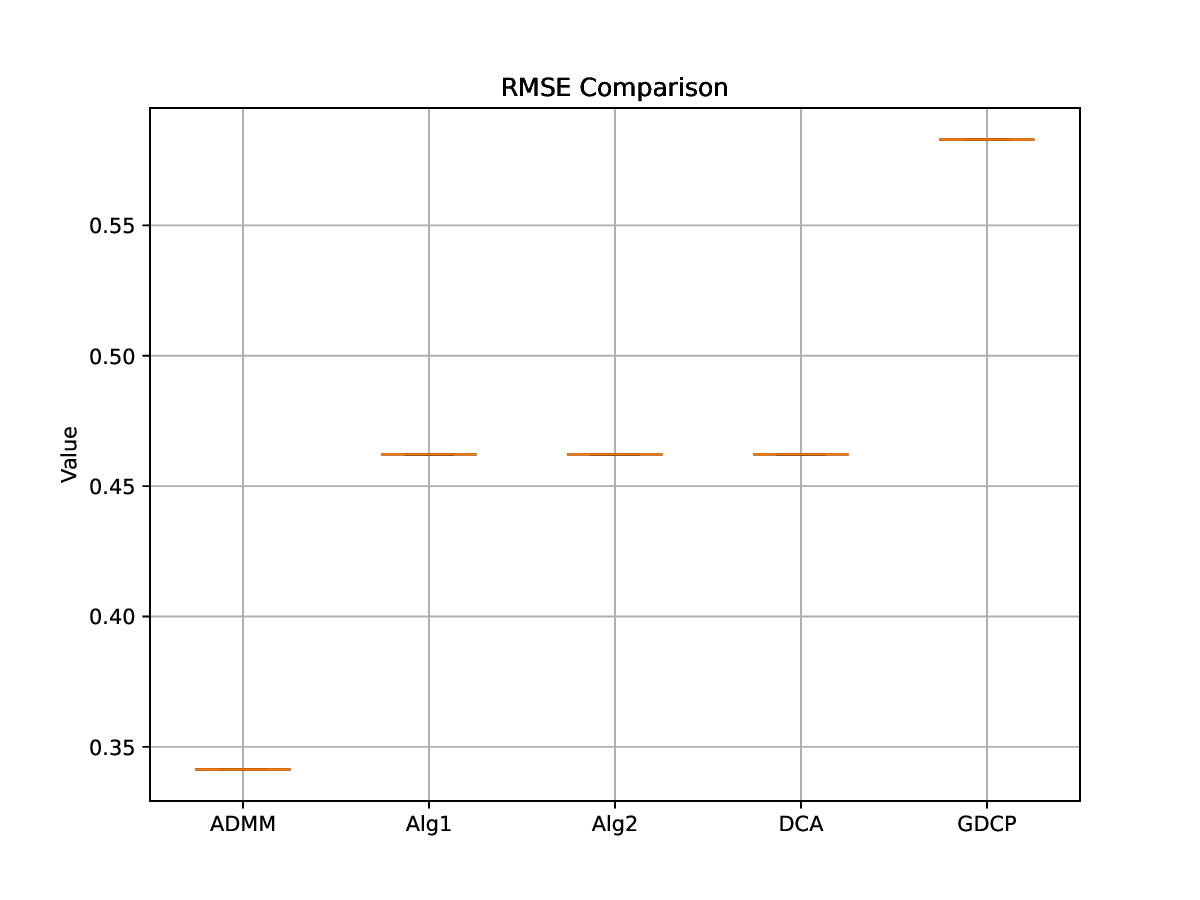}
	\end{subfigure}
	\caption{Comparison of MAE, MSE and RMSE for 30\% split of dataset1.} \label{fig3c3}
\end{figure}

\begin{figure}[h!]
	\begin{subfigure}{1.1\textwidth}
		\centering
		\includegraphics[width=.3\linewidth]{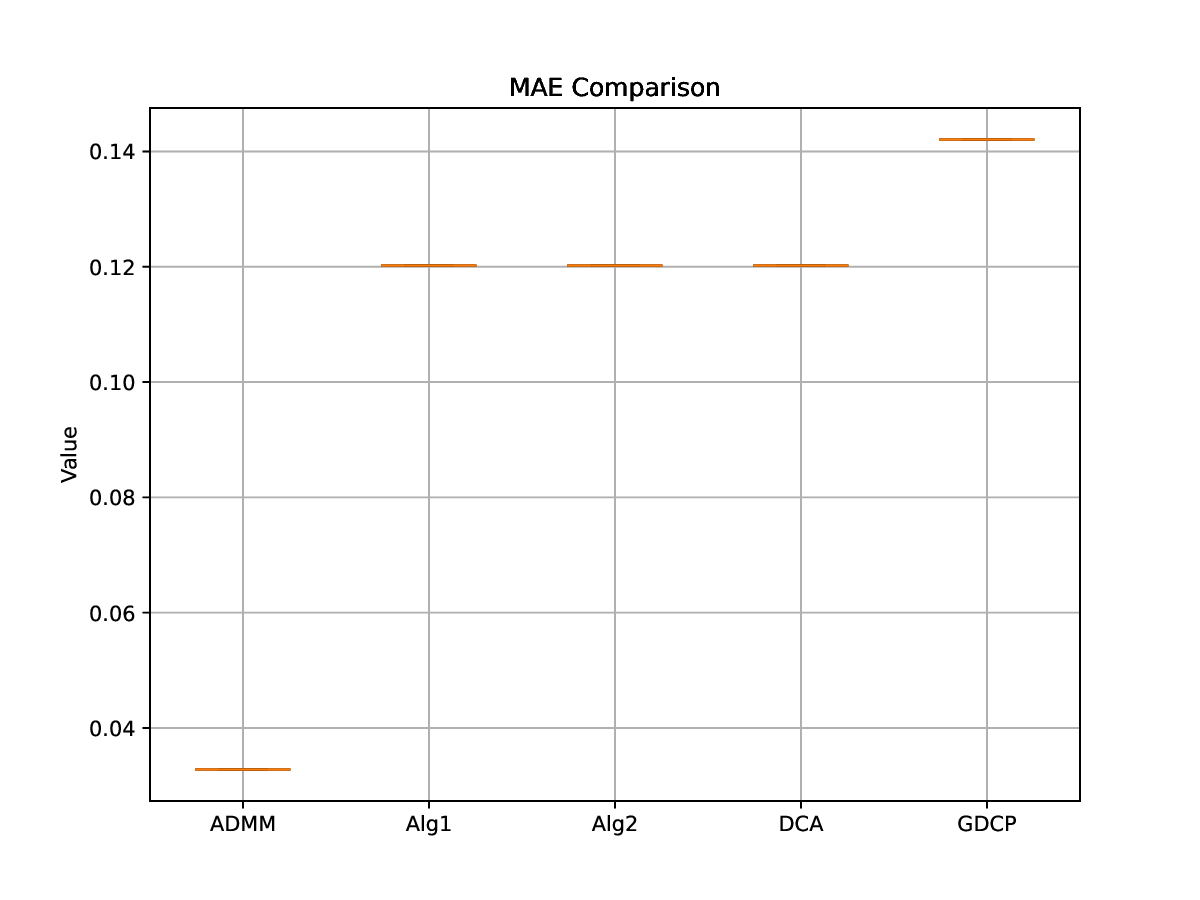}
		\includegraphics[width=.3\linewidth]{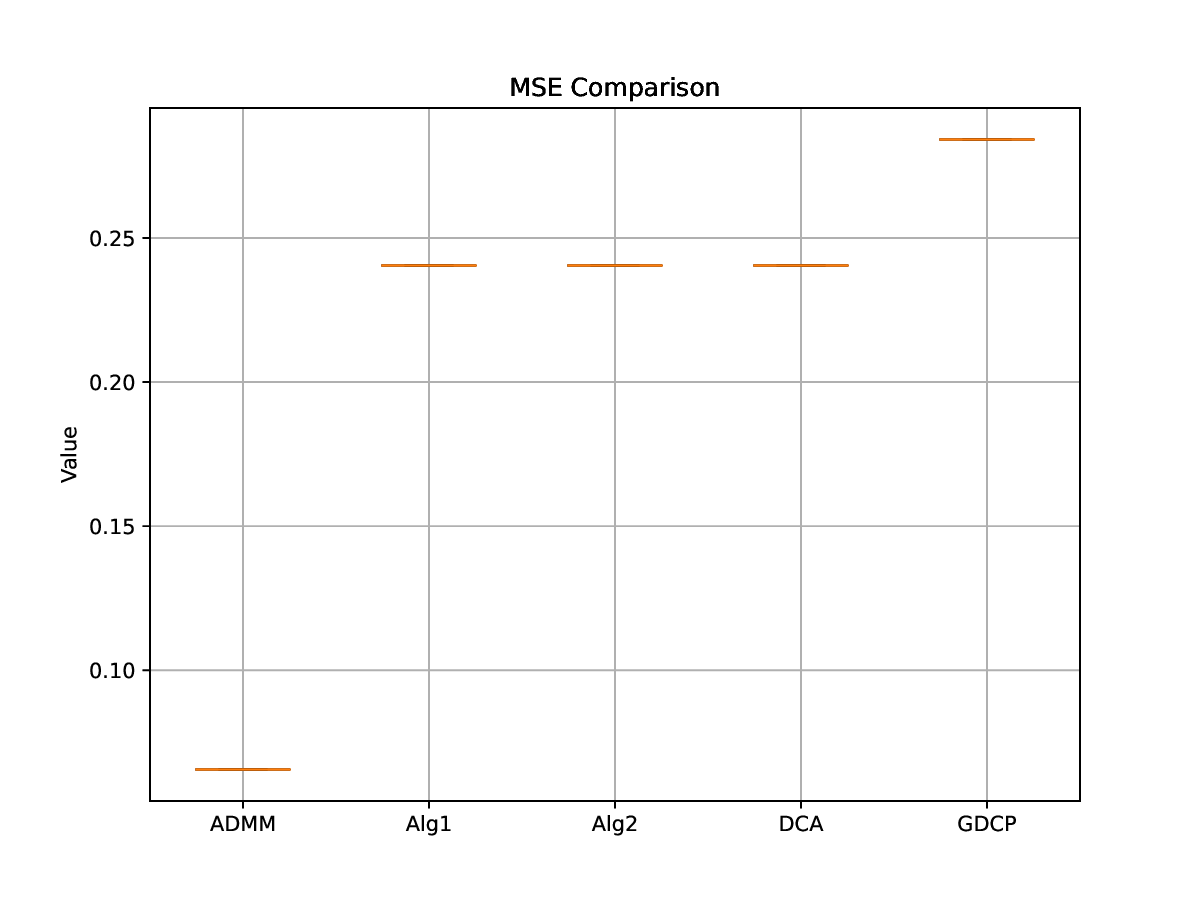}
		\includegraphics[width=.3\linewidth]{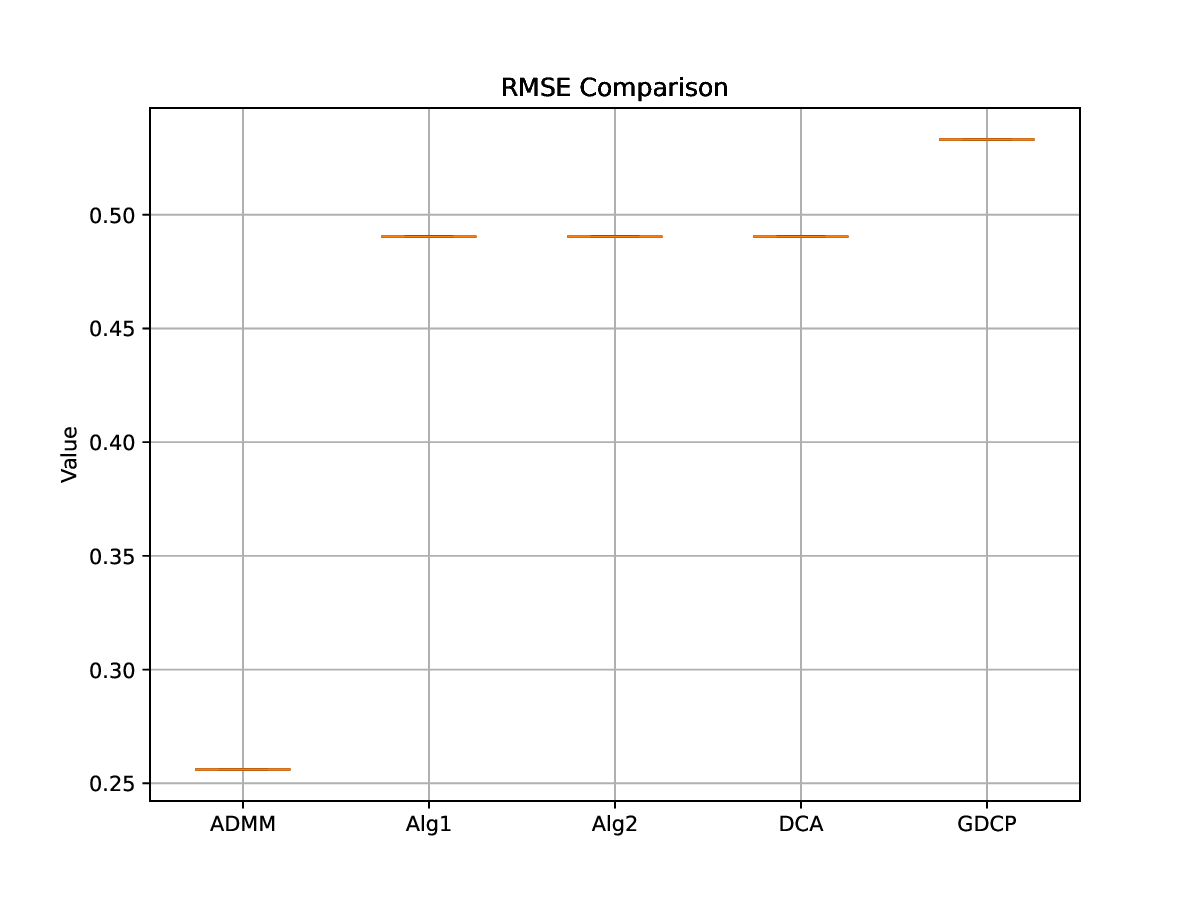}
	\end{subfigure}
	\caption{Comparison of MAE, MSE and RMSE for 40\% split of dataset1.} \label{fig3c4}
\end{figure}

\begin{table}
	\centering
	\caption{Experimental results of algorithms for 10\% split of dataset2.}\label{data2_10}
	\begin{tabular}{|l|c|c|c|c|c|}
		\toprule
		& ADMM & Alg1 & Alg2 & DCA & GDCP \\
		\midrule
		Accuracy & 0.9368 & 0.8316 & 0.9050 & 0.9052 & 0.7473 \\
		Precision & 0.9512 & 0.7454 & 0.8696 & 0.8696 & 0.6727 \\
		Time & 2.7984 & 2.5741 & 2.2906 & 2.6751 & 2.7383 \\
		MAE & 0.1263 & 0.3368 & 0.1895 & 0.1895 & 0.5052 \\
		MSE & 0.2526 & 0.6737 & 0.3789 & 0.3789 & 1.0105 \\
		RMSE & 0.5026 & 0.8206 & 0.6156 & 0.6156 & 1.0052 \\
		\bottomrule	
	\end{tabular}
\end{table}
\begin{table}
	\centering
	\caption{Experimental results of algorithms for 20\% split of dataset2.}\label{data2_20}
	\begin{tabular}{|l|c|c|c|c|c|}
		\toprule
		& ADMM & Alg1 & Alg2 & DCA & GDCP \\
		\midrule
		Accuracy & 0.9316 & 0.9211 & 0.8684 & 0.8474 & 0.8632 \\
		Precision & 0.9468 & 0.9271 & 0.8158 & 0.7931 & 0.8142 \\
		Time & 2.8183 & 2.5693 & 2.3673 & 2.4144 & 2.7594 \\
		MAE & 0.1368 & 0.1579 & 0.2632 & 0.3053 & 0.2737 \\
		MSE & 0.2737 & 0.3158 & 0.5263 & 0.6105 & 0.5474 \\
		RMSE & 0.5231 & 0.5620 & 0.7255 & 0.7813 & 0.7398 \\
		\bottomrule	
	\end{tabular}
\end{table}

\begin{table}
	\centering
	\caption{Experimental results of algorithms for 30\% split of dataset2.}\label{data2_30}
	\begin{tabular}{|l|c|c|c|c|c|}
		\toprule
		& ADMM & Alg1 & Alg2 & DCA & GDCP \\
		\midrule
		Accuracy & 0.9087 & 0.9018 & 0.9298 & 0.8596 & 0.7895 \\
		Precision & 0.8868 & 0.9116 & 0.9640 & 0.8471 & 0.9298 \\
		Time & 2.8389 & 2.0428 & 2.3072 & 2.3861 & 2.5579 \\
		MAE & 0.1824 & 0.1965 & 0.1404 & 0.2807 & 0.4211 \\
		MSE & 0.3649 & 0.3930 & 0.2807 & 0.5614 & 0.8421 \\
		RMSE & 0.6040 & 0.6269 & 0.5298 & 0.7493 & 0.9177 \\
		\bottomrule	
	\end{tabular}
\end{table}

\begin{table}
	\centering
	\caption{Experimental results of algorithms for 40\% split of dataset2.}\label{data2_40}
	\begin{tabular}{|l|c|c|c|c|c|}
		\toprule
		& ADMM & Alg1 & Alg2 & DCA & GDCP \\
		\midrule
		Accuracy & 0.9076 & 0.8179 & 0.8285 & 0.8258 & 0.8073 \\
		Precision & 0.9508 & 0.7963 & 0.8082 & 0.8131 & 0.7873 \\
		Time & 2.9516 & 2.5901 & 2.5424 & 2.6523 & 2.9498 \\
		MAE & 0.1847 & 0.3641 & 0.3430 & 0.3482 & 0.3852 \\
		MSE & 0.3694 & 0.7282 & 0.6861 & 0.6964 & 0.7704 \\
		RMSE & 0.6078 & 0.8533 & 0.8282 & 0.8346 & 0.88775 \\
		\bottomrule	
	\end{tabular}
\end{table}

\begin{figure}[h!]
	\begin{subfigure}{1.1\textwidth}
		\centering
		\includegraphics[width=.3\linewidth]{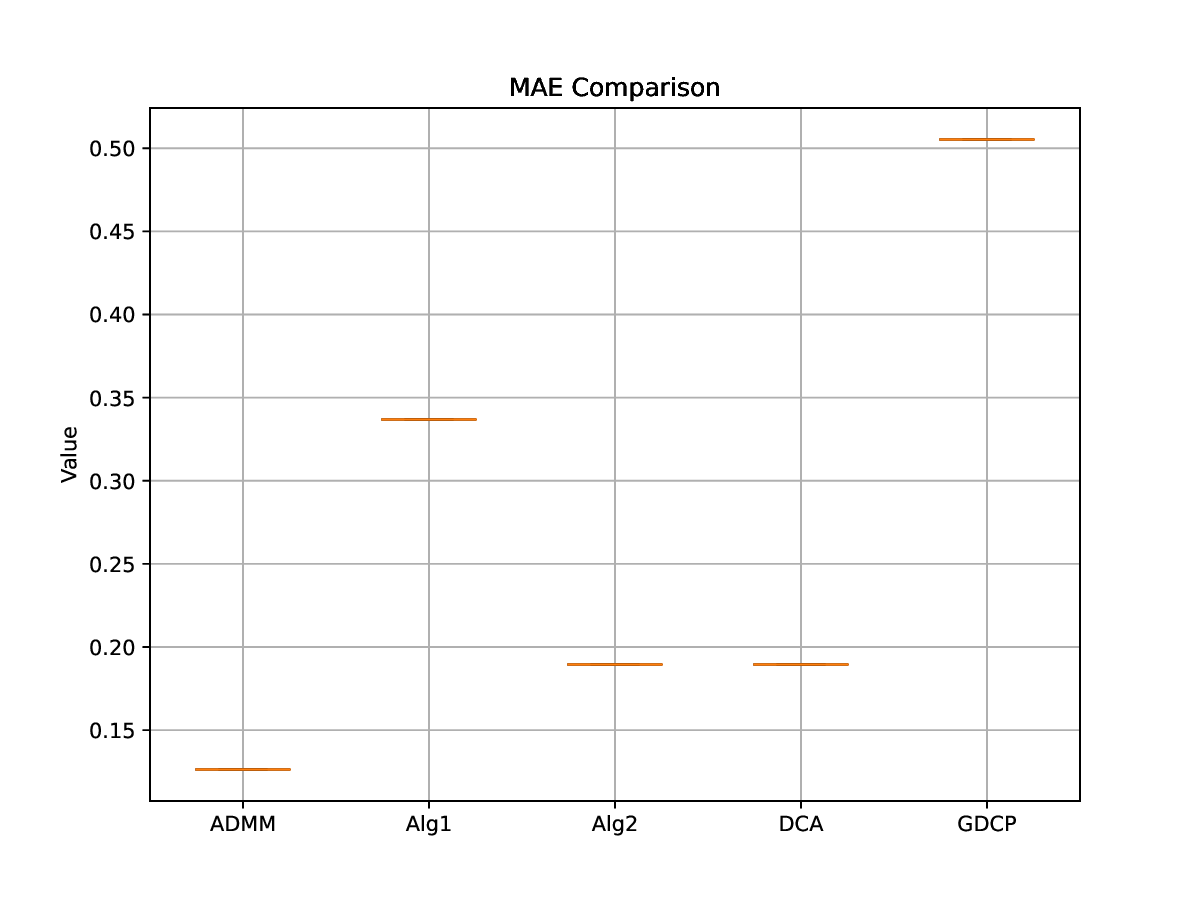}
		\includegraphics[width=.3\linewidth]{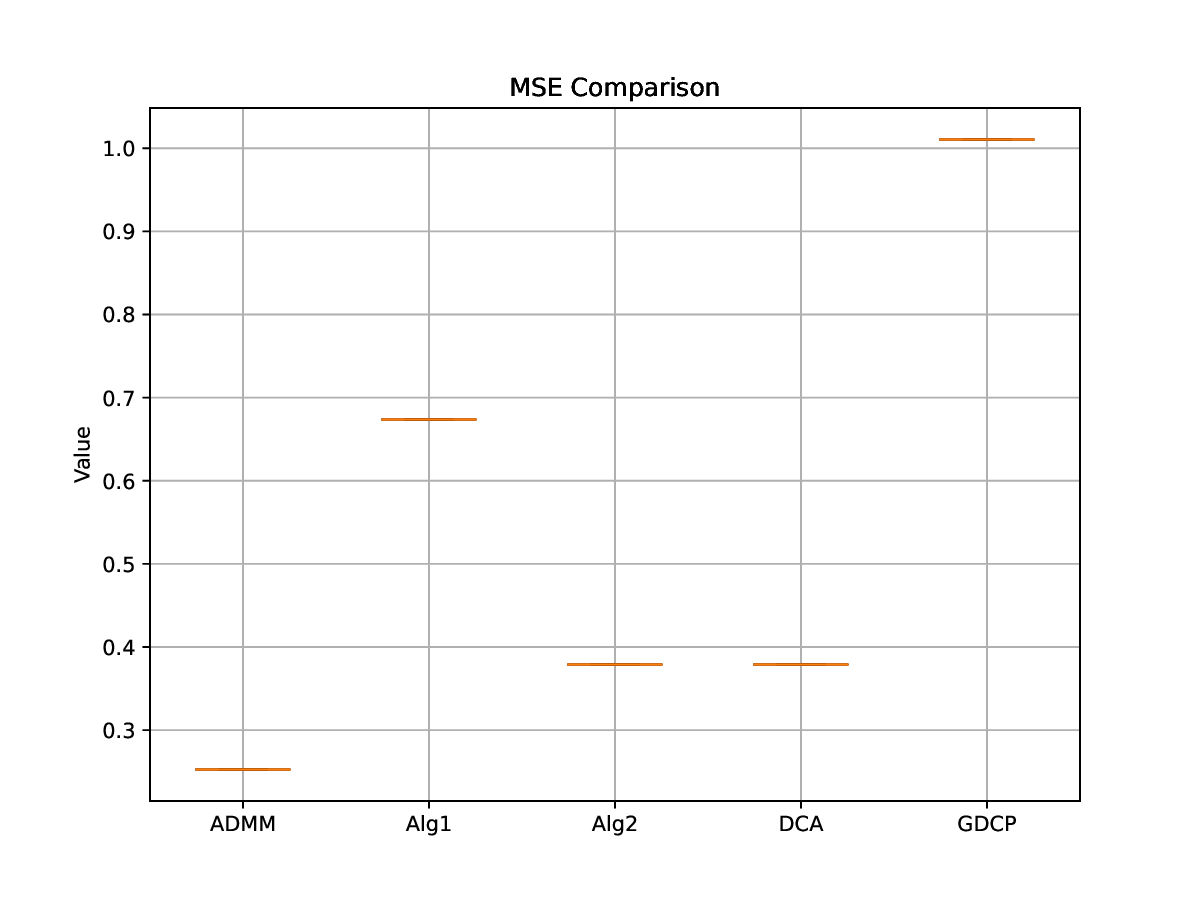}
		\includegraphics[width=.3\linewidth]{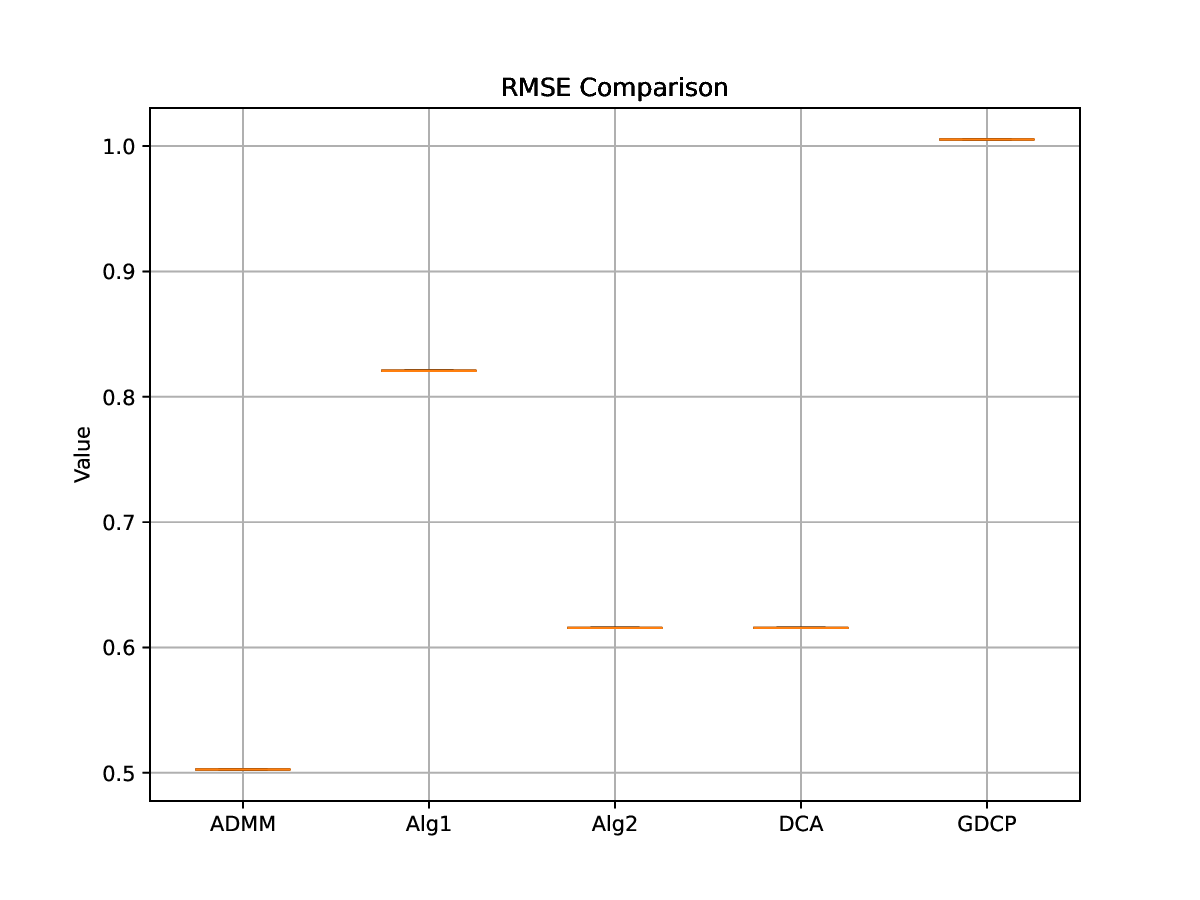}
	\end{subfigure}
	\caption{Comparison of MAE, MSE and RMSE for 10\% split of dataset2} \label{fig3d1}
\end{figure}

\begin{figure}[h!]
	\begin{subfigure}{1.1\textwidth}
		\centering
		\includegraphics[width=.3\linewidth]{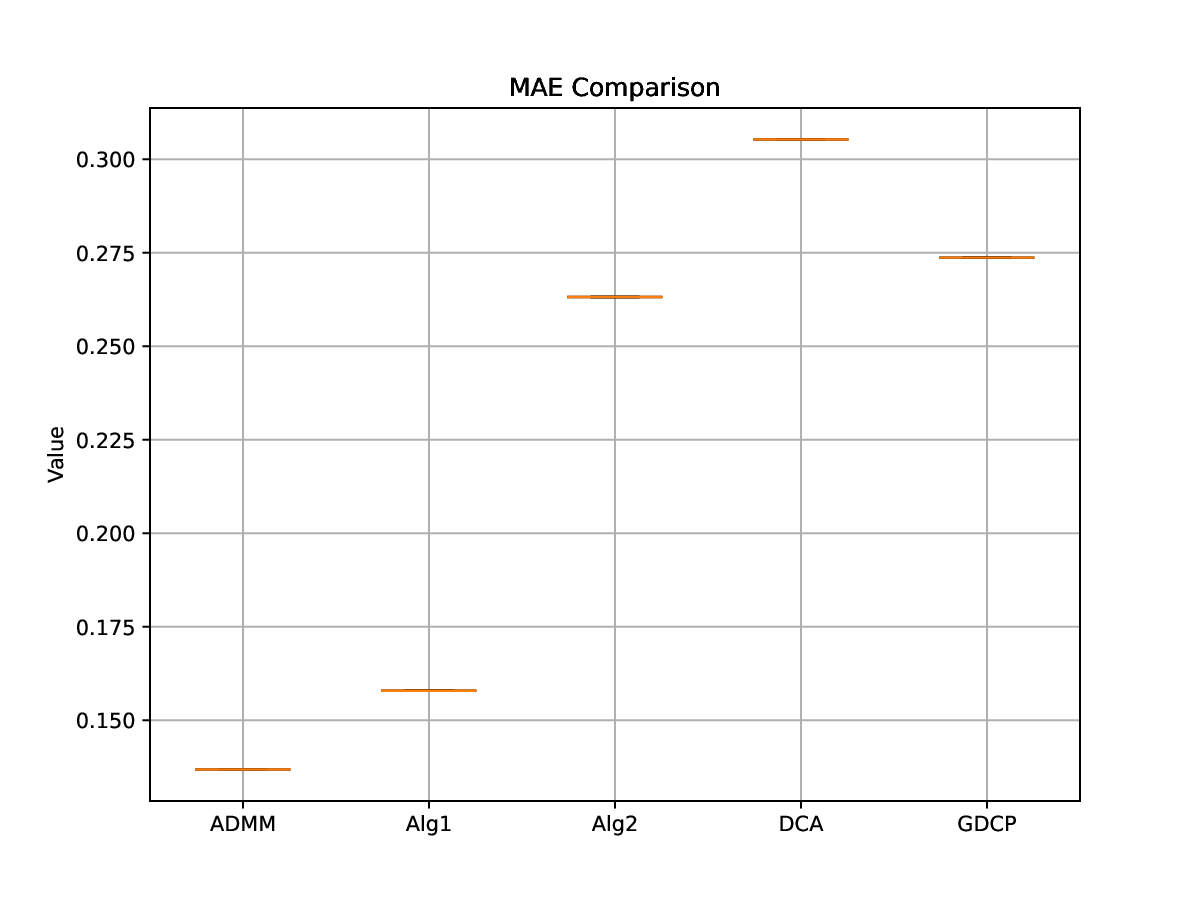}
		\includegraphics[width=.3\linewidth]{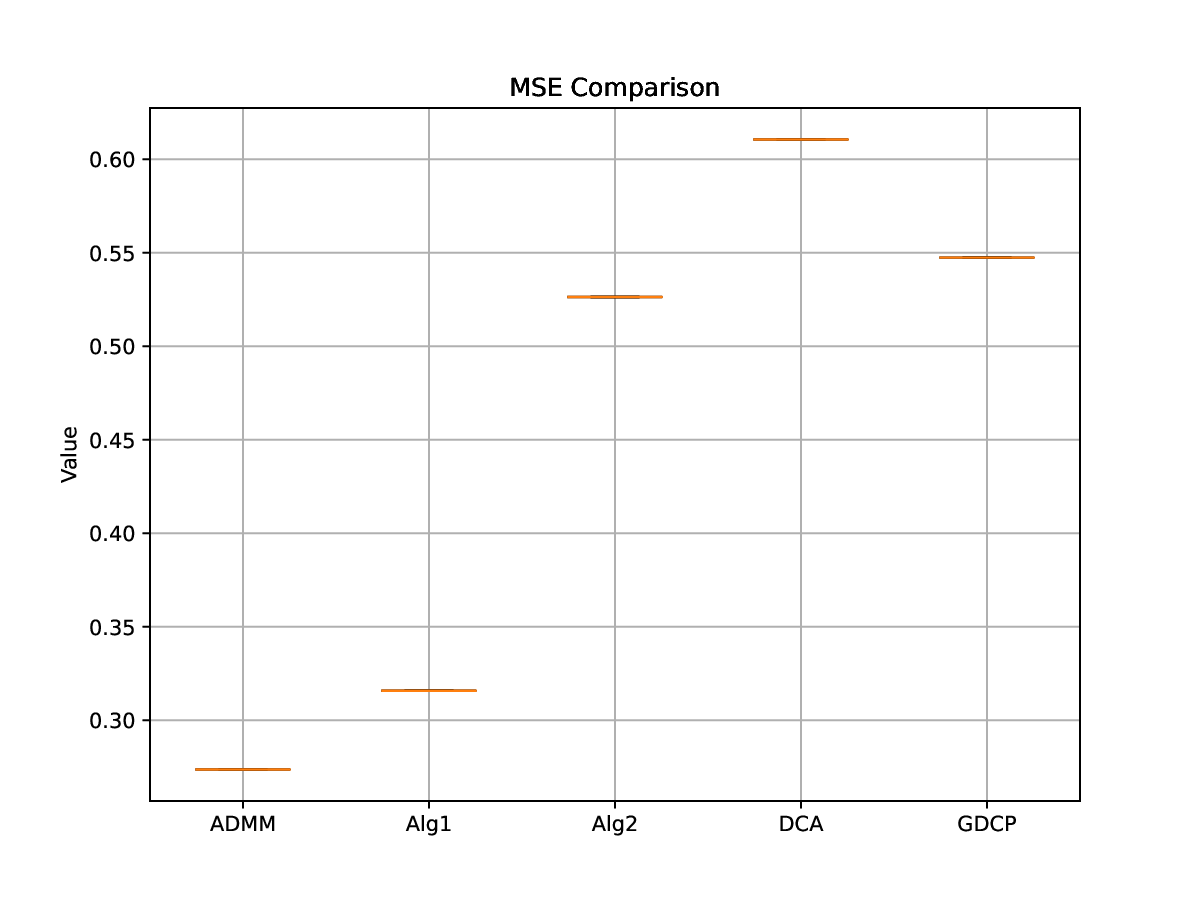}
		\includegraphics[width=.3\linewidth]{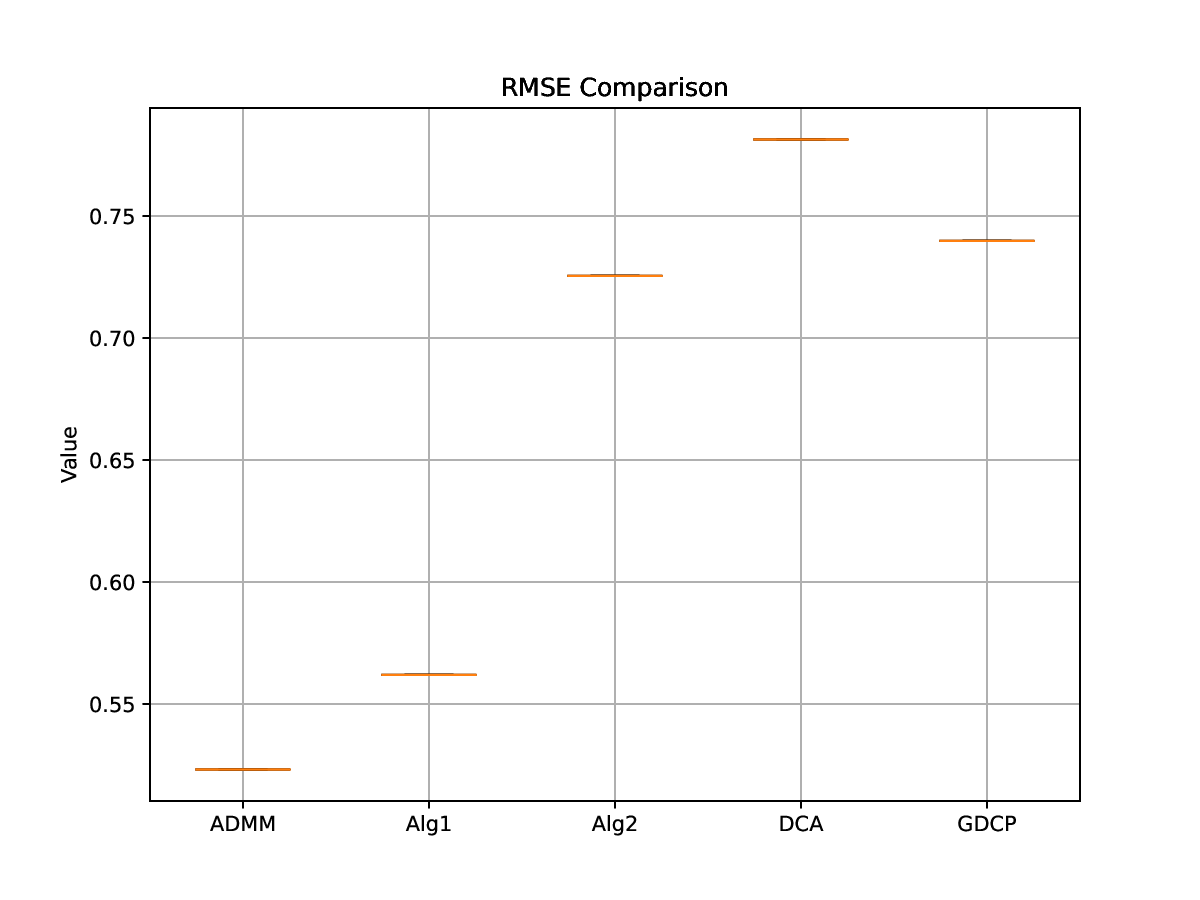}
	\end{subfigure}
	\caption{Comparison of MAE, MSE and RMSE for 20\% split of dataset2} \label{fig3d2}
\end{figure}

\begin{figure}[h!]
	\begin{subfigure}{1.1\textwidth}
		\centering
		\includegraphics[width=.3\linewidth]{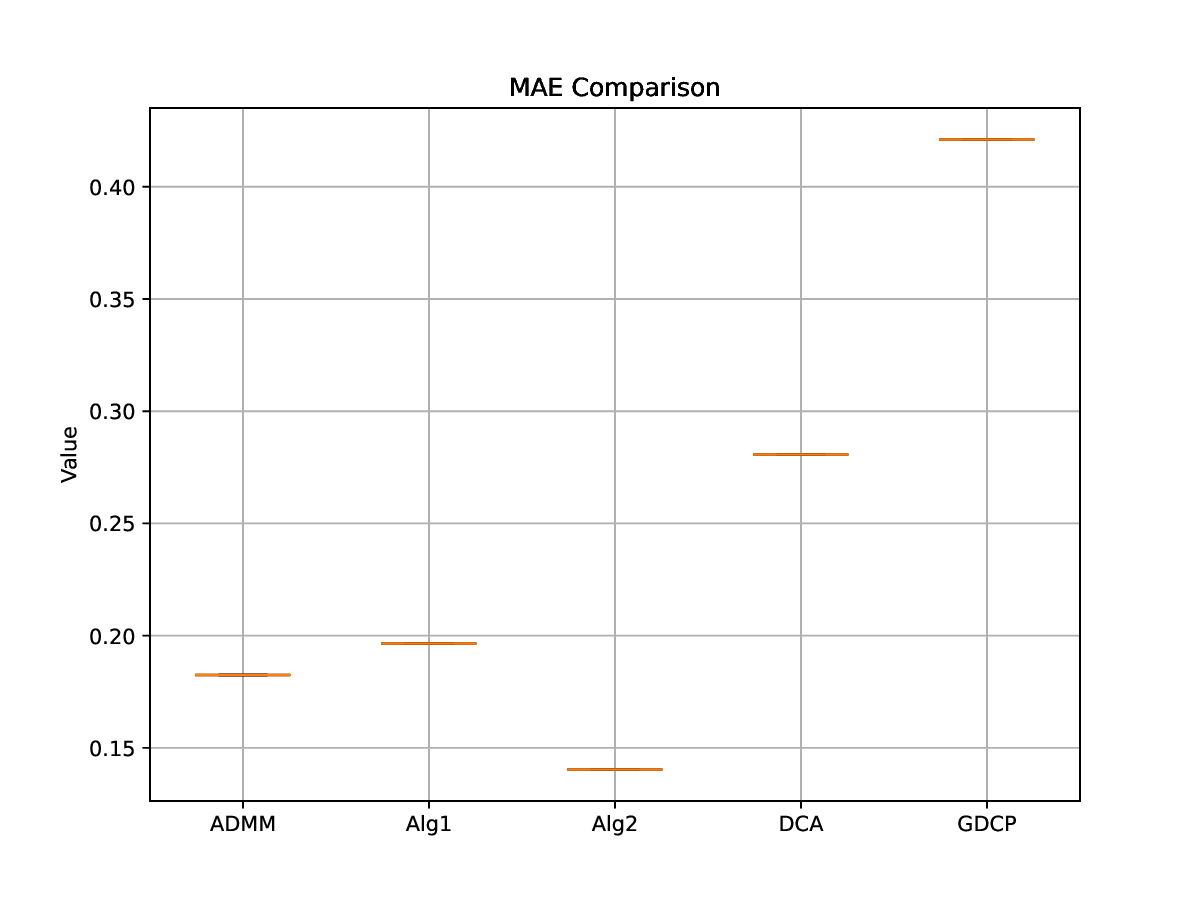}
		\includegraphics[width=.3\linewidth]{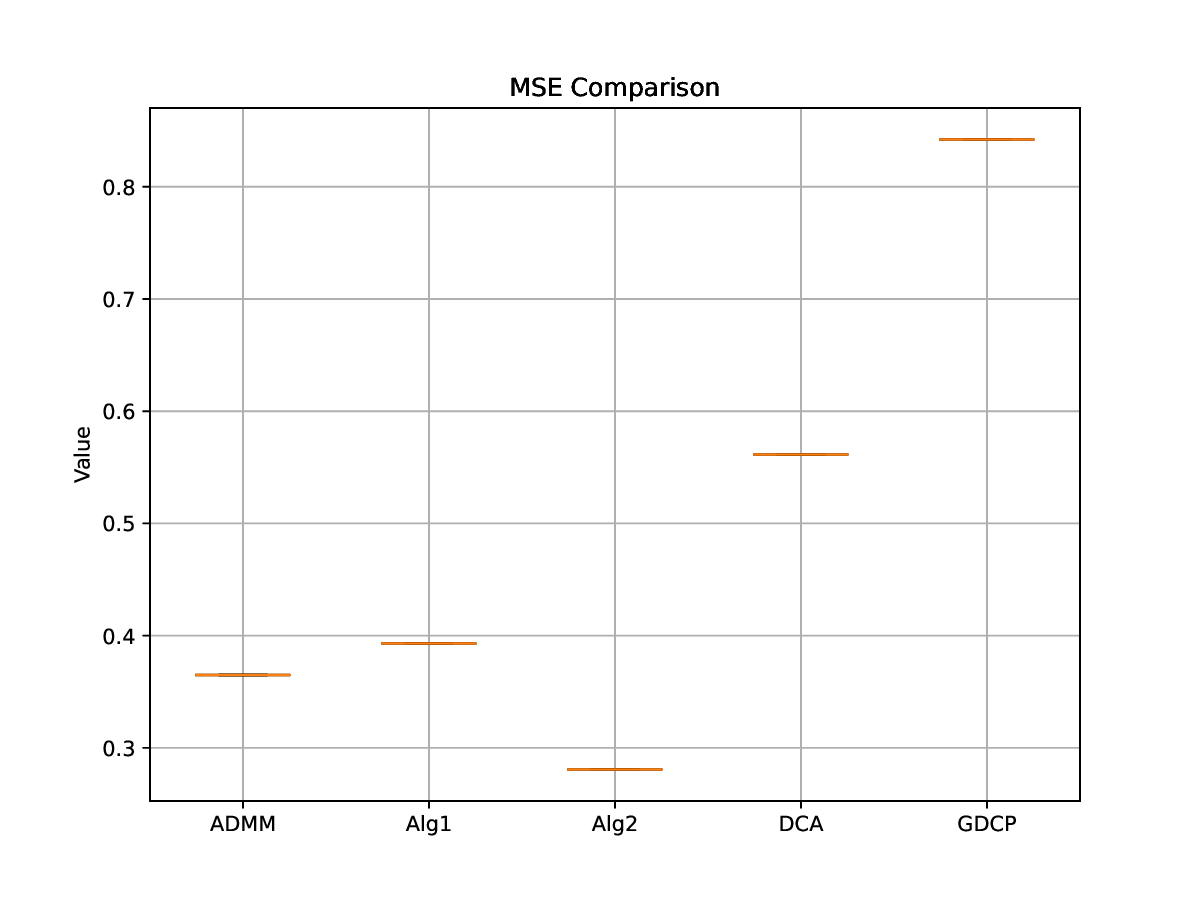}
		\includegraphics[width=.3\linewidth]{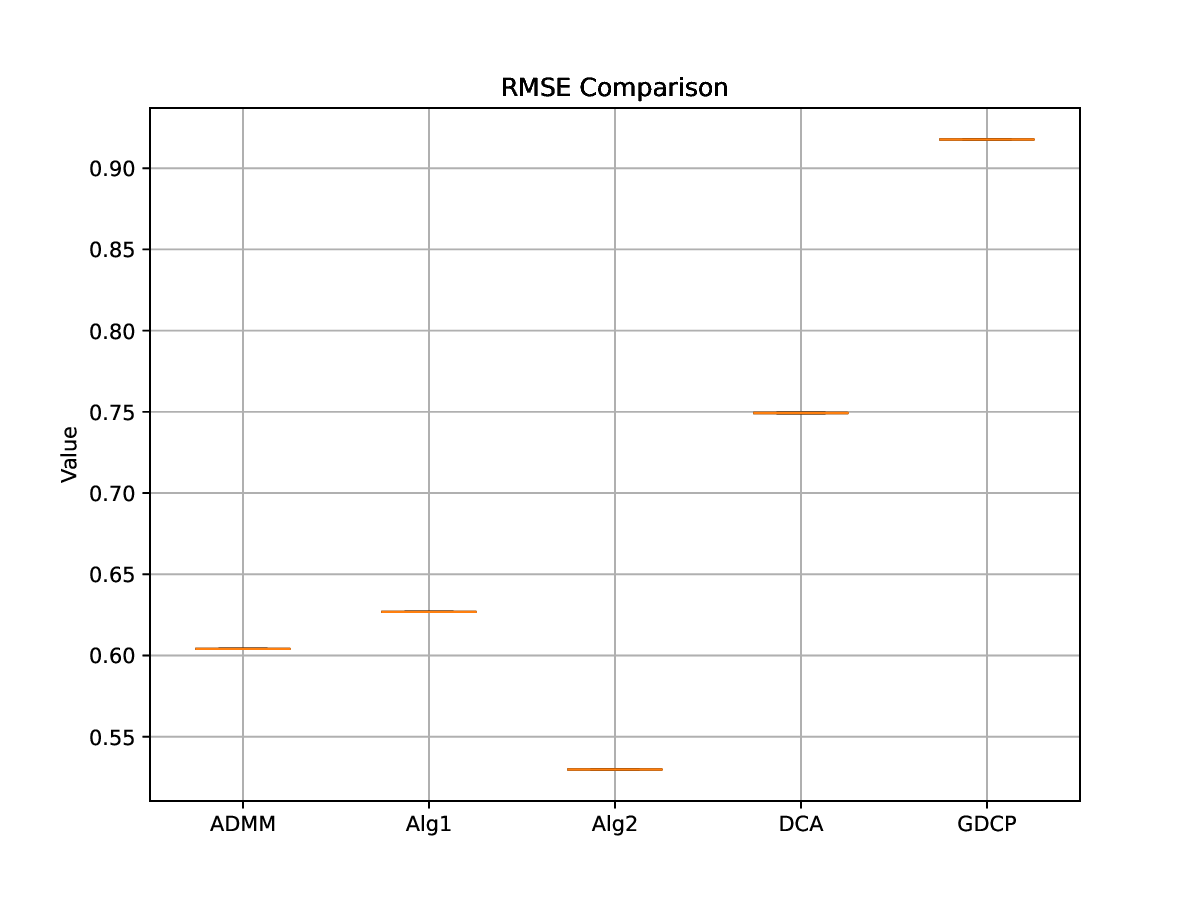}
	\end{subfigure}
	\caption{Comparison of MAE, MSE and RMSE for 30\% split of dataset2} \label{fig3d3}
\end{figure}

\begin{figure}[h!]
	\begin{subfigure}{1.1\textwidth}
		\centering
		\includegraphics[width=.3\linewidth]{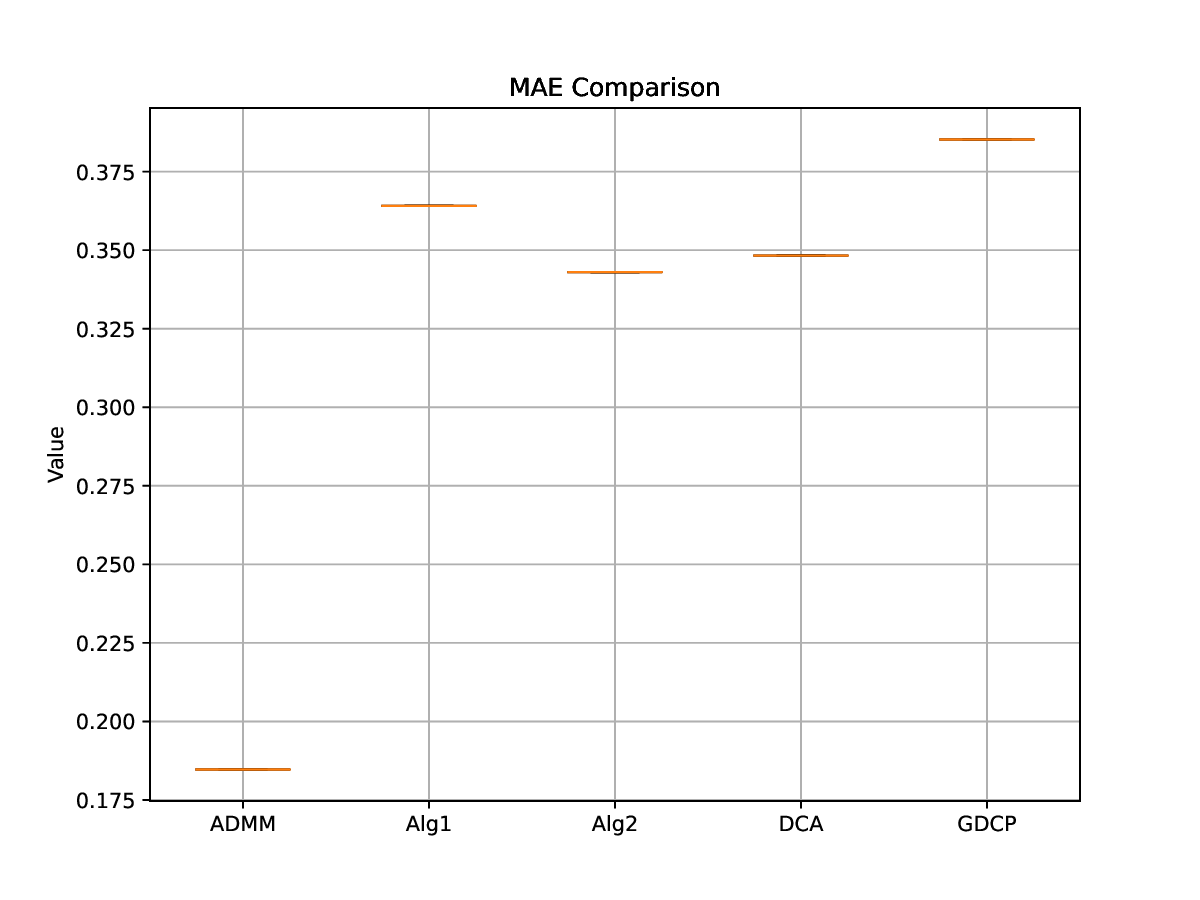}
		\includegraphics[width=.3\linewidth]{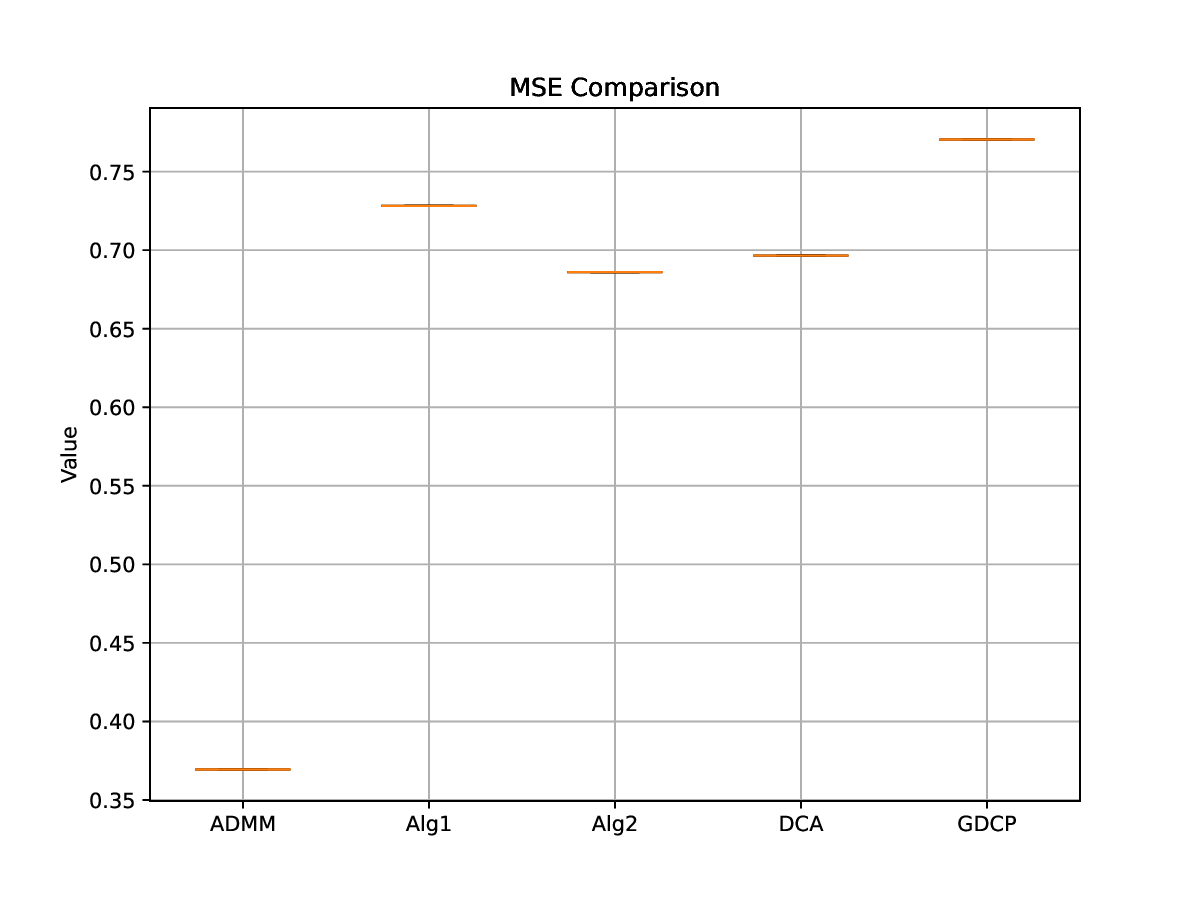}
		\includegraphics[width=.3\linewidth]{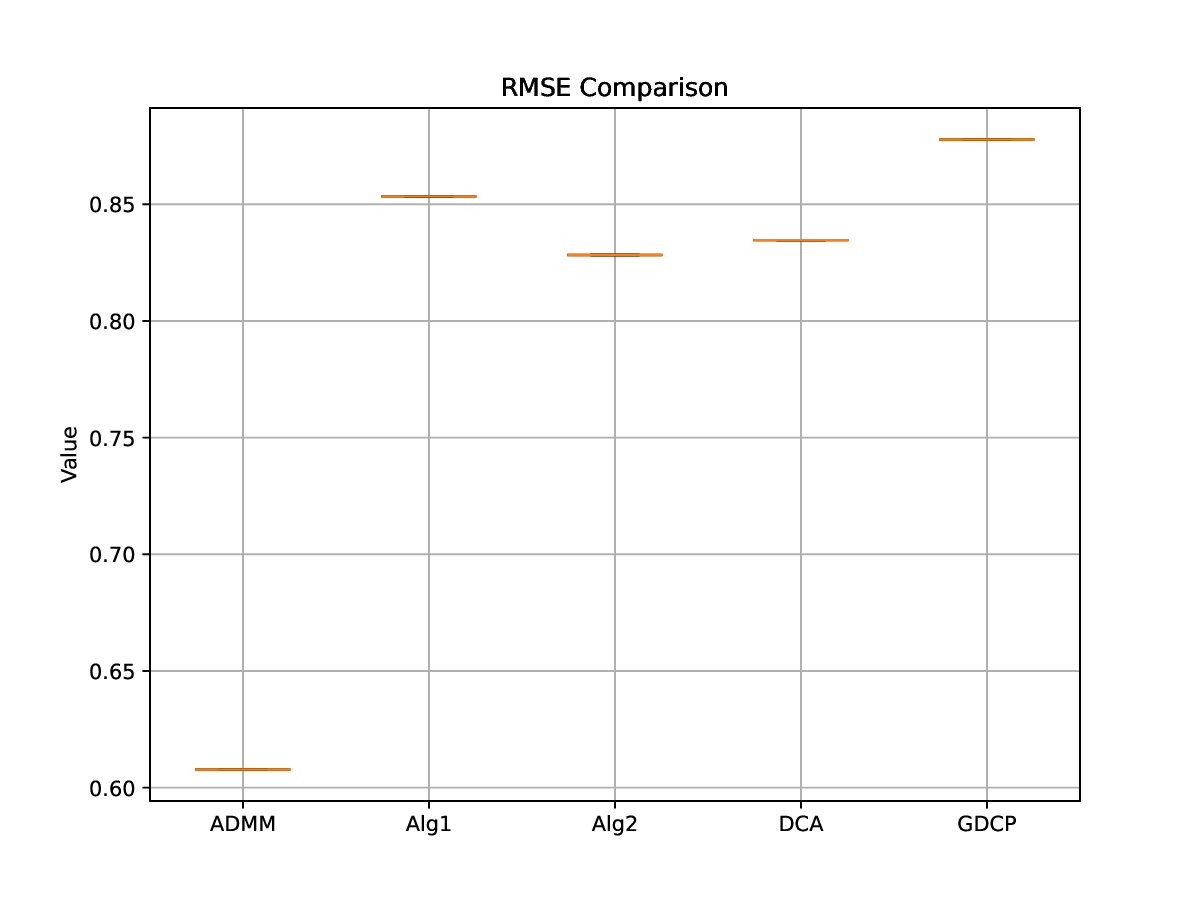}
	\end{subfigure}
	\caption{Comparison of MAE, MSE and RMSE for 40\% split of dataset2} \label{fig3d4}
\end{figure}

\subsubsection*{Discussion of results}
The results of the experiments for dataset1 reveal that Algorithm \ref{alg2} and \ref{alg1} consistently demonstrate better convergence rates compared to the other algorithms across all four splits of the dataset. While the ADMM algorithm achieves the better accuracy and precision scores, it suffers from longer execution times compared to the others. Specifically, Algorithm \ref{alg2} and \ref{alg1} exhibit better accuracy and precision scores than both the DCA and GDCP methods. Notably, the GDCP method displays the worst performance with lowest accuracy and highest MAE, MSE, and RMSE scores among all methods evaluated. These findings underscore the effectiveness of the proposed Algorithm \ref{alg2} and \ref{alg1} in achieving rapid convergence and competitive performance in accuracy and precision metrics when compared to other algorithms in the experiment. However, the ADMM method, while boasting superior accuracy and precision, lags in execution time efficiency. This behaviour is supported by findings in the literature which had shown the ADMM having poor convergence rate.

For dataset2, the proposed Algorithm \ref{alg2} and \ref{alg1} also showed to have the best rate of convergence than other methods in the experiments.
However, a nuanced observation reveals that for the (10:90) and (40:60) splits of the dataset, the accuracy and precision scores of the ADMM and DCA methods surpass those of the proposed algorithms. Despite achieving commendable accuracy and precision scores, the ADMM method continues to exhibit the poorest execution time among the evaluated methods. Furthermore, consistent with the findings in the literature, the DCA algorithm showcases promising performance, particularly in feature selection tasks. Also, the GDCP method consistently underperforms across all splits of the dataset, displaying the lowest accuracy and highest error metrics among all methods considered. These observations underscore the importance of considering various performance metrics and dataset characteristics when evaluating algorithm efficacy. While Algorithm \ref{alg2} and \ref{alg1} demonstrate strong convergence rates, the ADMM and DCA methods exhibit competitive accuracy and precision scores, but at the expense of longer execution times. Additionally, this findings align with prior research indicating the efficiency of the DCA algorithm in feature selection tasks. In conclusion, we noted that all five methods consistently yield high values of MSE, and RMSE scores. This outcomes may arise from the algorithms reaching the maximum iteration limit of 2000, where the resulting point computed is not a critical point of the optimization model.

\section{Final Remarks}\label{Sec:Final}
We have proved in this paper that new two variants of the unified Douglas-Rachford splitting algorithm converge weakly to a
critical point of a generalized DC programming in Hilbert spaces. We recovered the unified Douglas-Rachford splitting algorithm from our proposed
algorithms. Numerical illustrations from machine learning showed that our algorithms have practical implementations, are efficient, and outperform some other related algorithms for solving this class of generalized DC programming. As part of the future project, we study the accelerated versions of our proposed algorithms.

\end{document}